\documentclass[12pt]{article}

\usepackage{amsmath,amssymb,amsfonts,amsthm}
\usepackage{graphicx}
\usepackage{color}

\newtheorem{corollary}{Corollary}
\newtheorem{theorem}[corollary]{Theorem}
\newtheorem{lemma}[corollary]{Lemma}

\newtheorem{proposition}[corollary]{Proposition}

\newcommand{\R}{\mathbb{R}}
\newcommand{\C}{\mathbb{C}}

\newcommand{\Z}{\mathbb{Z}}
\newcommand{\N}{\mathbb{N}}
\newcommand{\HH}{\mathbb{H}}
\newcommand{\E}{\mathbb{E}}
\newcommand{\U}{\mathbb{U}}
\newcommand{\A}{\mathcal{A}}
\newcommand{\Gamm}{\mathcal C}

\def\diam{\mathop{\mathrm{diam}}}

\def\Im{{\rm Im}\,}
\def\Re{{\rm Re}\,}
\def\SLEkk#1/{$\mathrm{SLE}_{#1}$}
\def\SLEk/{\SLEkk{\kappa}/}
\def\SLEtwo/{\SLEkk2/}
\def\SLE/{$\mathrm{SLE}$}
\def\CLEkk#1/{$\mathrm{CLE}_{#1}$}
\def\CLEk/{\CLEkk{\kappa}/}
\def\CLEtwo/{\CLEkk2/}
\def\CLE/{$\mathrm{CLE}$}
\def\GLEkk#1/{$\mathrm{GLE}_{#1}$}
\def\GLEk/{\GLEkk{\kappa}/}
\def\GLEtwo/{\GLEkk2/}
\def\GLE/{$\mathrm{GLE}$}
\def\SLEkr/{$\mathrm{SLE}_{\kappa, \rho}$}

\def\Ito/{It\^o}
\def \eps {\varepsilon}
\def \P {{\bf P}}

\def \E {{\bf E}}

\def\hcap{\mathrm{hcap}}

\def\diam{\mathrm{diam}}
\def\H{\mathbb{H}}
\def\U{\mathbb{U}}

\def\BESd/{$\mathrm{BES}^\delta$}
\def\BESdx/{$\mathrm{BES}^\delta_x$}
\def\SBESd/{$\mathrm{SBES}^\delta$}
\def\BESQd/{$\mathrm{BESQ}^\delta$}

\def \interior {{\hbox {int}}}
\def \eps {\varepsilon}

\def \gammabis {{\ell}}
\def \Gammabis {{\mathcal L}}
\numberwithin{corollary}{section}

\begin{document}

\title{Conformal loop ensembles:\\ the Markovian characterization and the loop-soup construction}
\author{{\sc Scott Sheffield\thanks{Massachusetts Institute of Technology. Partially supported by NSF grants DMS 0403182 and DMS 0645585 and
OISE 0730136.}} \,  and
{\sc Wendelin Werner\thanks{Universit\'e Paris-Sud 11  and Ecole Normale Sup\'erieure. Research supported in part by ANR-06-BLAN-00058}
}
}
\date {}

\maketitle
\begin {abstract}
For random collections of self-avoiding loops in two-dimensional domains, we define a simple and natural conformal restriction property that is conjecturally satisfied by the
scaling limits of interfaces in models from statistical physics. This property is basically the combination of conformal invariance and the locality of the interaction in the model. Unlike the Markov property that Schramm used to characterize SLE curves (which involves conditioning on partially generated interfaces up to arbitrary stopping times), this property only involves conditioning on entire loops and thus appears at first glance to be weaker.

Our first main result is that there exists exactly a one-dimensional family of random loop collections with this property---one for each $\kappa \in (8/3,4]$---and that the loops are forms of SLE$_\kappa$. The proof proceeds in two steps. First, uniqueness is established by showing that every such loop ensemble can be generated by an ``exploration'' process based on SLE.

Second, existence is obtained using the two-dimensional Brownian loop-soup, which is a Poissonian random collection of loops in a planar domain.  When the intensity parameter $c$ of the loop-soup is less than $1$, we show that the outer boundaries of the loop clusters are disjoint simple loops (when $c>1$ there is a.s.\ only one cluster) that satisfy the conformal restriction axioms.  We prove various results about loop-soups, cluster sizes, and the $c=1$ phase transition.

Taken together, our results imply that the following families are equivalent:
\begin{enumerate}
\item The random loop ensembles traced by branching Schramm-Loewner Evolution (SLE$_\kappa$) curves for $\kappa$ in $(8/3, 4]$.
\item The outer-cluster-boundary ensembles of Brownian loop-soups for $c \in (0, 1]$.
\item The (only) random loop ensembles satisfying the {\em conformal restriction axioms}.
\end{enumerate}
\end{abstract}

\newpage

\tableofcontents

\newpage

\section {Introduction}

\subsection {General introduction}

{\bf SLE and its conformal Markov property.}
Oded Schramm's SLE processes introduced in \cite {Sch} have deeply changed the way mathematicians and physicists understand critical phenomena in two dimensions.
Recall that a chordal SLE is a random non-self-traversing curve in a simply connected domain, joining two prescribed boundary points of the domain.
Modulo conformal invariance hypotheses that have been proved to hold in several cases, the scaling limit of an interface that appears in various two-dimensional models from statistical physics, when boundary conditions are chosen in a particular way, is one of these SLE curves.  For instance, in the Ising model on a triangular lattice, if one connected arc $d_+$ of the boundary of a simply connected region $D$ is forced to contain only $+$ spins whereas the complementary arc $d_-$ contains only $-$ spins, then there is a random interface that separates the cluster of $+$ spins attached to $d_+$ from the cluster of $-$ spins attached to $d_-$; this random curve has recently been proved by Chelkak and Smirnov to converge in distribution to an SLE curve (SLE$_3$) when one lets the mesh of the lattice go to zero (and chooses the critical temperature of the Ising model) \cite {Sm2,ChSm}.

Note that SLE describes the law of one particular interface, not the joint law of all interfaces (we will come back to this issue later). On the other hand, for a given model, one expects all macroscopic interfaces to have similar geometric properties, i.e., to locally look like an SLE.

\medbreak

\begin {figure}[htbp]
\begin {center}
\includegraphics [width=2.2in]{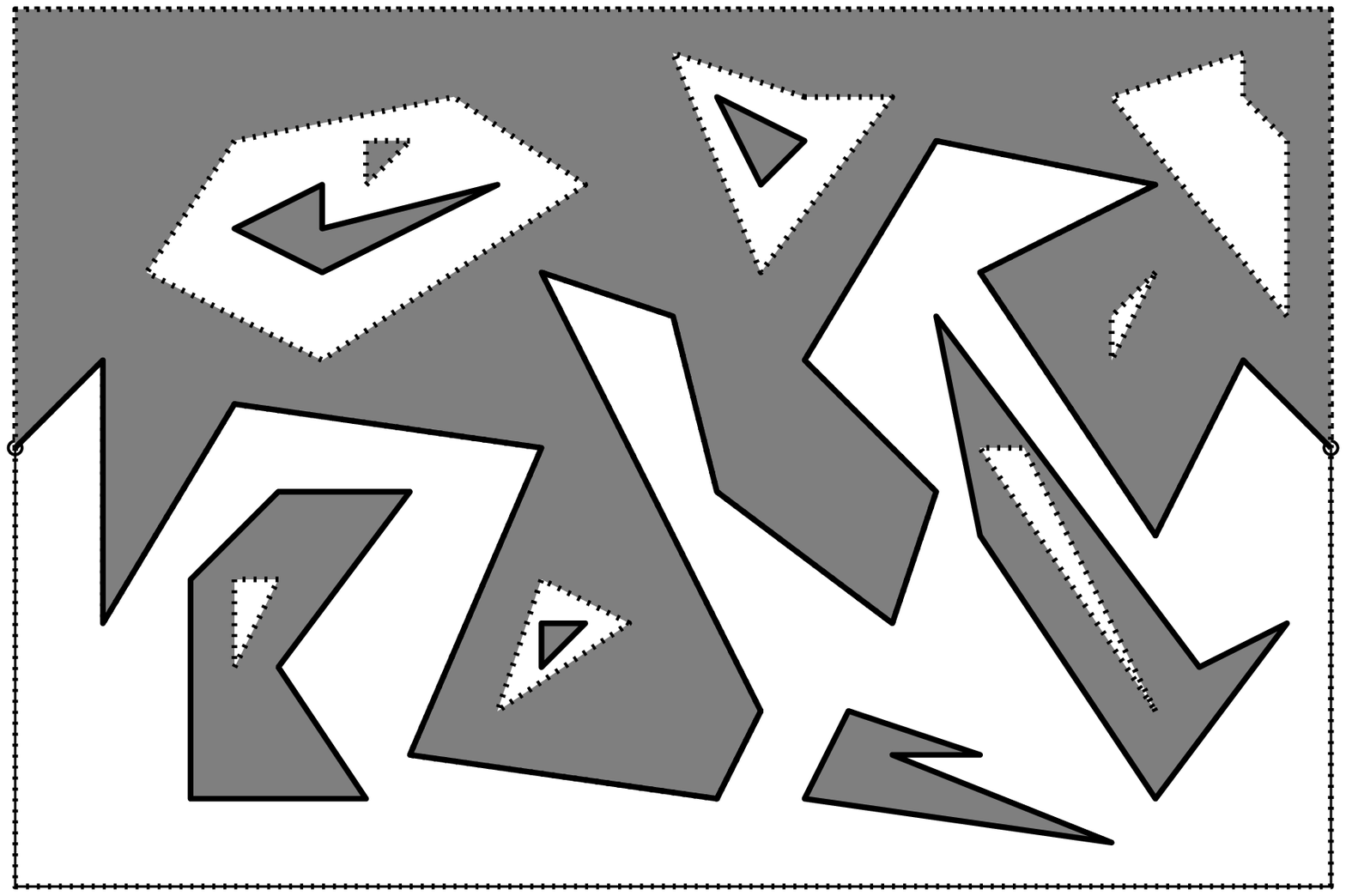}
\includegraphics [width=2.2in]{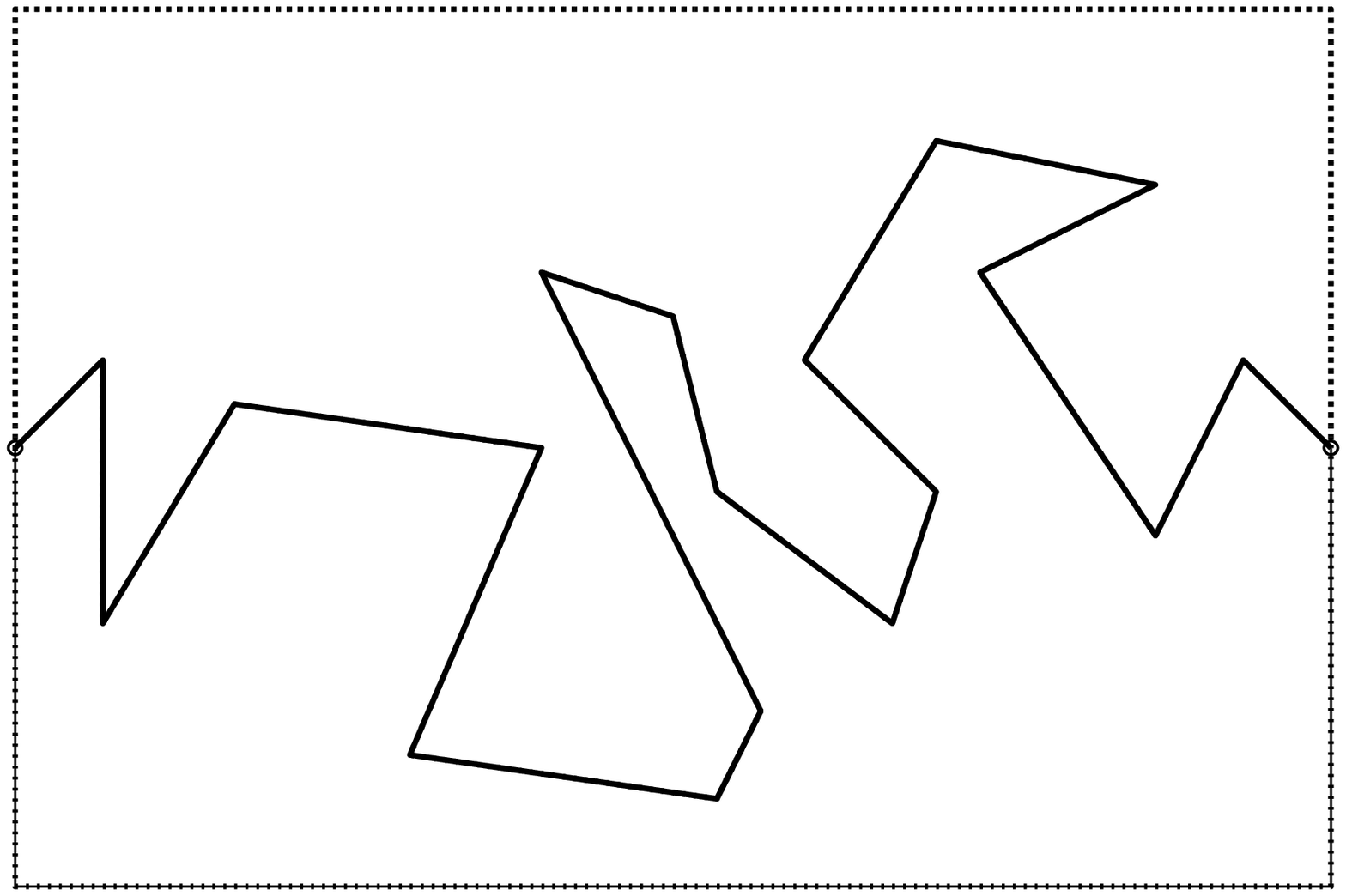}
\caption {A coloring with good boundary conditions {(black on one boundary arc, white on the complementary boundary arc)} and the chordal interface (sketch)}
\end {center}
\end {figure}

The construction of SLE curves can be summarized as follows:
The first observation, contained in Schramm's original paper \cite {Sch}, is the ``analysis'' of the problem: Assuming that the two-dimensional models of statistical physics have a conformally invariant scaling limit, what can be said about the scaling limit of the interfaces? If one chooses the boundary conditions in a suitable way, one can identify a special interface that joins two boundary points (as in the Ising model mentioned above).
Schramm argues that if this curve has a scaling limit, and if its law is conformally invariant, then it should satisfy an ``exploration'' property in the scaling limit.  This property, combined with conformal invariance, implies that it can be defined via iterations of independent random conformal maps. With the help of Loewner's theory for slit mappings, this leads naturally to the definition of the (one parameter) family of SLE processes, which are
 random increasing families of compact sets (called Loewner chains), see \cite {Sch} for more details.  Recall that Loewner chains are constructed via continuous iterations of infinitesimal conformal perturbations of the identity, and they do not a priori necessarily correspond to actual planar curves.

A second step, essentially completed in \cite{RS}, is to start from the definition of these SLE processes as random Loewner chains, and to prove that
they indeed correspond to random two-dimensional curves.
This constructs a one-parameter family of SLE random curves joining two boundary points of a domain, and the previous steps shows that if a random curve is conformally invariant (in distribution) and satisfies
 the exploration property, then it is necessarily one of these SLE curves.

One can study various properties of these random Loewner chains.  For instance, one can compute critical exponents such as in \cite {LSW1,LSW2}, determine their fractal dimension as in \cite {RS,Be}, derive special properties of certain SLE's -- locality, restriction -- as in \cite {LSW1,LSWr}, relate them to discrete lattice models such as uniform spanning trees, percolation, the discrete Gaussian Free Field or the Ising model as in \cite {LSWlesl,Sm1,Sm2,CN, SchSh}, or to the Gaussian Free Field and its variants as in \cite {SchSh, Dub, M1, M} etc.  Indeed, at this point the literature is far too large for us to properly survey here. For conditions that ensure that discrete interfaces converge to SLE paths, see the recent contributions \cite{K, Sh}.
\medbreak

{\bf Conformal Markov property for collections of loops.}
A natural question is how to describe the ``entire'' scaling limit of the lattice-based model, and not only that of one particular interface.
In the present paper, we will answer the following question: Supposing that a discrete random system gives rise in its scaling limit to a conformally invariant collection of loops (i.e., interfaces) that remain disjoint (note that this is not always the case; we will comment on this later), what can these random conformally invariant collections of loops be?

\medbreak

\begin {figure}[htbp]
\begin {center}
\includegraphics [width=4.5in,angle=180]{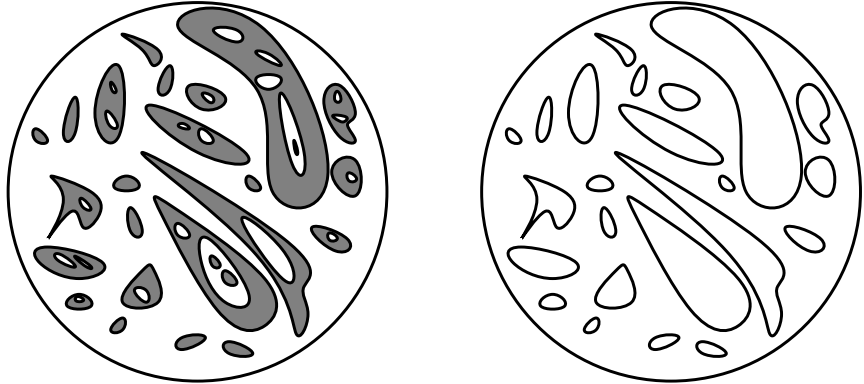}
\caption {A coloring and the corresponding outermost loops (sketch)}
\end {center}
\end {figure}
\medbreak

More precisely, we will define and study random collections of loops that combine conformal invariance and a natural restriction property (motivated by the fact that
the discrete analog of this  property trivially holds for the discrete models we have in mind). We call such collections of loops Conformal Loop Ensembles (CLE).
The two main results of the present paper can be summarized as follows:
\begin {theorem}\label {mainresults}
\begin {itemize}
\item
For each CLE, there exists a value $\kappa \in (8/3, 4]$ such that with probability one, all loops of the CLE are SLE$_\kappa$-type loops.
\item
Conversely, for each $\kappa \in (8/3, 4]$, there exists exactly one CLE with SLE$_\kappa$ type loops.
\end {itemize}
\end {theorem}

In fact, these statements will be derived via two almost independent steps, that involve different techniques:

\begin {enumerate}
\item We first derive the first of these two statements together with the uniqueness part of the second one. This will involve a detailed analysis of the CLE property, and consequences about possible ways to ``explore'' a CLE. Here, SLE techniques will be important.
\item We derive the existence part of the second statement using clusters of Poisson point processes of Brownian loops (the Brownian loop-soups).
\end {enumerate}

In the end, we will have two remarkably different explicit constructions of these conformal loop ensembles CLE$_\kappa$ for each $\kappa$ in $(8/3,4]$ (one based on SLE, one based on loop-soups).
This is useful, since many properties that seem very mysterious from one perspective are easy from the other.  For example, the (expectation) fractal dimensions of the individual loops and of the set of points not surrounded by any loop can be explicitly computed with SLE tools \cite{SSW}, while many monotonicity results and FKG-type correlation inequalities are immediate from the loop-soup construction \cite{Wls}. One illustration of the interplay between these two approaches is already present in this paper: One can use SLE tools to determine exactly the value of the critical intensity that separates the two percolative phases of the Brownian loop-soup (and to our knowledge, this is the only self-similar percolation model where this critical value has been determined).

In order to try to explain the logical construction of the proof, let us outline these two parts separately in the following two subsections.

\subsection {Main statements and outline: The Markovian construction}

Let us now describe in more detail the results of the first part of the paper, corresponding to Sections \ref {S.2} through \ref {S.8bis}.
We are going to study random families $\Gamma = ( \gamma_j, j \in J)$ of non-nested simple disjoint loops in simply connected domains.
For each simply connected $D$, we let $P_D$ denote the law of this loop-ensemble in $D$. We say that this family is {\bf conformally invariant} if for any two simply connected domains $D$ and $D'$ (that are not equal to the entire plane) and conformal transformation $\psi: D \to D'$,  the image of $P_D$ under $\psi$ is $P_{D'}$.

We also make the following ``local finiteness'' assumption: if $D$ is equal to the unit disc $\U$, then for any $\eps >0$, there are $P_{\U}$ almost surely only finitely many loops of radius larger than
$\epsilon$ in $\Gamma$.

Consider two simply connected domains $D_1 \subset D_2$, and sample a family
$(\gamma_j, j \in J)$ according to the law $P_{D_2}$ in the larger domain $D_2$.
Then, we can subdivide the family $\Gamma= (\gamma_j, j \in J)$ into two parts: Those that do not stay in $D_1$, and those that stay in $D_1$ (we call the latter $(\gamma_j, j \in J_1)$). Let us now define $D_1^*$ to be the random set obtained when removing from the set $D_1$ all the loops of $\Gamma$ that do not fully stay in $D_1$, together with their interiors.
We say that the family $P_D$ {\bf satisfies restriction} if, for any such $D_1$ and $D_2$, the conditional law of $(\gamma_j, j \in J_1)$ given $D_1^*$ is
$P_{D_1^*}$ (or more precisely, it is the product of $P_D$ for each connected component $D$ of $D_1^*$).
When a family is conformally invariant and satisfies restriction, we say that it is a Conformal Loop Ensemble (CLE). The goal of the paper is
to characterize and construct all possible CLEs.

\begin {figure}[htbp]
\begin {center}
\includegraphics [width=1.8in]{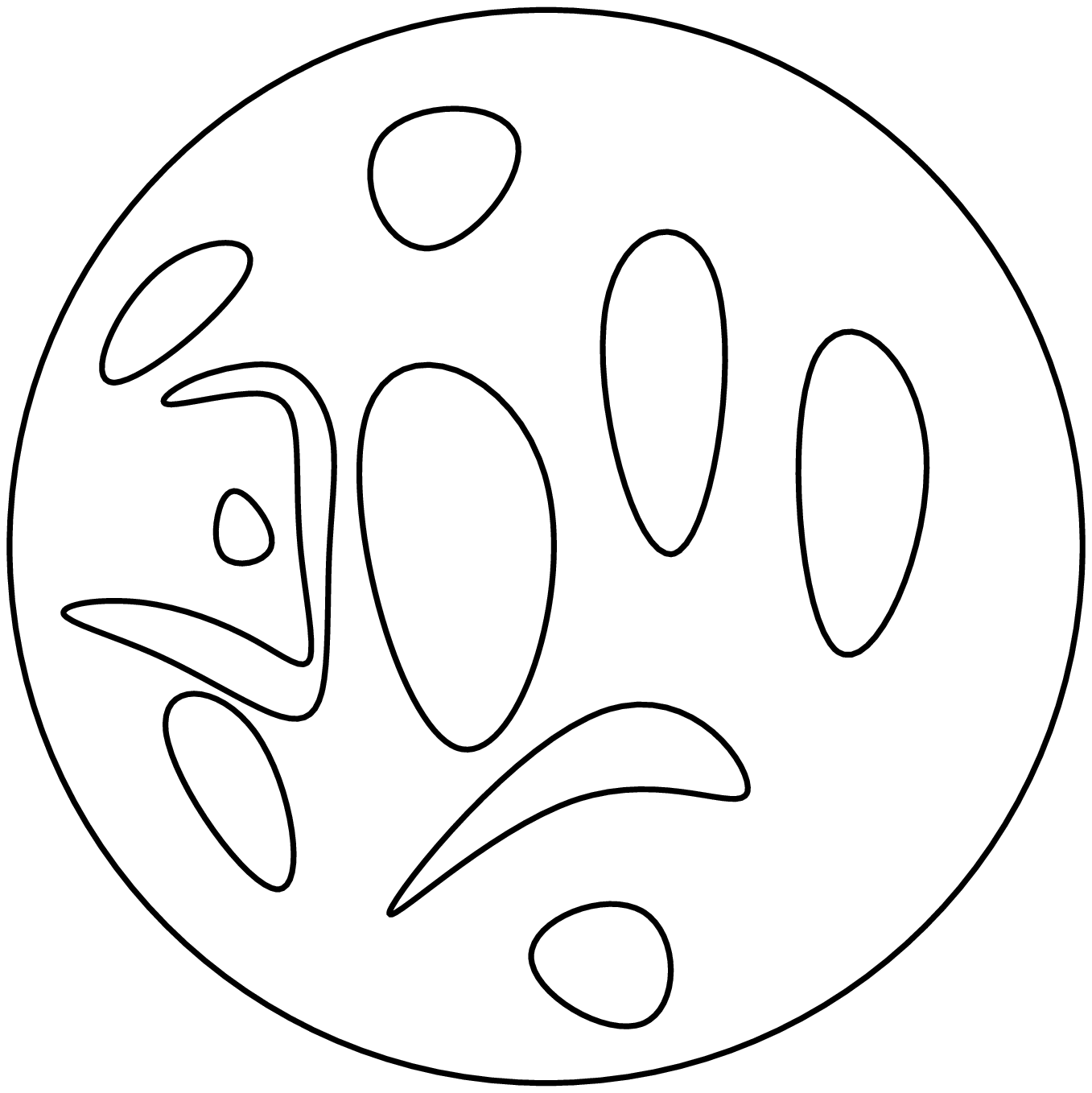}
\includegraphics [width=1.8in]{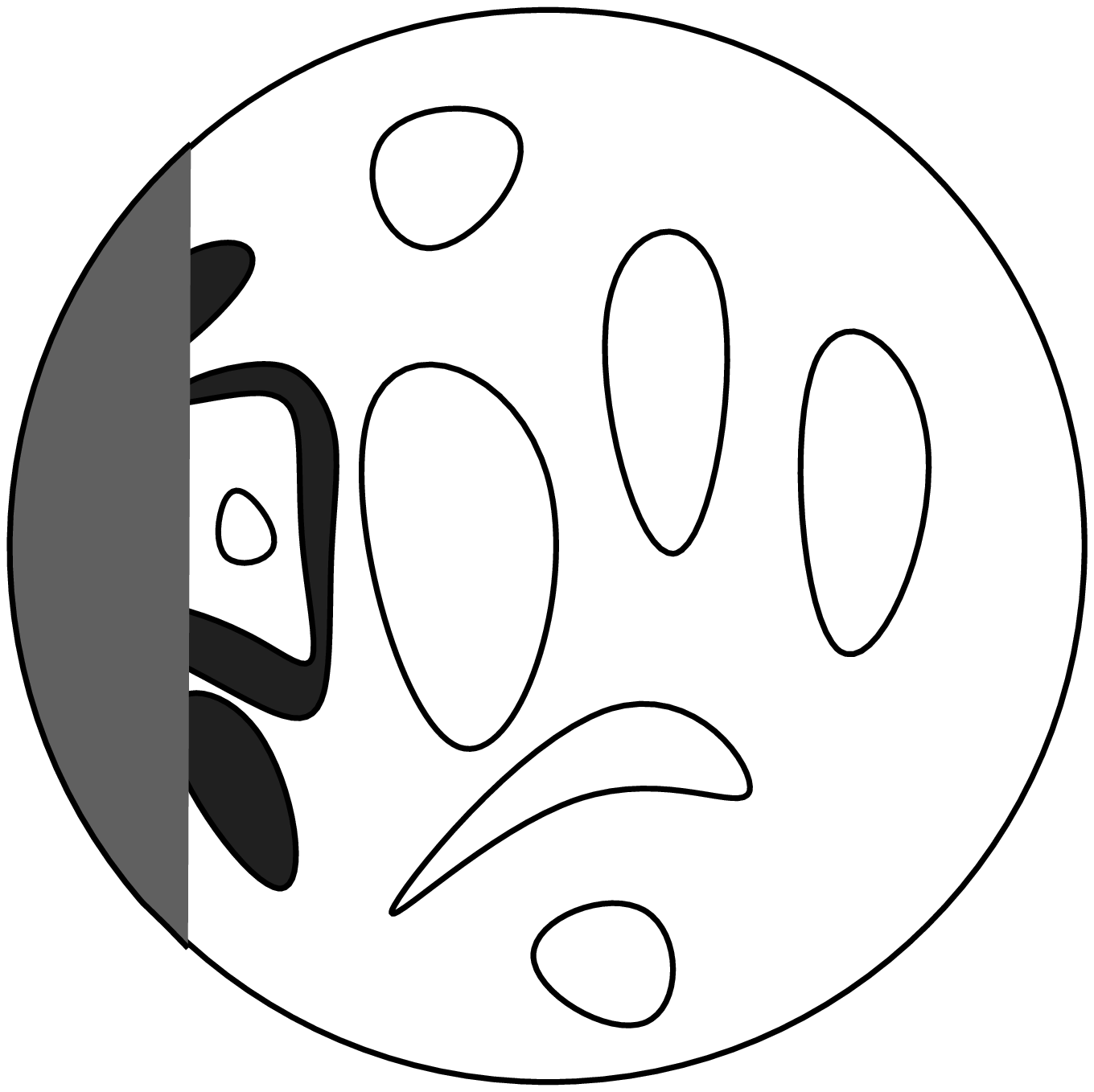}
\caption {Restriction property (sketch): given the set of loops intersecting $D_2 \setminus D_1$ (the grey wedge on the left of the right figure)
the conditional law of the remaining loops is an independent CLE in each component of the (interior of the) complement of this set.}
\end {center}
\end {figure}

By conformal invariance, it is sufficient to describe $P_D$ for one given simply connected domain. Let us for instance consider $D$ to be the upper half-plane $\HH$.
A first step in our analysis will be to prove that for all $z \in \HH$, if $\Gamma$ is a CLE, then the conditional law of the unique loop $\gamma(z) \in \Gamma$ that surrounds
$z$, conditionally on the fact that $\gamma (z)$ intersects the $\eps$-neighborhood of the origin, converges as $\eps \to 0$ to a probability measure $P^z$ on ``pinned loops'', i.e., loops in $\HH$ that touch the real line only at the origin. We will derive various properties of $P^z$, which will eventually enable us to relate it to SLE. One simple way to describe this relation is as follows:

\begin {theorem}
\label {t1}
If $\Gamma$ is a CLE, then the measure $P^z$ exists for all $z \in \HH$, and it is equal to the limit when $\eps \to 0$ of the law of an SLE$_\kappa$ from $\eps$ to $0$ in $\HH$ conditioned to disconnect $z$ from infinity in $\HH$, for some $\kappa \in (8/3, 4]$ (we call this limit the SLE$_\kappa$ {\bf bubble measure}).
\end {theorem}

This shows that all the loops of a CLE are indeed in some sense ``SLE$_\kappa$ loops'' for some $\kappa$. In fact, the way in which  $P^z$ will be described (and in which this theorem will actually be proved) can be understood as follows (this will be the content of Proposition \ref {p44}) in the case where $z=i$: Consider $A$ the lowest point on $[0, i ] \cap \gamma (i)$, and $H$ the unbounded connected component of the domain obtained by removing from $\HH \setminus [0,A]$ all the loops of the CLE that intersect $[0,A)$. Consider the conformal map $\Phi$ from $H$ onto $\HH$ with $\Phi(i)=i$ and $\Phi(A)=0$. Then, the law of $\Phi (\gamma(i))$ is exactly $P^i$.
\begin {figure}[htbp]
\begin {center}
\includegraphics [width=5in]{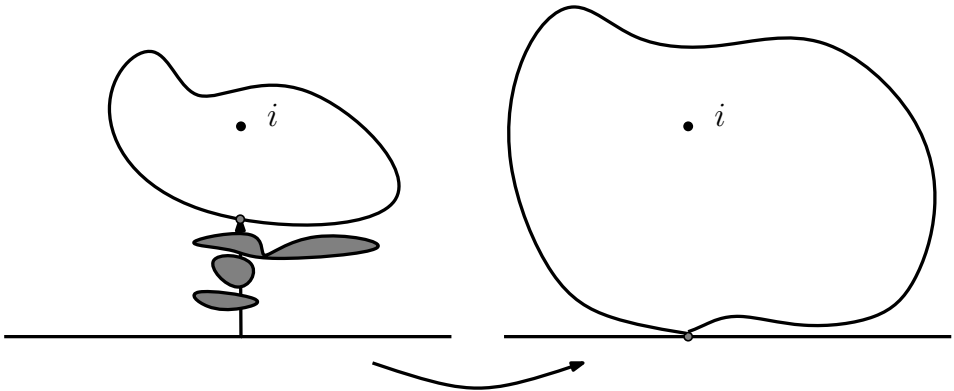}
\caption {Description of $P^i$ (sketch).}
\end {center}
\end {figure}

Theorem \ref {t1} raises the question of whether two different CLE distributions can correspond to the same measure $P^z$. We will prove that it is not possible, i.e., we will describe a way to reconstruct the law of the CLE out of the knowledge of $P^z$ only, using a construction based on a Poisson point process of pinned loops:

\begin {theorem}
\label {t2}
 For each $\kappa \in (8/3, 4]$, there exists at most one CLE such that $P^z$ is the SLE$_\kappa$ bubble measure.
\end {theorem}

In a way, this reconstruction procedure can be interpreted as an ``excursion theory'' for CLEs.
It will be very closely related to the decomposition of a Bessel process via its Poisson point process of excursions. In fact, this
 will enable us to relate our CLEs to the random loop ensembles defined in \cite {Sh}
using branching SLE processes, which we now briefly describe. Recall that when $\kappa \le 4$, $SLE_\kappa$ is a random simple curve from one marked boundary point $a$ of a simply connected domain $D$ to another boundary point $b$. If we now compare the law of an SLE from $a$ to $b$ in $D$ with that of an SLE from $a$ to $b'$ in $D$ when $b \not= b'$, then they clearly differ, and it is also immediate to check
that the laws of their initial parts (i.e., the laws of the paths up to the first time they exit some fixed small neighborhood of $a$) are also not identical. We say that SLE$_\kappa$ is not target-independent.
However, a variant of SLE($\kappa$) called SLE($\kappa$, $\kappa- 6$) has been shown by Schramm and Wilson \cite{SchW} (see also \cite {Sh}) to be target-independent.
This makes it possible to couple such processes starting at $a$ and aiming at two different points $b$ and $b'$ in such a way that they coincide until the first disconnection point. This in turn makes it possible to canonically define a conformally invariant ``exploration tree''
 of SLE ($\kappa$, $\kappa-6$) processes rooted at $a$, and a collection of loops
 called Conformal Loop Ensembles in \cite {Sh}.  It is conjectured in \cite {Sh} that this one-parameter  collection of loops indeed corresponds to the scaling limit of a wide class of discrete lattice-based models, and that for each $\kappa$, the law of the constructed family of loops is independent of the starting point $a$.

The  branching SLE ($\kappa, \kappa -6 $) procedure works for any $\kappa \in (8/3,8]$, but the obtained loops are
 simple and disjoint loops only when $\kappa \le 4$. In this paper, we use the term CLE to refer to any collection of loops satisfying conformal invariance and restriction, while using the term CLE$_\kappa$ to refer to the random collections of loops constructed in \cite{Sh}.  We shall prove the following:
\begin {theorem}
\label {t3}
Every CLE is in fact a CLE$_\kappa$ for some $\kappa \in (8/3,4]$.
\end {theorem}

Let us stress that we have not yet proved at this point that the CLE$_\kappa$ are themselves CLEs (and this was also not established in \cite {Sh}) -- nor that
the law of CLE$_\kappa$ is root-independent. In fact, it is not proved at this point that CLEs exist at all. All of this will follow from the second part.

\subsection {Main statements and outline: The loop-soup construction}

We now describe the content of Sections \ref {S.2bis} to \ref {S.4}.
The {Brownian loop-soup}, as defined in \cite{LW}, is a Poissonian random countable collection
of Brownian loops contained within a fixed simply-connected domain $D$.  We will actually only need to
consider the outer boundaries of the Brownian loops, so we will take the perspective that a loop-soup is
a random countable collection of simple loops (outer boudaries of Brownian loops
 can be defined as SLE$_{8/3}$ loops, see \cite {W}).
Let us stress that our {conformal loop ensembles}
are also random collections of simple loops, but that, unlike the loops of the Brownian loop-soup, the
loops in a CLE are almost surely all disjoint from one another.

The loops of the {\em Brownian loop-soup}  $\Gammabis = (\gammabis_j, j \in J)$ in the unit disk $\U$
 are the points of a Poisson point process with intensity
$c \mu$, where $c$ is an {\em intensity} constant, and $\mu$ is the {\em Brownian loop measure in $\U$}.
The Brownian loop-soup measure $\mathbb P = \mathbb P_c$
is the law of this random collection $\Gammabis$.

When $A$ and $A'$ are two closed bounded subsets of a bounded domain $D$, we denote by
$L(A,A'; D)$ the $\mu$-mass of the set of loops that intersect both sets $A$ and $A'$, and stay in $D$.
When the distance between $A$ and $A'$ is positive, this mass is finite \cite{LW}.
Similarly, for each fixed positive $\epsilon$, the set of loops that stay in the bounded domain $D$ and have diameter
larger than $\epsilon$, has finite mass for $\mu$.

The conformal restriction property of the Brownian loop measure $\mu$ (which in fact characterizes the measure up to a multiplicative constant; see \cite{W}) implies the following two facts (which are essentially the only features of the Brownian loop-soup that we shall use in the present paper):
\begin{enumerate}
\item Conformal invariance: The measure $\mathbb P_c$ is invariant under any Moebius transformation of the unit disc onto itself. This invariance makes it in fact possible to define the law
$\mathbb P_D$ of the loop-soup in any simply connected domain $D \not= \C$ as the law of the image of $\Gammabis$ under any given conformal map $\Phi$ from $\U$ onto $D$ (because the law of this image does not depend on the actual choice of $\Phi$).
\item Restriction: If one restricts a loop-soup in $\U$ to those loops that stay in
a simply connected domain $U \subset \U$, one gets a sample of $\mathbb P_U$.
\end{enumerate}
We will work with the usual definition (i.e., the usual normalization) of the measure $\mu$ (as in \cite {LW} --- note that there can be some confusion about a factor $2$ in the definition, related to whether one keeps track of the orientation of the Brownian loops or not). Since we will be talking about some explicit values of $c$ later, it is important to specify this normalization. For a direct definition of the measure $\mu$ in terms of Brownian loops, see \cite {LW}.

\vskip 5cm

\begin {figure}[htbp]
\hskip 4cm
\includegraphics [width=1.6in]{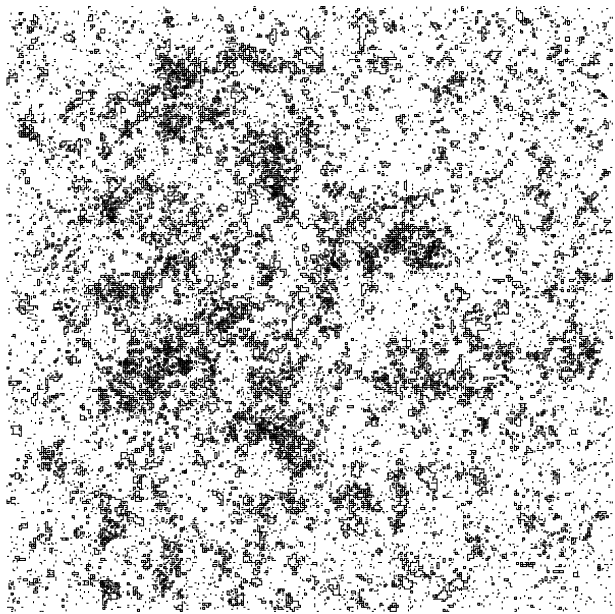}
\caption {Sample of a random-walk loop-soup approximation \cite {LTF} of a Brownian loop-soup in a square, by Serban Nacu}
\end {figure}

As mentioned above, \cite{Wls} pointed out a  way to relate Brownian loop-soups clusters to SLE-type loops:
Two loops in $\Gammabis$ are said to be adjacent if they intersect.  Denote
by $\mathcal C(\Gammabis)$ the set of clusters of loops under this relation.  For each element
$C \in \mathcal C(\Gammabis)$ write $\overline C$ for the closure of the union of all the loops in $C$ and denote
by $\overline \Gamm$ the family of all $\overline C$'s.

\begin {figure}[htbp]
\begin {center}
\includegraphics [width=2in]{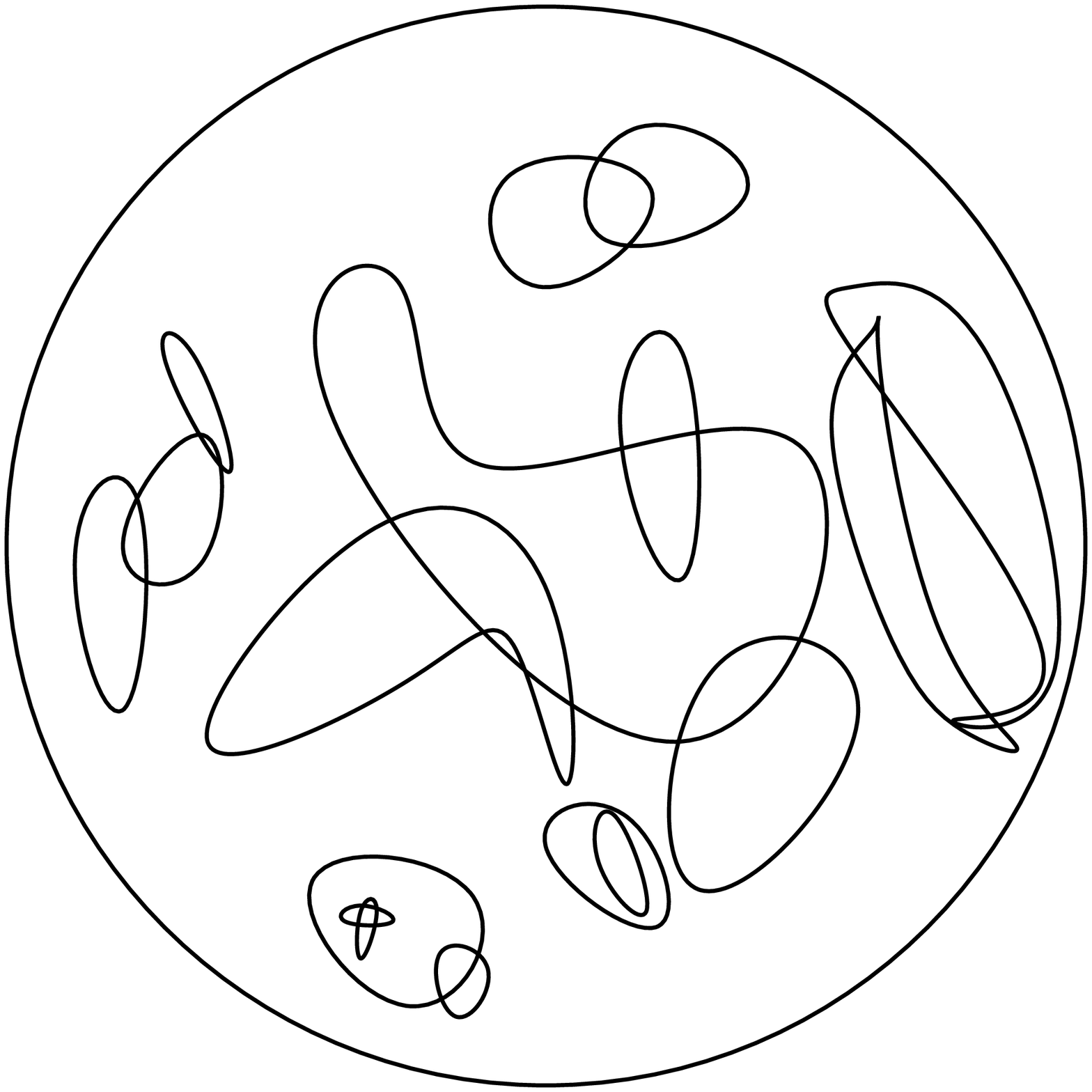}
\includegraphics [width=2in]{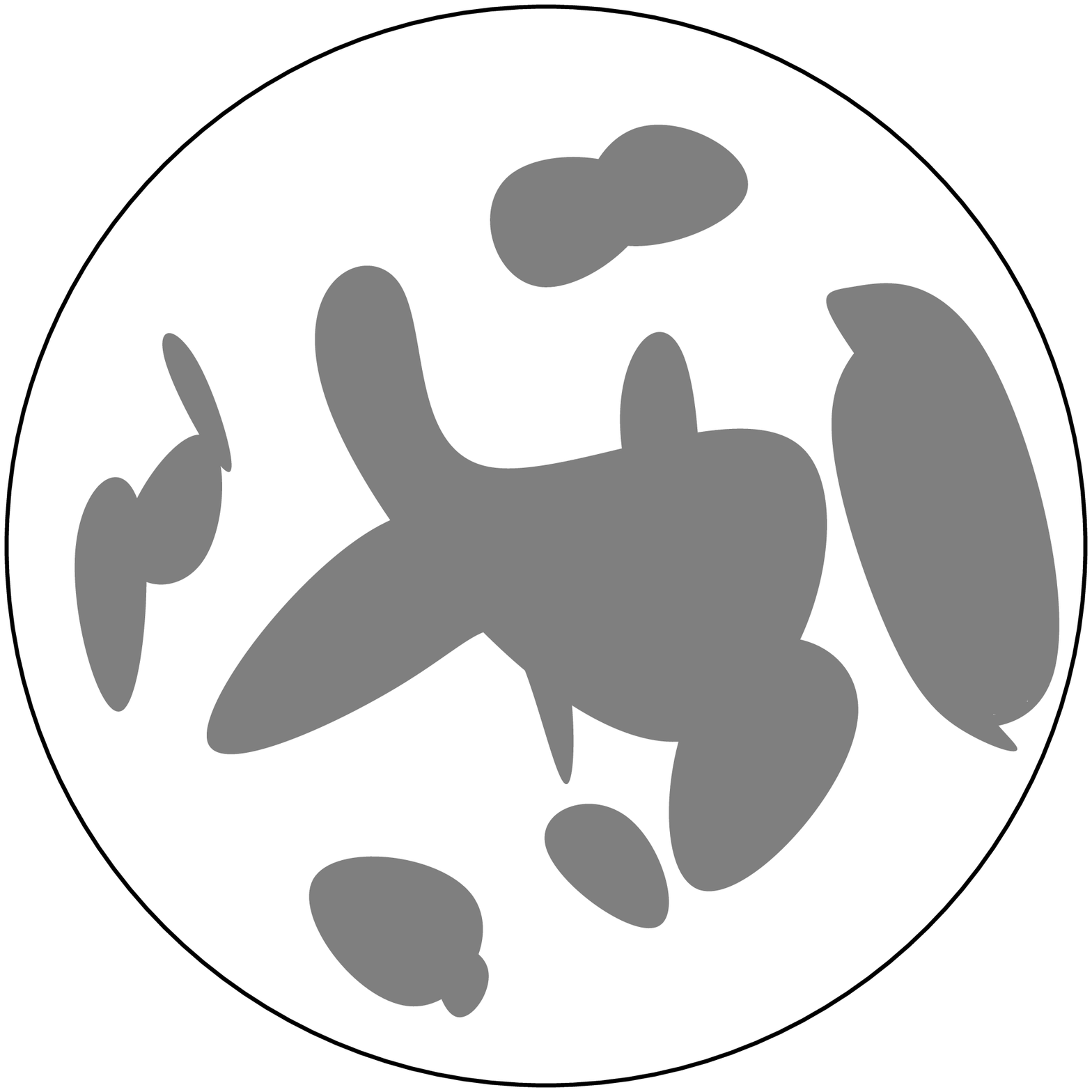}
\caption {A loup-soup and the fillings of its outermost clusters (sketch)}
\end {center}
\end {figure}

We write  $F(C)$ for the {\em filling} of $C$, i.e., for
the complement of the unbounded connected component of $\C \setminus \overline C$.
A cluster $C$ is called {\em outermost} if there exists no $C'$ such that $C \subset F(C')$.
The {\em outer boundary} of such an outermost cluster $C$ is the boundary of $F(C)$.
Denote by $\Gamma$ the set of outer boundaries of outermost clusters of $\Gammabis$.

\medbreak

Let us now state the main results of this second step:

\begin{theorem} \label{loopsoupsatisfiesaxioms}
Suppose that  $\Gammabis$ is the Brownian loop-soup with intensity $c$ in $\U$.
\begin {itemize}
 \item
If $c \in (0,1]$,
then $\Gamma$ is a random countable collection of disjoint simple loops that satisfies the conformal restriction axioms.
\item
If $c>1$, then there is almost surely only one cluster in $\mathcal C(\Gammabis)$.
\end {itemize}
\end{theorem}

It therefore follows from our Markovian characterization that $\Gamma$ is a CLE$_\kappa$ (according to the branching SLE($\kappa,
\kappa-6$) based definition in \cite{Sh}) for {\it some}
$\kappa \in (8/3,4]$.  We will in fact also derive the following correspondence:

\begin{theorem} \label{kappacorrespondence}
Fix $c \in (0,1]$ and let $\Gammabis$ be a Brownian loop-soup of intensity $c$ on $\U$.
Then $\Gamma$ is a CLE$_\kappa$ where $\kappa \in (8/3,4]$ is determined by the relation
$c=(3\kappa-8)(6-\kappa)/2\kappa$.
\end{theorem}

\subsection {Main statements and outline: Combining the two steps}

Since every $\kappa \in (8/3,4]$ is obtained for exactly one value of $c \in (0,1]$ in Theorem \ref {kappacorrespondence}, we immediately get thanks to Theorem \ref {t3} that
the random simple loop configurations satisfying the conformal restriction axioms
are precisely the CLE$_\kappa$ where $\kappa \in (8/3,4]$, which completes the proof of Theorem \ref {mainresults} and of the fact
that the following three descriptions of simple loop ensembles are equivalent:
\begin{enumerate}
\item The random loop ensembles traced by branching Schramm-Loewner Evolution (SLE$_\kappa$) curves for $\kappa$ in $(8/3, 4]$.
\item The outer-cluster-boundary ensembles of Brownian loop-soups for $c \leq 1$.
\item The (only) random loop ensembles satisfying the CLE axioms.
\end{enumerate}

Let us now list some further consequences of these results.
Recall from \cite {Be} that the Hausdorff dimension of an SLE$_\kappa$ curve is almost surely $1+ (\kappa /8)$. Our results therefore imply that the boundary of a loop-soup cluster of intensity $c \le 1$ has dimension
$$\frac { 37 -c - \sqrt { 25  + c^2 - 26 c }}{24} .$$
Note that just as for Mandelbrot's conjecture for the dimension of Brownian boundaries \cite {LSW4/3}, this statement does not involve SLE, but its proof does. In fact the
result about the dimension of Brownian boundaries can be viewed as the limit when $c \to 0$ of this one.

Furthermore we may define the {\em carpet}
of the CLE$_\kappa$ to be the random closed set obtained by removing from $\U$ the interiors (i.e. the bounded connected component of their complement) of all the loops $\gamma$ of $\Gamma$,
and recall that SLE methods allowed \cite {SSW} to compute its
``expectation dimension'' in terms of $\kappa$.
The present loop-soup construction of CLE$_\kappa$ enables to prove
(see \cite {NW}) that this expectation dimension is indeed equal to its almost sure Hausdorff dimension $d$, and that in terms of $c$,
\begin {equation}
 d (c) =  \frac {187 - 7c + \sqrt {25 + c^2 - 26 c }}{96}
\end {equation}
The critical loop-soup (for $c=1$) corresponds therefore to a carpet of dimension $15/8$.

Another direct consequence of the previous results is the ``additivity property'' of CLE's:
If one considers two independent CLE's in the same simply connected domain $D$ with non-empty boundary, and looks at the union of these two, then either one can find a cluster whose boundary contains $\partial D$, or the outer boundaries of the obtained outermost clusters in this union form another CLE. This is simply due to the fact that each of the CLE's can be constructed via Brownian loop soups (of some
intensities $c_1$ and $c_2$) so that the union corresponds to a Brownian loop-soup of intensity $c_1+c_2$. This gives for instance a clean direct geometric meaning to the general idea (present on various occasions in the physics literature) that relates in some way two independent copies of the Ising model to the Gaussian Free Field in their large scale limit:
The outermost boundaries defined by the union of two independent CLE$_3$'s in a domain (recall \cite {ChSm} that CLE$_3$ is the scaling limit of the Ising model loops, and note that it corresponds to $c=1/2$) form a CLE$_4$ (which corresponds to ``outermost'' level lines of the Gaussian Free Field, see \cite {SchSh2,Dub} and to $c=1 = 1/2 + 1/2$).

\subsection {Further background}

In order to put our results in perspective, we briefly recall some closely related work on conformally invariant structures.

\medbreak
{\bf Continuous conformally invariant structures giving rise to loops.}
There exist several natural ways to construct conformally invariant structures in a domain $D$. We have already mentioned the Brownian loop-soup that will turn out be instrumental in the present paper when constructing explicitely CLEs.
Another natural conformally invariant structure that we have also just mentioned is the Gaussian Free Field. This is a classical basic object in Field Theory. It has been shown (see \cite {SchSh, Dub}) that it is very closely related to SLE processes, and that one can detect all kinds of SLEs within the Gaussian Free Field. In particular, this indicates that CLEs (at least when $\kappa=4$) can in fact also be defined and found as ``geometric'' lines in a Gaussian Free Field.

\medbreak
{\bf Discrete models.}
A number of discrete lattice-based models have been conjectured to give rise to conformally invariant structures in the fine-mesh limit. For some of these models, these structures can be described by random collections of loops. We have already mentioned that Smirnov \cite {Sm1,Sm2,Sm3} has proved this conjecture for some important models (percolation, Ising model --- see also \cite {LSWlesl,SchShexpl,SchSh} for some other cases). Those models that will be directly relevant to the present paper (i.e., with disjoint simple loops) include the Ising model and the discrete  Gaussian Free Field level lines (\cite {Sm3, ChSm, SchSh, Dub}).
The scaling limits of percolation and of the FK-model related to the Ising model give rise to interfaces that are not disjoint. These are of course also very interesting objects (see \cite {SchSm,CN,Su} for the description of the percolation scaling limit), but they are not the subject of the present paper.
Conjecturally, each of the CLEs that we will be describing corresponds to the scaling limit of one of the
so-called $O(N)$ models, see e.g. \cite {N,KNln}, which are one simple way to define discrete random collections of non-overlapping loops.

In fact, if one starts from a lattice-based model for which one controls the (conformally invariant) scaling limit of an observable (loosely speaking, the scaling limit of the probability of some event), it seems possible (see Smirnov \cite {Sm2}) to use this to actually prove the convergence of the entire discrete ``branching'' exploration procedure to the corresponding branching SLE($\kappa, \kappa -6 $) exploration tree. It is likely that it is not so much harder
to derive the ``full'' scaling limit of all interfaces than to show the convergence of one particular interface to SLE.

Another quite different family of discrete models that might (more conjecturally) be related to the CLEs that we are studying here, are the ``planar maps'', where one chooses at random a planar graph, defined modulo homeomorphisms, and that are conjecturally closely related to the above (for instance via their conjectured relation with the Gaussian Free Field). It could well be that CLEs are rather directly related to random planar maps chosen in a way to contain ``large holes'', such as the ones that are studied in \cite {LGM}. In fact, CLEs, planar maps and the Gaussian Free Field should all be related to each other via Liouville quantum gravity, as described in \cite {DSh}.

\medbreak
{\bf Conformal Field Theory.}
Note finally that Conformal Field Theory, as developed in the theoretical physics community since the early eighties \cite {BPZ}, is also a setup to describe the scaling limits of all correlation functions of these critical two-dimensional lattice models. This indicates that a description of the entire scaling limit of the lattice models in geometric SLE-type terms could be useful in order to construct such fields.
CLE-based constructions of Conformal Field Theoretical objects ``in the bulk'' can be found in Benjamin Doyon's papers \cite {Do1,Do2}.
It may also be mentioned that aspects of the present paper (infinite measures on ``pinned configurations'') can be interpreted naturally in terms of insertions of boundary operators.

\vfill
\eject

\part*{Part one: the Markovian characterization}

\addcontentsline{toc}{part}{Part one: the Markovian characterization}

\section {The CLE property} \label{s.CLEproperty} \label {S.2}

\subsection {Definitions}

A {\bf simple loop} in the complex plane will be the image of the unit circle in the plane under a continuous injective map (in other words we will identify two loops if one of them is obtained by a bijective reparametrization of the other one; note that our loops are not oriented). Note that a simple loop $\gamma$ separates the plane into two connected components that we call its interior $\interior (\gamma)$ (the bounded one) and its exterior (the unbounded one) and that each one of these two sets characterizes the loop. There are various natural distances and $\sigma$-fields that one can use for the space ${\cal L}$ of loops. We will use the $\sigma$-field $\Sigma$ generated by all the events of the type
$ \{ O \subset \interior (\gamma) \}$ when $O$ spans the set of open sets in the unit plane. Note that this $\sigma$-field is also generated by the events of the type $\{ x \in \interior ( \gamma ) \}$ where $x$ spans a countable dense subset $Q$ of the plane (recall that we are considering simple loops so that $O \subset \interior ( \gamma)$ as soon as $O \cap Q \subset \interior (\gamma)$).

In the present paper, we will consider (at most countable) collections $\Gamma= (\gamma_j , j \in J)$ of simple loops. One way to properly define such a collection is to identify it with the point-measure $$\mu_\Gamma= \sum_{j \in J} \delta_{\gamma_j}.$$
 Note that this space of collections of loops is naturally equipped with the $\sigma$-field
 generated by the sets $\{ \Gamma \ : \ \# ( \Gamma \cap A) = k \} = \{ \Gamma \ : \ \mu_\Gamma (A) = k \}$, where $A \in \Sigma$ and $k \ge 0$.

We will say that $(\gamma_j , j \in J)$   is a {\bf simple loop configuration} in the bounded simply connected domain $D$ if
the following conditions hold:
\begin {itemize}
\item For each $j \in J$, the loop $\gamma_j$ is a simple loop in $D$.
\item For each $j \not= j' \in J$, the loops $\gamma_j$ and $\gamma_{j'}$ are disjoint.
\item For each $j \not= j' \in J$, $\gamma_j$ is not in the interior of $\gamma_{j'}$: The loops are not nested.
\item For each $\eps >0$, only finitely many loops $\gamma_j$ have a diameter greater than $\eps$.  We call this the {\bf local finiteness} condition.
\end {itemize}
All these conditions are clearly measurable with respect to the $\sigma$-field discussed above.
\medbreak

We are going to study random simple loop configurations with some special properties. More precisely, we will say that the random  simple loop configuration $\Gamma = ( \gamma_j , j \in J)$ in the unit disc $\U$ is a {\bf Conformal Loop Ensemble} (CLE) if it satisfies the following properties:
\begin {itemize}
\item Non-triviality: The probability that $J \not= \emptyset$ is positive.
\item Conformal invariance:  The law of $\Gamma$ is invariant under any conformal transformation from $\U$ onto itself. This invariance makes it in fact possible to define the law of the loop-ensemble $\Gamma_D$ in any simply connected domain $D \not= \C$ as the law of the image of $\Gamma$ under any given conformal map $\Phi$ from $\U$ onto $D$ (this is because the law of this image does not depend on the actual choice of $\Phi$). We can also define the law of a loop-ensemble in any open domain $D\not= \C$ that is the union of disjoint open simply connected sets by taking independent loop-ensembles in each of the connected components of $D$. We call this law $P_D$.
\item Restriction: To state this important property, we need to introduce some notation. Suppose that $U$ is a simply connected subset of the unit disc. Define
$$I=I(\Gamma, U)= \{j \in J \ : \ \gamma_j \not\subset U \}$$ and $ J^* = J^* (\Gamma, U)=
J \setminus I = \{j \in J \ : \ \gamma_j \subset U \}$. Define the (random) set
$$ U^* =  U^* ( U, \Gamma) = U \setminus \cup_{j \in I} \overline {\interior ( \gamma_j)} .
$$
This set $ U^*$ is a (not necessarily simply connected) open subset of $\U$ (because of the local finiteness condition). The restriction property is that (for all $U$), the conditional law of  $(\gamma_j, j \in J^*)$ given  $U^*$ (or alternatively given the family $(\gamma_j, j \in I)$) is $P_{U^*}$.
\end {itemize}

\medbreak

This definition is motivated by the fact that for many discrete loop-models that are conjectured to be conformally invariant in the scaling limit, the discrete analog of this restriction property holds. Examples include the $O(N)$ models (and in particular the critical Ising model interfaces). The goal of the paper is to classify all possible CLEs, and therefore the possible conformally invariant scaling limits of such loop-models.

The non-nesting property can seem surprising since the discrete models allow nested loops. The CLE in fact describes (when the domain $D$ is fixed) the conjectural scaling limit of the law of the ``outermost loops'' (those that are not surrounded by any other one). In the discrete models, one can discover them ``from the outside'' in such a way that the conditional law of the remaining loops given the outermost loops is just made of independent copies of the model in the interior of each of the discovered loops.
Hence, the conjectural scaling limit of the full family of loops is obtained by iteratively defining CLEs inside each loop.

At first sight, the restriction property does not look that restrictive. In particular, as it involves only interaction between entire loops, it may seem weaker than the conformal exploration property of SLE (or of branching SLE($\kappa, \kappa-6$)), that describes the way in which the path is progressively constructed.  However (and this is the content of Theorems \ref {t1} and \ref {t2}), the family of such CLEs is one-dimensional too, parameterized by $\kappa \in (8/3,4]$.

\subsection {Simple properties of CLEs}

We now list some simple consequences of the CLE definition. Suppose that $\Gamma= (\gamma_j, j \in J)$ is a CLE in $\U$.
\begin {enumerate}
\item
Then, for any given $z \in \U$, there almost surely exists a loop $\gamma_j$ in $\Gamma$ such that $z \in \interior (\gamma_j)$. Here is a short proof of this fact:
Define $u=u(z)$ to be the probability that $z$ is in the interior of some loop in $\Gamma$. By Moebius invariance, this quantity $u$ does not depend on $z$.
Furthermore, since $P (J \not= \emptyset) >0$, it follows that $u> 0$ (otherwise the expected area of the union of all interiors of loops would be zero). Hence, there exists $r \in (0,1)$ such that with a positive probability $p$, the origin is in the interior of some loop in $\Gamma$ that intersects the slit $[r,1]$ (we call $A$ this event). We now define $U = \U \setminus [r,1)$ and apply the restriction property. If $A$ holds, then the origin is in the interior of some loop of $\Gamma$. If $A$ does not hold, then the origin is in one of the connected components of $\tilde U$ and the conditional probability that it is surrounded by a loop
in this domain is therefore still $u$. Hence,
$ u = p + (1-p) u$ so that $u =1$.

\item
The previous observation implies immediately that $J$ is almost surely infinite. Indeed, almost surely, all the points $1-1/n, n \ge 1$ are surrounded by a loop, and any given loop can only surround finitely many of these points (because it is at positive distance from $\partial \U$).

\item Let $M( \theta)$ denote the set of configurations $\Gamma = ( \gamma_j, j \in J)$ such that for all $j \in J$, the radius $[0, e^{i \theta}]$ is never locally ``touched without crossing'' by $\gamma_j$  (in other words, $\theta$ is a local extremum of none of the $\arg (\gamma_j)$'s).
Then, for each given $\theta$,  $\Gamma$ is  almost surely in $M(\theta)$. Indeed, the argument of a given loop that does not pass through the origin can anyway at most have countably many ``local maxima'', and there are also countably many loops. Hence, the set of $\theta$'s such that $ \Gamma \notin M (\theta)$ is at most countable.
 But the law of the CLE is invariant under rotations, so that $P ( \Gamma \in M(\theta))$ does not depend on $\theta$.
Since its mean value (for $\theta \in [0, 2 \pi]$) is $1$, it is always equal to $1$.

If we now define, for all $r >0$,  the Moebius transformation of the unit disc such that $\psi (1)=1$, $\psi (-1) = -1$ and $\psi' (1) =r$, the invariance of the CLE law under $\psi$ shows that for each given $r$, almost surely, no loop of the CLE locally touches $\psi ( [-i,i])$ without crossing it.

\item For any $r <1$, the probability that $r \U$ is entirely contained in the interior of one single loop is positive: This is because each simple loop $\gamma$ that surrounds the origin can be approximated ``from the outside'' by a loop $\eta$ on a grid of rational meshsize with as much precision as one wants. This implies in particular that one can find one such loop $\eta$ in such a way that the image of one loop $\gamma$ in the CLE under a conformal map from $\interior  (\eta)$ onto $\U$ that preserves the origin has an interior containing $r \U$. Hence, if we apply the restriction property to $U = \interior  (\eta)$, we get readily that with positive probability, the interior of some loop in the CLE contains $r\U$. Since this property will not be directly used nor needed later in the paper, we leave the details of the proof to the reader.

\item The restriction property continues to hold if we replace the simply connected domain $U \subset \U$ with the union $U$ of countably many disjoint simply connected domains $U_i \subset \U$.   That is, we still have that the conditional law of $(\gamma_j, j \in J^*)$ given $U^*$ (or alternatively given the family $(\gamma_j, j  \in I)$) is $P_{U^*}$.  To see this, note first that applying the property separately for each $U_i$ gives us the marginal
    conditional laws for the set of loops within each of the $U_i$. Then, observe that
the conditional law of the set of loops in $U_i$ is unchanged when one further conditions on the set of loops in $\cup_{i' \not = i} U_{i'}$.
Hence, the sets of loops in the domains $U_1^*, \ldots, U_i^*, \ldots$ are in fact independent (conditionally on $(\gamma_j, j \in I)$).
\end {enumerate}

\section {Explorations}
\label {explo}

\subsection {Exploring CLEs -- heuristics}
\label {heuristics}
Suppose that $\Gamma= (\gamma_j, j \in J)$ is a CLE in the unit disc $\U$. Suppose that $\eps >0$ is given. Cut out from the disc a little given shape $S= S(\eps) \subset \U$ of radius $\eps$ around $1$.  If $y$ is a point on the unit circle, then
 we may write $yS$ for $y$ times the set $S$ --- i.e., $S$ rotated around the circle via multiplication by $y$. The precise shape of $S$ will not be so important; for concreteness, we may at this point think of $S$ as being equal to the $\eps$-neighborhood of $1$ in the unit disc. Let $U_1$ denote the connected component that contains the origin of the set obtained when removing from $U_1':= \U \setminus S$ all the loops that do not stay in $U_1'$. If the loop $\gamma_0$ in the CLE that surrounds the origin did not go out of $\U \setminus S$, then the (conditional) law of the CLE restricted to $U_1$ (given the knowledge of $U_1$) is again a CLE in this domain (this is just the CLE restriction property). We then define the conformal map $\Phi_1$ from $U_1$ onto $\U$ with $\Phi_1(0)=0$ and $\Phi_1' (0) >0$.

\begin {figure}[htbp]
\begin {center}
\includegraphics [width=5in]{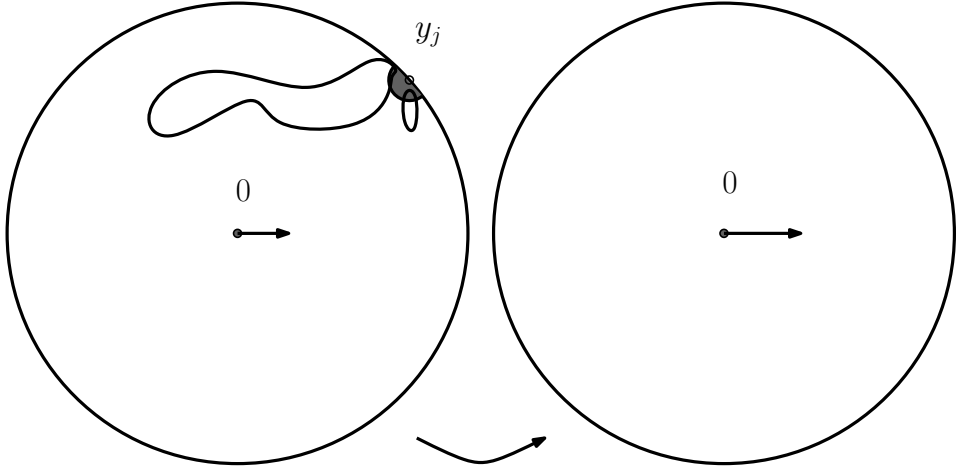}
\caption {The first exploration step (sketch)}
\end {center}
\end {figure}

Now we again explore a little piece of $U_1$: we choose some point $y_1$ on the unit circle and define $U_2'$ to be the domain obtained when removing from $U_1$ the preimage (under $\Phi_1$) of the shape $S$ centered around $y_1$ (i.e., $ U_2'=\Phi_1^{-1} ( \U \setminus y_1 S)$). Again, we define the connected component $U_2$ that contains the origin of the domain obtained when removing from $U_2'$ the loops that do not stay in $U_2'$ and the conformal map $\Phi_2$ from $U_2$ onto $\U$ normalized at the origin.

We then explore in $U_2$ if $\gamma_0 \subset U_2$, and so on.
One can iterate this procedure until we finally ``discover'' the loop $\gamma_0$ that surrounds the origin. Clearly, this will happen after finitely many steps with probability one, because at each step the derivative $\Phi_n' (0)$ is multiplied by a quantity that is bounded from below by a constant $v>1$ (this is because at each step, one composes $\Phi_n$ with a conformal map corresponding to the removal of at least a shape $y_n S$ in order to define $\Phi_{n+1}$). Hence, if we never discovered $\gamma_0$, it would follow from Koebe's $1/4$-Theorem that $d( \partial U_n, 0) \to 0$ as $n \to \infty$, and this would contradict the fact that $\gamma_0$ is almost surely at positive distance from $0$.

We call $N$ the random finite step after which the loop $\gamma_0$ is discovered i.e., such that $\gamma_0 \subset U_N$ but
$\gamma_0 \not \subset U_{N+1}'$.
It it important to notice that at each step until $N$, one is in fact repeating the same experiment (up to a conformal transformation), namely cutting out the shape $S$ from $\U$  and then cutting out all loops that intersect $S$.
Because of the CLE's conformal restriction property, this procedure defines an i.i.d.\ sequence of steps, stopped at the geometric random variable $N$, which is the first step
 at which one discovers a loop surrounding the origin that intersects $\U \setminus y_NS$.
 This shows also that the conditional law of the CLE in $\U$ given the fact that $\gamma_0 \cap S \not= \emptyset$ is in fact identical to the image under  $y_N^{-1} \Phi_{N}$ of the CLE in $U_N$.

In the coming sections, we  will use various sequences $y_n=y_n( \eps)$.
One natural possibility is to simply always choose $y_n =1$. This will give rise to the $\eps$ radial-explorations that will be discussed in Section \ref {S.5}.
However, we first need another procedure to choose $y_n (\eps)$ that will enable to control the behavior of $\Phi_{N(\eps)}$, of $U_{N(\eps)}$ and of $y_{N(\eps)}$ as $\eps$ tends to $0$. This will then allow us show that the conditional law of the CLE in $\U$ given the fact that $\gamma_0 \cap S(\eps) \not= \emptyset$ has a limit when $\eps \to 0$.

\subsection {Discovering the loops that intersect a given set}
\label {explodet}
 The precise shape of the sets $S$ that we will use will in fact not be really important, as long as they are close to small semi-discs.
It will be convenient to define, for each $y$ on the unit circle, the set $D(y, \eps)$ to be the image of the set $\{ z \in \HH \ : \ | z | \le \eps\} $, under the conformal map $\Psi : z \mapsto y (i-z) / (i+z)$ from the upper half-plane $\HH$ onto the unit disc such that $\Psi (i)= 0$ and $\Psi (0)= y$. Note that $|\Psi'(0)| =2$, so that when $\eps$ is very small, the set $D( y, \eps)$ is close to the intersection of a small disc of radius  $2\eps$ around $y$ with the unit disc.
This set $D(1, \eps)$ will play the role of our set $S(\eps)$.

Suppose that $\Gamma= (\gamma_j, j \in J)$ is a given (deterministic) simple loop-configuration in $\U$.  (In this section, we will derive deterministic statements that we will apply to CLEs in the next section.)
We suppose that:

\begin {enumerate}
 \item In $\Gamma$, one loop (that we call $\gamma_0$) has $0$ in its interior.
 \item $A \subset \overline \U$ is a given closed simply connected set such that $\U \setminus A$ is simply connected, $A$ is the closure of the interior of $A$, and the length of $\partial A \cap \partial \U$ is positive.
 \item The loop $\gamma_0$ does not intersect $A$.
 \item All $\gamma_j$'s in $\Gamma$ that intersect $\partial A$ also intersect the interior of $A$.
 \end {enumerate}

Our goal will be to explore almost all (when $\eps$ is small) large loops of $\Gamma$ that intersect $A$ by iterating explorations of $\eps$-discs.

\medbreak
When $\eps >0$ is given, it will be useful to have general critera that imply that
 a subset $V$ of the unit disc contains $D(y, \eps)$ for at least one $y \in \U$:
Consider two independent Brownian motions, $B^1$ and $B^2$ started from the origin, and stopped at their first hitting times $T_1$ and $T_2$ of the unit
circle. Consider $U^a$ and $U^b$ the two connected components of $\U \setminus (B^1 [0,T_1] \cup B^2 [0, T_2])$ that have an arc of $\partial \U$ on their boundary. Note that for small enough $\epsilon$, the probability $p(\eps)$ that both $U^a$ and $U^b$ contain some $D(y, \eps)$ is clearly close to $1$.

Suppose now that $V$ is a closed subset of $\U$ such that $\U \setminus V$ is simply connected, and let $u(V)$ be the probability that one of the two random sets $U^a$ or $U^b$ is a subset of $V$. Then:

\begin {lemma}
\label {UaUb}
For all $u>0$, there exists a positive $\eps_0 = \eps_0 (u)$ such that there exists $y \in \partial \U$ with $D(y, \eps) \subset V$ as soon as $u(V) \ge u$.
\end {lemma}

\begin {proof}
The definition of $p (\eps)$ and of $u(V)$ shows that $V$ contains some $D(y, \eps)$ with a probability at least
$u - (1-p(\eps))$. Since this is a deterministic fact about $V$, we conclude that the set $V$ does indeed contain some set $D(y, \eps)$ for some $y \in \partial \U$ as soon as
$p (\eps ) \ge 1 - u/2$.  It therefore suffices to choose $\eps_0$ in such a way that $u_0 = 2 (1 - p ( \eps_0))$.
\end {proof}

Define now a particular class of iterative exploration procedures as follows: Let $U_0=\U$ and $\Phi_0(z) = z$. For $j \ge 0$:
\begin {itemize}
\item
Choose some $y_j$ on $\partial \U$ in such a way that $\Phi_j^{-1} ( D( y_j, \eps)) \subset A$.
\item
Define $U_{j+1}$ as the connected component that contains the origin of the set obtained by removing from $U_{j+1}' : = U_j \setminus \Phi_j^{-1} (D(y_j, \eps))$ all the loops in $\Gamma$ that do not stay in $U_{j+1}'$.
\item
Let $\Phi_{j+1}$ be the conformal map from $U_{j+1}$ onto $\U$ such that $\Phi_{j+1} (0)= 0$ and
$\Phi_{j+1}'(0) > 0$.
\end {itemize}

There is only one way in which such an iterative definition can be brought to an end, namely if at some step $N_0$, it is not possible anymore to find a point $y$ on  $\partial \U$ such that $\Phi_{N_0}^{-1} ( D (y, \eps)) \subset A$
(otherwise it means that at some step $N$, one actually has discovered the loop $\gamma_0$, so that $U_{N+1}$ is not well-defined but we know that this cannot be the case because we have assumed that $\gamma_0 \cap A = \emptyset$).
Such explorations $(\Phi_n, n \le N_0)$ will be called {\em $\eps$-admissible explorations of the pair $(\Gamma, A)$}.

\begin {figure}[htbp]
\begin {center}
\includegraphics [width=2.5in]{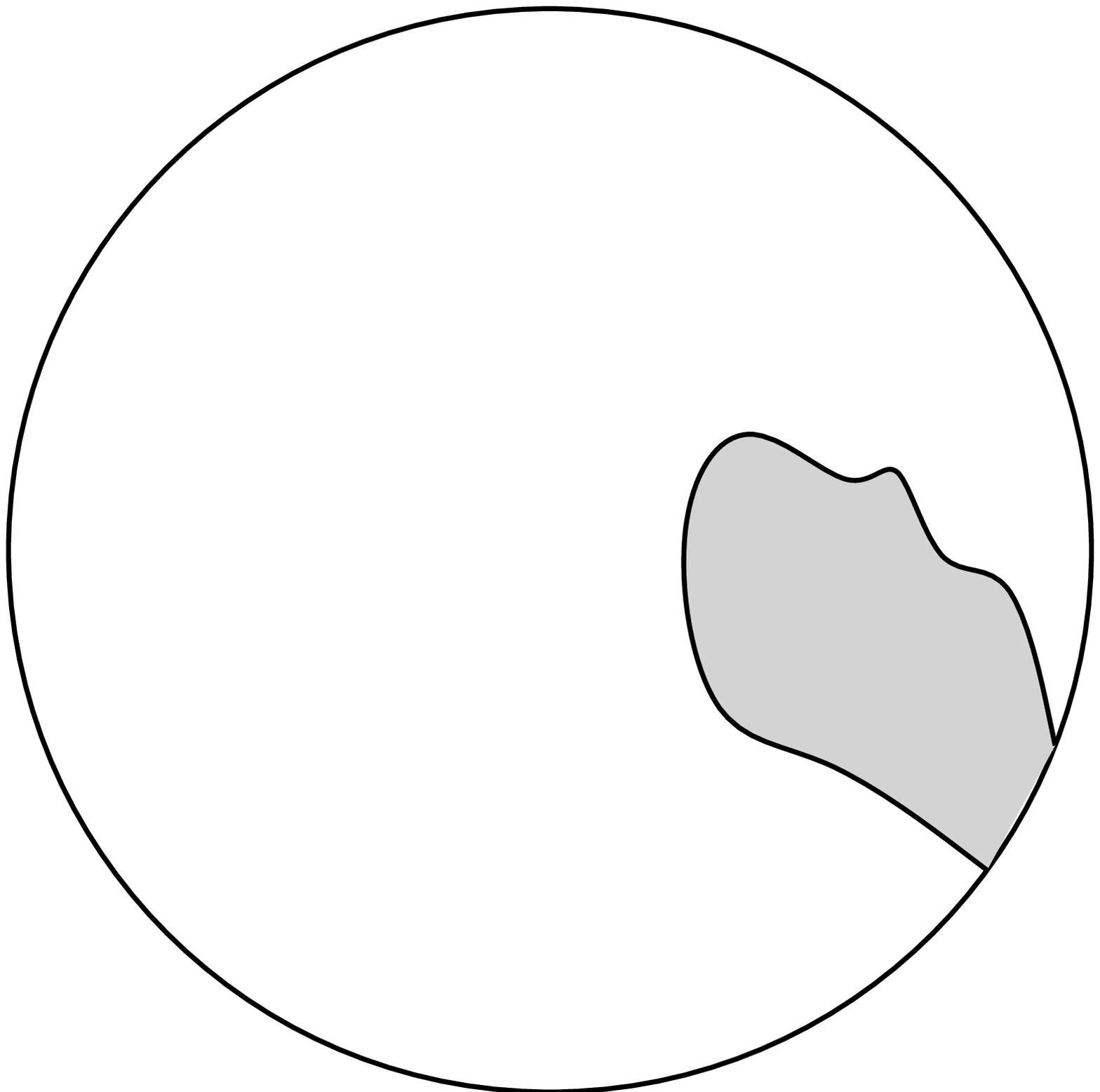}
\includegraphics [width=2.5in]{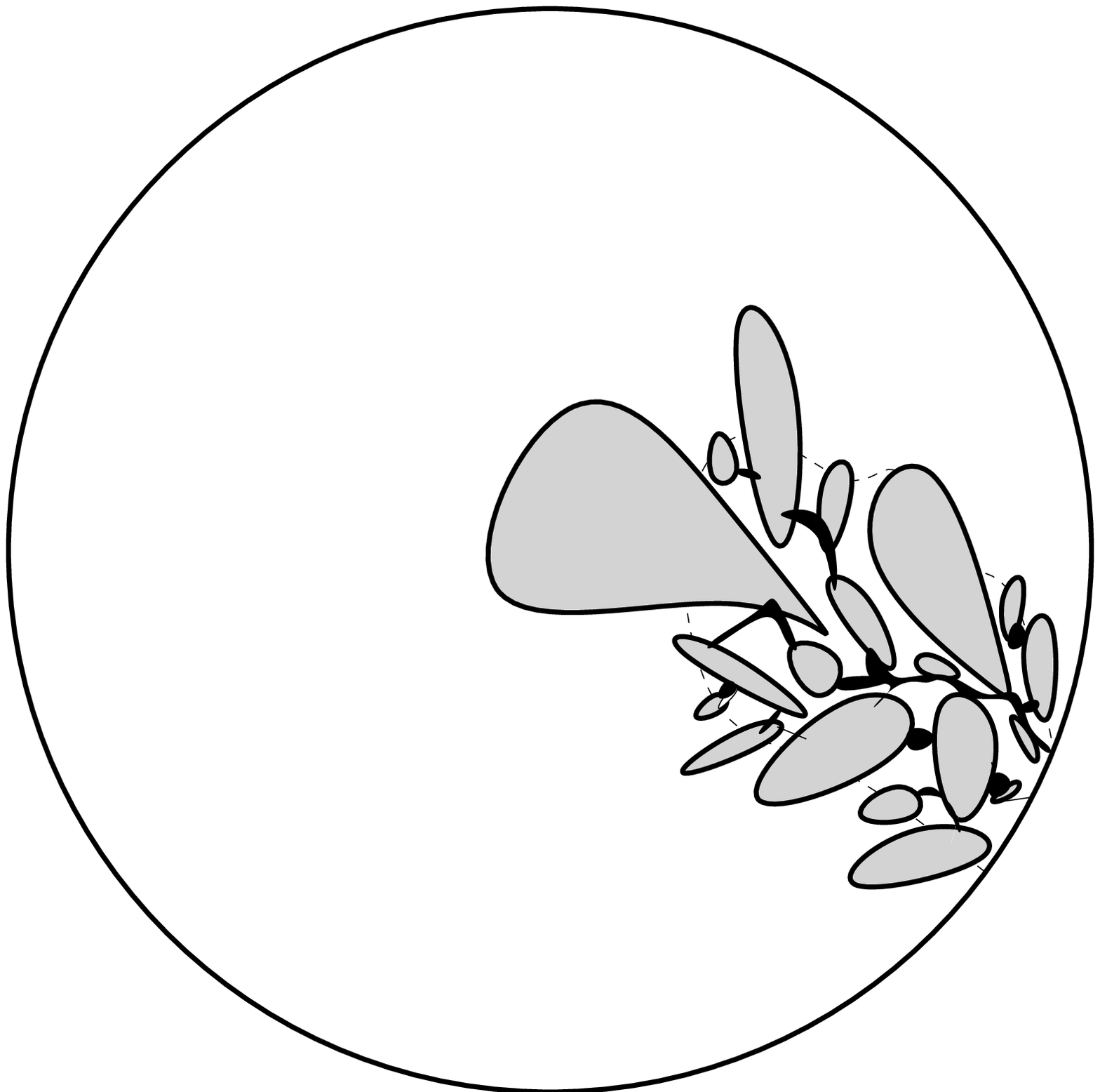}
\caption {A domain $A$, sketch of a completed $\eps$-admissible exploration of $(\Gamma, A)$}
\end {center}
\end {figure}
\medbreak

Our goal is to show that when $\eps$ gets smaller, the set $U_{N_0}$ is close to $\tilde U$, where
$\tilde U$ is the connected component containing the origin of  $U \setminus \cup_{i \in I(U) } \gamma_i $ (here $U= \U \setminus A$).

The local finiteness condition implies that the boundary of $\tilde U$ consists of points that are either on $\partial U$ or on some loop $\gamma_j$ (in this case, we say that this loop $\gamma_j$ contributes to this boundary).

\begin {lemma}
For  every $\alpha> 0$, there exists $\eps_0'= \eps_0' ( \Gamma, \alpha, A)$ such that for all $\eps \le \eps_0'$, every loop of diameter greater than $\alpha$ that contributes to $\partial \tilde U$ is discovered by any $\eps$-admissible exploration of $(\Gamma, A)$.
\end {lemma}

\begin {proof}
Suppose now that $\gamma_j$ is a loop in $\Gamma$ that contributes to the boundary of $\tilde U$. Our assumptions on $\Gamma$ and $A$ ensure that $\gamma_j$ therefore intersects both the interior of $A$ and $\U \setminus A$.  This implies that we can define three discs $d_1$, $d_2$ and $d_3$ in the interior of $\gamma_j$ such that $d_1 \subset \U \setminus A$, and $\overline {d_2} \subset d_3 \subset A$.

Suppose that for some $n\le N_0$, this loop $\gamma_j$ has not yet been discovered at step $n$.
Since $\gamma_j \cap \partial \tilde U \not= \emptyset$ and
$\tilde U \subset U_n$, we see that $\gamma_j \subset U_n$. Since this loop has a positive diameter, and since $\Gamma$ is locally finite, we can conclude that
with a positive probability $u$ that depends on $(\Gamma, A, \gamma_j, d_1, d_2, d_3)$, two Brownian motions $B^1$ and $B^2$  started from the origin
behave as follows:
\begin {figure}[htbp]
\begin {center}
\includegraphics [width=3.2in]{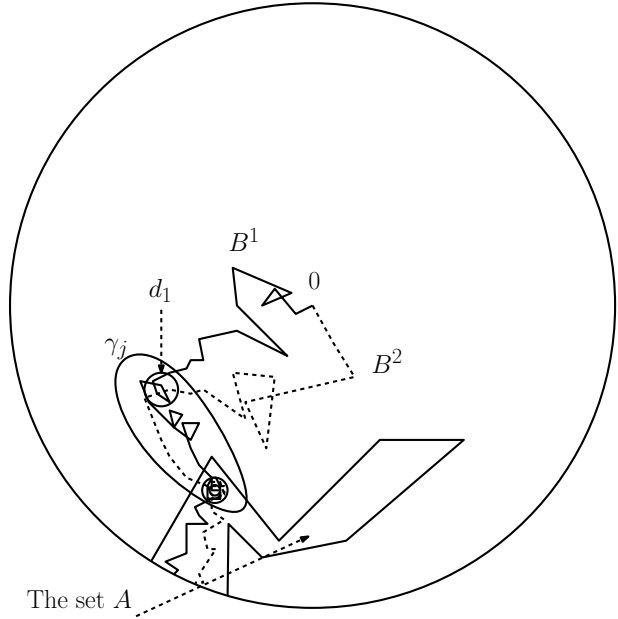}
\caption {The loop $\gamma_j$, the three discs and the Brownian motions (sketch)}
\end {center}
\end {figure}
\begin {itemize}
\item They both enter $d_1$ without hitting $A$ or $\partial U$ or any of the other loops $\gamma_i$ for $i \in I(U)$.
\item They both subsequently enter  $d_2$ without going out of $\interior( \gamma_j)$.
\item They both subsequently disconnect $d_2$ from the boundary of $d_3$ before hitting it.  (This in particular guarantees that the curves hit one another within
the annulus $d_3 \setminus d_2$.)
\item They both subsequently hit $\partial \U$ without going out of $A$.
\end {itemize}
This shows that one of the sets $U^a$ and $U^b$ as defined before Lemma \ref {UaUb} is contained in $A$ with probability at least $u$.
In fact, if we stop the two Brownian motions at their first exit of $U_n$ instead on the hitting time of $\partial \U$, the same phenomenon will hold: One of the two sets
$U^a_n$ and $U^b_n$ (with obvious notation) will be contained in $U_n \cap A$ with probability at least $u$. By conformal invariance of planar Brownian motion, if we apply Lemma \ref {UaUb} to the conformal images of these two Brownian motions under $\Phi_n$, we get that if $\eps$ is chosen to be sufficiently small, then it is
always possible to  find an $\eps$-admissible point $y_{n+1}$. Hence, $N_0 > n$, i.e.\ $n$ is not  the final step of the exploration.

As a consequence, we see that the loop $\gamma_j$ is certainly discovered before $N_0$, i.e., that $\gamma_j \subset \U \setminus  U_{N_0}$, for all $\eps \le \eps_0' ( \gamma_j, \Gamma, A)$. The lemma follows because for each positive $\alpha$, there are only finitely many loops of diameter greater than $\alpha$ in $\Gamma$.
\end {proof}

Loosely speaking, this lemma tells us that indeed, $U_{N_0}$ converges to $\tilde U$ as $\eps \to 0$. We now make this statement more precise, in terms of the conformal maps $\Phi_{N_0}$ and $\tilde \Phi$, where
 $\tilde \Phi$ denotes the conformal map from $\tilde U$ onto $\U$ with $\tilde \Phi(0)=0$ and $\tilde\Phi' (0) >0$.
 Let us first note that
$\tilde U \subset U_{N_0}$  because the construction of $U_{N_0}$ implies that before $N_0$ one can only discover loops that intersect $A$.

Let us now consider a two-dimensional Brownian motion $B$ started from the origin, and define $T$ (respectively $\tilde T$) the first time at which it exits $U$ (resp. $\tilde U$). Let us make a fifth assumption on $A$ and $\Gamma$:
\begin {itemize}
\item {5.}
Almost surely,  $B_T \in \partial \U  \cup ( \cup_j  \interior ( \gamma_j ) )$.
\end {itemize}

Note that this is indeed almost surely the case for a CLE (because $\Gamma$ is then independent of $B_T$ so that $B_T$ is a.s.\ in the interior of some loop if it is not on $\partial \U$).
This assumption implies  that almost surely, either $\tilde T < T$ (and $B_{\tilde T}$ is on the boundary of some loop of
positive diameter) or $B_T=B_{\tilde T} \in \partial \U$. The previous result shows that if $\hat T$ denotes the exit time of $U_{N_0}$ (for some given $\eps$-admissible exploration), then ${\hat T} = {\tilde T}$ for all small enough $\eps$.

It therefore follows that
$\Phi_{N_0}$ converges to $\tilde \Phi$ in the sense that for all proper compact subsets $K$ of $\U$, the functions $\Phi_{N_0}^{-1}$ converge uniformly to $\tilde \Phi^{-1}$ in $K$ as $\eps \to 0$. We shall use this notion of convergence on various occasions throughout the paper.
 Note that this it corresponds to the convergence with respect to a distance $d$, for instance
$$ d ( \varphi_1, \varphi_2) = \sum_{n \ge 1} 2^{-n} \max_{ |z| \le 1 - 1/n}  \| \varphi_1^{-1} (z) - \varphi_2^{-1} (z) \|.$$
We have therefore shown that:
\begin {lemma}
For each given loop configuration and $A$ (satisfying conditions 1-5), $d ( \Phi_{N_0} , \tilde \Phi)$ tends to $0$ as $\eps \to 0$,
uniformly with respect to all $\eps$-admissible explorations of $(\Gamma, A)$.
\end {lemma}

Suppose now that $\gamma_0 $ intersects the interior of $A$. Exactly the same arguments show that there exists $\eps_1 = \eps ( \Gamma, A)$ such that for all
``$\eps$-admissible choices'' of the $y_j$'s for $\eps \le \eps_1$, one discovers $\gamma_0$ during the exploration (and this exploration is then stopped in this way).

\subsection {Discovering random configurations along some given line}

For each small $\delta$, we define the wedge
$W_\delta= \{ u e^{i \theta} \ : u \in (0,1) \hbox { and } |\theta| \le \delta \}$. For each positive $r$, let  $\tilde W_r$ denote the image of the positive half disc
$\{ z \in \U \ : \ \Re (z) > 0 \}$ under the Moebius transformation of the unit disc with $\psi (1)=1$, $\psi (-1) = -1$ and $\psi' (1) =r$.
Note that $r \mapsto \tilde W_r$ is continuously increasing on $(0,1]$ from $\tilde W_{0+} = \{ 1 \}$ to the positive half-disc $\tilde W_1$.
For all non-negative integer $k \le 1/ \delta$, we then define
$$
A^{\delta, k} = W_\delta \cap \tilde W_{k \delta}.
$$
Suppose that $\delta$ is fixed, and that $\Gamma$ is a loop-configuration satisfying conditions 1-5 for all set $A^{\delta, k}$ for $k \le K$,
where
$$
K = K ( \Gamma, \delta) = \max \{k \ : \ \gamma_0 \cap A_k^\delta = \emptyset \}.
$$
We are going to define the conformal maps $\tilde \Phi^{\delta, 1}, \ldots, \tilde \Phi^{\delta, K}$ corresponding to the conformal map $\tilde \Phi$ when $A$ is respectively equal to $A^{\delta, 1}, \ldots, A^{\delta, K}$.

For each given $\delta$ and $\eps$, it is possible to define an $\eps$-admissible chain of explorations of $\Gamma$ and $A^{\delta, 1}, A^{\delta, 2}, \ldots $ as follows:
Let us first start with an $\eps$-admissible exploration of $(\Gamma, A^{\delta,1})$. If $K \ge 1$, then such an exploration does not encounter $\gamma_0$, and we then continue to explore until we get an $\eps$-admissible exploration of $(\Gamma , A^{\delta,2})$, and so on, until the last value $K'$ of $k$ for which the exploration of $(\Gamma, A^{\delta, k})$ fails to discovers $\gamma_0$.
In this way, we define conformal maps
$$ \tilde \Phi_\eps^{\delta, 1}, \ldots , \tilde \Phi_\eps^{\delta, K'}$$
corresponding to the sets discovered at each of these $K'$ explorations. Note that $K' \ge K$.
One can then also start to explore the set $A^{\delta, K'+1}$ until one actually discovers the loop $\gamma_0$.

\begin {figure}[htbp]
\begin {center}
\includegraphics [width=3in]{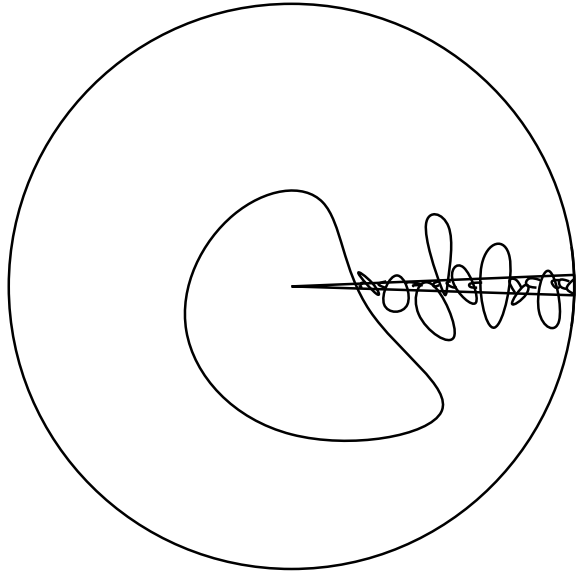}
\caption {An exploration-chain and $\gamma_0$ (sketch)}
\end {center}
\end {figure}

This procedure therefore defines a single $\eps$-admissible exploration (via some sequence $(\Phi_n, y_n)$), that explores the sets $A^{\delta, j}$'s in an ordered way, and finally stops at some step $N$,
i.e., the last step before one actually discovers $\gamma_0$. We call this an $\eps$-admissible exploration chain of $A^{\delta, 1}, A^{\delta, 2}, \ldots$.
Our previous results show that uniformly over all such $\eps$-admissible exploration-chains (for each given $\Gamma$ and $\delta$):
\begin {itemize}
\item $K'=K$ for all sufficiently small $\eps$.
\item $\lim_{\eps \to 0} (\tilde \Phi_\eps^{\delta, 1}, \ldots , \tilde \Phi_\eps^{\delta, K'}) =( \tilde \Phi^{\delta, 1}, \ldots, \tilde \Phi^{\delta, K})$.
\end {itemize}

We now suppose that  $\Gamma$ is a {\em random} loop-configuration. Then, for each $\delta$, $K=K(\Gamma, \delta)$ is random. We assume that for each given $\delta$,
the conditions 1-5 hold almost surely for each of the $K$ sets $A_1^\delta, \ldots, A_K^\delta$.
The previous results therefore hold almost surely;
this implies for instance that for each $\eta >0$, there exists $\eps_2 ( \delta)$ such that for all such $\eps$-admissible exploration-chain of $(\Gamma, A_1^\delta, A_2^\delta,  \ldots)$ with $\eps \le \eps_2 (\delta)$,
$$
P ( K'=K \hbox { and } d ( \tilde \Phi_\eps^{\delta, K}, \tilde \Phi^{\delta, K}) < \eta ) \ge 1- \eta.
$$

We will now wish to let $\delta$ go to $0$ (simultaneously with $\eps$, taking $\eps(\delta)$ sufficiently small)
so that we will (up to small errors that disappear as $\eps$ and $\delta$ vanish) just explore the loops that intersect the
segment $[0,1]$ ``from $1$ to $0$'' up to the first point at which it meets $\gamma_0$. We therefore define
$$ R = \max \{ r \in [0,1] \ , \ r \in \gamma_0 \}.$$
We define the open set $\hat U$ as the connected component containing the origin of the set obtained by removing from
$\U \setminus [R, 1]$ all the loops that intersect $(R, 1]$. Note that $\gamma_0 \subset \hat U \cup \{R \}$.
We let $\hat \Phi$ denote the conformal map from $\hat U$ onto $\U$ such that $\hat \Phi (0)=0$ and $\hat \Phi' (0) > 0$.
We also define
$ \hat y = \hat \Phi (R)$.

\begin {figure}[htbp]
\begin {center}
\includegraphics [width=3in]{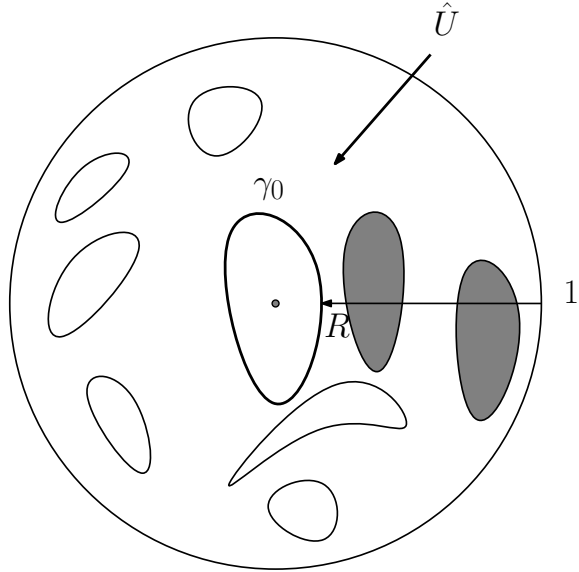}
\caption {Exploring up to $\gamma_0$ (sketch)}
\end {center}
\end {figure}

\begin {proposition}
\label {pwellchosen}
For a well-chosen function $\eps_3= \eps_3 ( \delta)$ (that depends on the law of $\Gamma$ only), for any choice of $\eps(\delta)$-admissible exploration-chain of the random loop-configuration $\Gamma$ and
$A^{\delta,1}, \ldots$ with $\eps(\delta) \le \eps_3 (\delta)$,  the random pair $(\Phi_N, y_N)$ converges almost surely to the pair $(\hat \Phi, \hat y)$
as $\delta \to 0$.
\end {proposition}

\begin {proof}
Note first that our assumptions on $\Gamma$ imply that almost surely $\Phi^{\delta, K} \to \hat \Phi$ as $\delta \to 0$ (this is a statement about $\Gamma$ that does not involve explorations).

We know that $\gamma_0$ intersects $A^{\delta, K+1} \setminus A^{\delta, K}$. The local finiteness of $\Gamma$ and the fact that any two loops are disjoint therefore implies that the diameter of the second largest loop (after $\gamma_0$) of $\Gamma$ that intersects this set almost surely tends to $0$ as $\delta \to 0$.
This in particular implies that almost surely, the distance between $\Phi_N$ and  $\tilde \Phi_\eps^{\delta, K}$  tends to $0$ as $\delta \to 0$ (uniformly with respect to the choice of the exploration, as long as $\eps(\delta)$ tends to $0$ sufficiently fast).

Recall finally that for each given $\delta$, $\tilde \Phi_\eps^{\delta, K} \to \Phi^{\delta, K}$ as $\eps \to 0$. Hence, if  $\eps(\delta)$ is chosen small enough, the map $\Phi_N$ indeed converges almost surely to $\hat \Phi$ as $\delta \to 0$.

It now remains to show that $y_N \to \hat y$.
Note that $R \in U_N$ (because at that step, $\gamma_0$ has not yet been discovered),
that $R \in A^{\delta, K'+1} \setminus A^{\delta, K'}$ (with high probability, if $\eps(\delta)$ is chosen
to be small enough).
On the other hand, the definition of the exploration procedure and of $K'$ shows that
$$\Phi_N^{-1} (D(y_N, \eps)) \cap (A^{\delta, K'+1} \setminus A^{\delta, K'}) \not= \emptyset$$
so that if we choose $\eps (\delta)$ small enough, then the Euclidean distance between ${\Phi_N^{-1} (D(y_N, \eps))}$ and $R$
tends to $0$ almost surely.

Let us look at the situation at step $N$: The loop $\Phi_N ( \gamma_0)$ in the unit disc is intersecting $D(y_N, \eps)$ (by definition of $N$), and it contains also the point $\Phi_N (R)$ (because $R \in \gamma_0$).
Suppose that
$|\Phi_N (R) - y_N |$ does not almost surely tend to $0$ (when $\delta \to 0$); then, with positive probability,
 we could find a sequence $\delta_j \to 0$ such that
$\Phi_N(R)$ and $y_N$ converge to different points on the unit circle along this subsequence.
In particular, the harmonic measure at the origin of any of the two parts of the loop
between the moment it visits $\Phi_N(R)$ and the $\eps$-neighborhood of $y_N$ in $\U$ is bounded away from $0$.
Hence, this is also true for the preimage $\gamma_0$ under $\Phi_N$: $\gamma_0$ contains  two disjoint paths from $R$ to
$\Phi_N^{-1} (D( y_N, \eps))$ such that their harmonic measure at $0$ in $\U$ is bounded away from $0$.
Recall that
$ \Phi_N^{-1} ( D( y_N, \eps)) \subset A^{\delta, K'+1} $.
In the limit when $\delta \to 0$, we therefore end up with a contradiction,
as we have two parts of $\gamma_0$ with positive harmonic measure from the origin, that join $R$ to some point of $[R,1]$, which is not possible
because $\gamma_0 \cap (R,1] \not= \emptyset$ and $\gamma_0$ is a simple loop.
Hence, we can conclude that $ |\Phi_N (R) - y_N | \to 0$ almost surely.

Finally, let us observe that $\Phi_N (R) \to \hat \Phi (R)$ almost surely (this follows for instance from the fact that a continuous path that stays inside $\gamma_0$
and joins the origin to $R$ stays both in
all $U_N$'s and in $\hat U$).
It follows that $y_N$ converges almost surely to $\hat y$.
\end {proof}

\section {The one-point pinned loop measure}
\label {S.5}

\subsection {The pinned loop surrounding the origin in $\U$}

We will use the previous exploration mechanisms in the context of CLEs.
It is natural to define the notion of Markovian explorations of a CLE.
Suppose now that $\Gamma$ is a CLE in the unit disc and that $\eps$ is fixed.
When $\gamma_0 \cap D(1, \eps) = \emptyset$, define just as before the set $U_1$ and the conformal map $\Phi_1$ obtained by
discovering the set of loops $\Gamma_1$ of $\Gamma$ that intersect $D(y_0, \eps)$.
Then we choose $y_1$ and proceed as before, until we discover (at step $N+1$) the loop $\gamma_0$ that surrounds the origin.
We say that the exploration is Markovian if for each $n$, the choice of $y_n$ is measurable with respect to the $\sigma$-field generated
by $\Gamma_1, \ldots, \Gamma_n$, i.e., the set of all already discovered loops.

A straightforward consequence of the CLE's restriction property is that for each $n$, conditionally on $\Gamma_1, y_1,  \ldots, \Gamma_n, y_n$ (and $n \le N$), the law of the set of loops of $\Gamma$ that stay in $U_n$  is that of a CLE in $U_n$.
In other words, the image of this set of loops under $\Phi_n$ is independent of  $\Gamma_1, y_1,  \ldots, \Gamma_n, y_n$ (on the event $\{ n \le N \})$.
In fact,  we could have used this independence property as a definition of Markovian explorations (it would allow extra randomness in the choice of the sequence $y_n$).

In other words, an exploration is Markovian if we can choose $y_n$ as we wish using the information
 about the loops that have already been discovered, but we are not allowed to use any information about the yet-to-be-discovered loops.
This ensures that one obtains an iteration of i.i.d.\ explorations as argued in subsection \ref {heuristics}.
In particular, if an exploration is Markovian, the random variable $N$ is geometric
$$
P ( N \ge n ) = P ( \gamma_0 \cap D(1, \eps) = \emptyset)^n,
$$
and $y_N^{-1} \Phi_N^{-1} ( \gamma_0 )$ is distributed according to the conditional law of $\gamma_0$ given $\{\gamma \cap D(1, \eps) \not= \emptyset\}$.

Recall that a CLE is a random loop configuration such that for any given $\delta$ and $k \le 1/\delta$, almost surely, all loops that intersect $A^{\delta, k}$ also intersect its interior.   We can therefore apply Proposition~\ref {pwellchosen}, and use
Markovian $\eps(\delta)$-admissible successive explorations of
$A^{\delta,1}, A^{\delta, 2},  \ldots$ Combining this with our description of  the
conditional law of $\gamma_0$ given $\{\gamma \cap D(1, \eps) \not= \emptyset\}$, we get the following result:

\begin {proposition}
\label {p44}
When $\eps \to 0$, the law of $\gamma_0$ conditioned on the event $\{\gamma_0 \cap D(1,\eps) \not= \emptyset \}$ converges to the law of
$\hat y^{-1}  \hat \Phi (\gamma_0) $ (using for instance the weak convergence with respect to the Hausdorff topology on compact sets).
\end {proposition}

Note that local finiteness of the CLE ensures that $\hat \Phi ( \gamma_0)$ is a simple loop in $\overline \U $  that intersects $\partial \U$ only at $\hat y$, so that
$\hat y^{-1} \hat \Phi (\gamma_0)$ is indeed a loop in $\overline \U$ that touches $\partial \U$ only at $1$.

\begin {figure}[htbp]
\begin {center}
\includegraphics [width=2.3in]{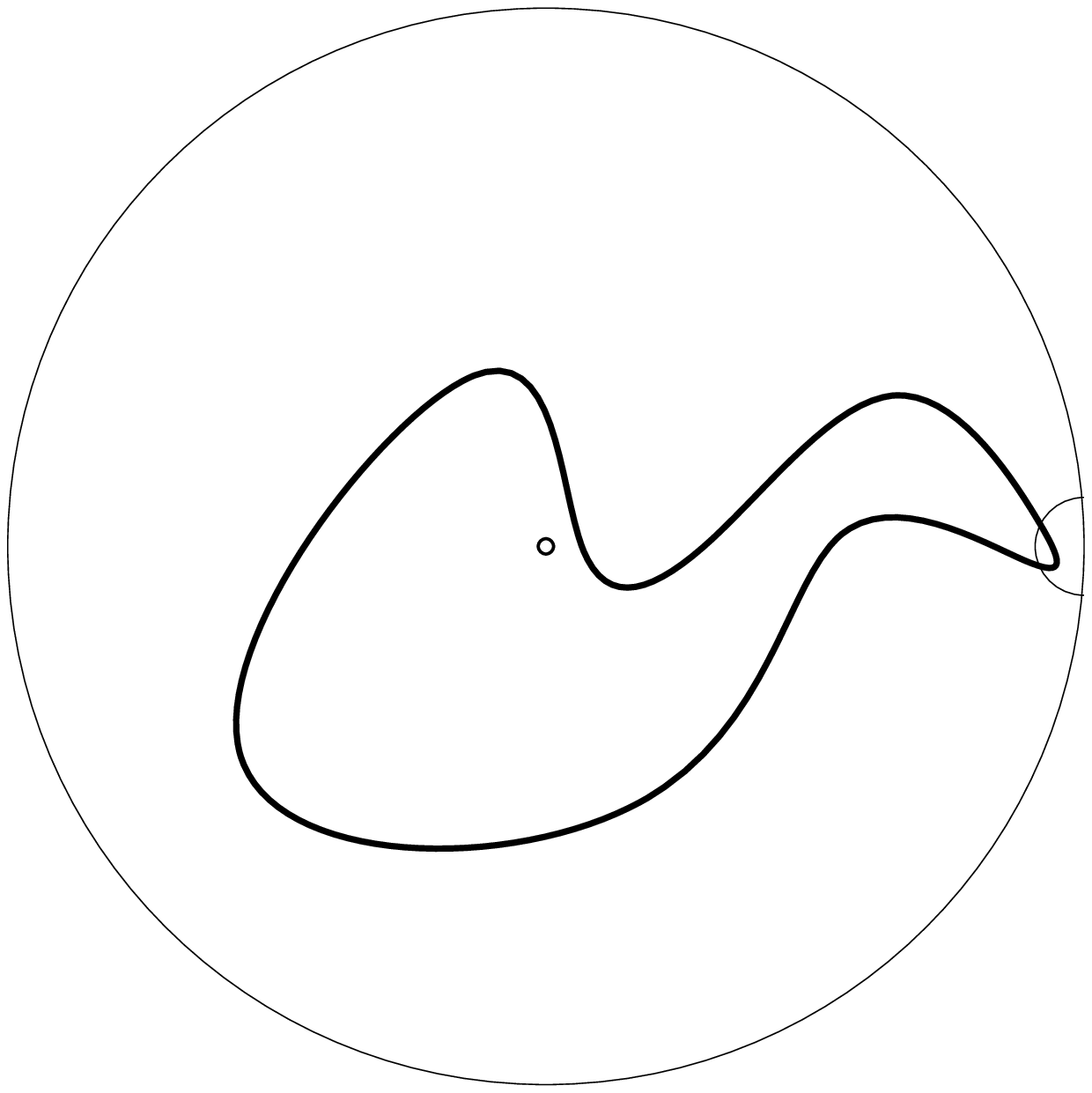}
\includegraphics [width=2.3in]{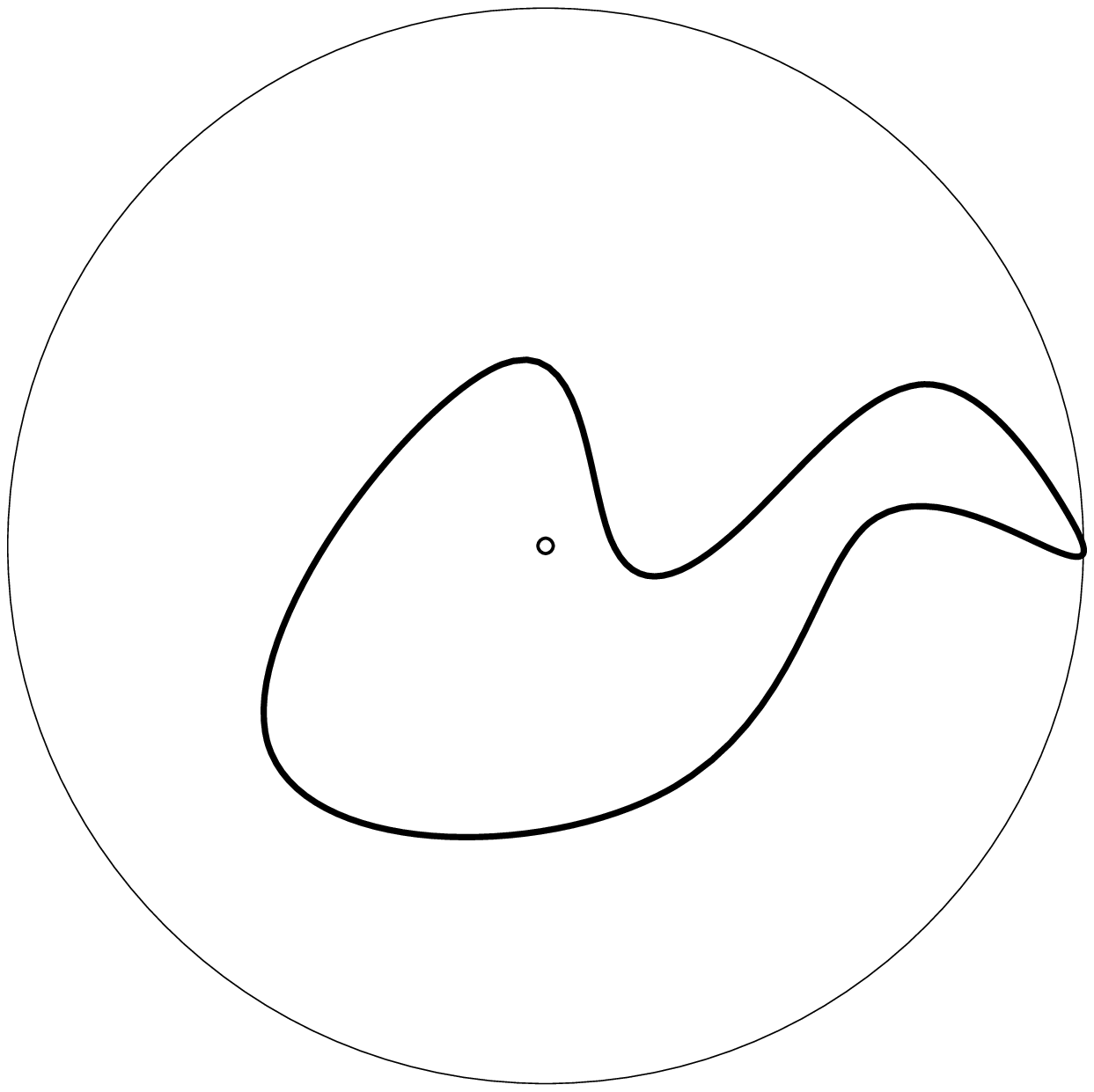}
\caption {The law of $\gamma_0$ conditioned on the event $\{\gamma_0 \cap D(1,\eps) \not= \emptyset \}$ converges (sketch)}
\end {center}
\end {figure}

This limiting law will inherit from the CLE various interesting properties. The loop $\gamma_0$ in the CLE can be discovered along the ray $[1,0]$ in the unit disc as in this proposition, but one could also have chosen any other smooth continuous simple curve from $\partial \U$ to $0$ instead of that ray and discovered it that way. This fact should correspond to some property of the law of this pinned loop.
Conformal invariance of the CLE will also imply some conformal invariance properties of this pinned loop.
 The goal of the coming sections is to derive and exploit some of these features.

\subsection {The infinite measure on pinned loops in $\HH$}

We are now going to associate to each CLE  a natural measure on loops that can loosely be described as the law of a loop ``conditioned to touch a given boundary point''. In the previous subsection, we have constructed a probability measure on loops in the unit disc that was roughly the law of $\gamma_0$ (the loop in the CLE that surrounds the origin)  ``conditioned to touch the boundary point $1$''.
We will extend this to an infinite measure on loops that touch the boundary at one point; the measure will be infinite because we will not prescribe the ``size'' of the boundary-touching loop; it can be viewed as a CLE ``excursion measure'' (``bubble measure'' would also be a possible description).
\begin {figure}[htbp]
\begin {center}
\includegraphics [width=2.3in]{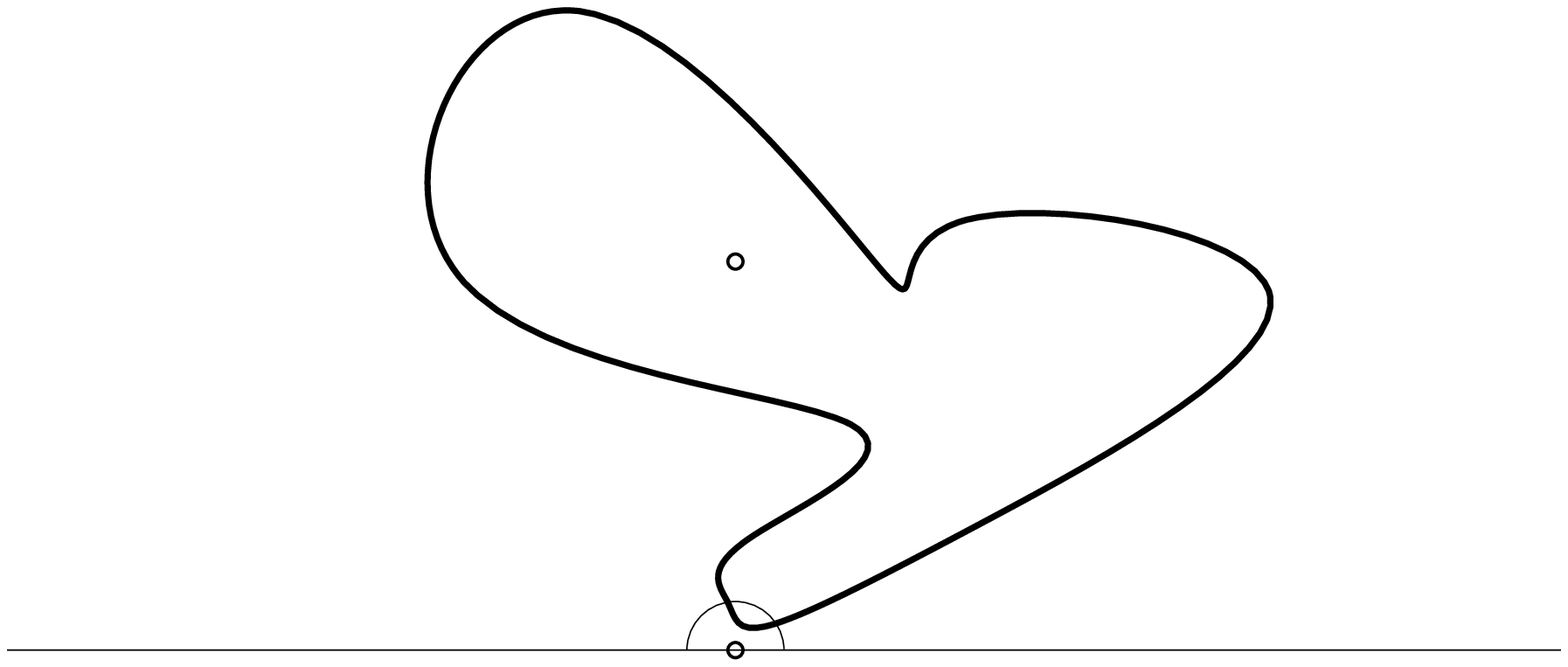}
\includegraphics [width=2.3in]{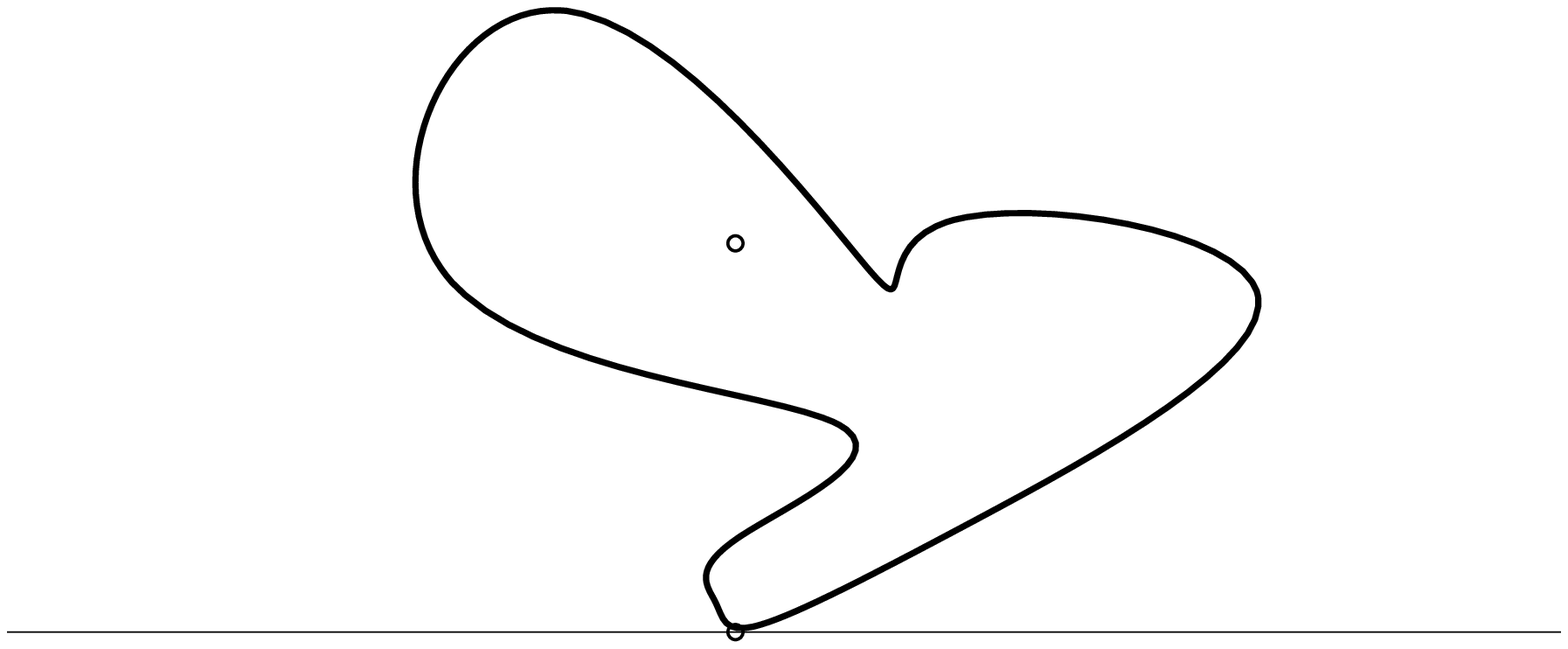}
\caption {The construction of $\mu^i$ (sketch)}
\end {center}
\end {figure}
We find it more convenient at this stage to work in the upper half-plane rather than the unit disc because scaling arguments will be easier to describe
in this setting.
We therefore first define the probability measure $\mu^i$ as the image of the law of $\hat \Phi (\gamma_0) $ under the conformal map from $\U$ onto the half-plane that maps $0$ to $i$, and $\hat y$ to $0$ (we use the notation $\mu^i$ instead of $P^i$ as in the introduction, because we will very soon be handling measures that are not probability measures).
In other words, the probability measure $\mu^i$ is the limit as $\eps \to 0$ of the law of the loop $\gamma (i)$ that surrounds $i$ in a CLE in the upper half-plane, conditioned by the fact that it intersects  the set $C_\eps$ defined by
$$ C_\eps = \{ z \in \HH \ : \ | z | = \eps \} .$$

For a loop-configuration $\Gamma$ in the upper half-plane and $z \in \HH$, we denote by $\gamma (z)$ the loop of $\Gamma$ that surrounds $z$ (if this  loop exists).

When $z = i \lambda $ for $\lambda > 0 $, scale invariance of the CLE shows that the limit as $\eps \to 0$ of the conditional law of $\gamma ( i \lambda) $ given that it intersects the disc of radius $\lambda \eps$ exists, and that it is just the image of $\mu^i$ under scaling.

Let us denote by $u(\eps)$ the probability that the loop $\gamma (i)$ intersects the disc of radius $\eps$ around the origin in a CLE.
The description of the measure $\mu^i$ (in terms of $\hat \Phi$) derived in the previous section shows that at least for almost all $\lambda$ sufficiently close to one, the loop that surrounds $i$ also surrounds $i \lambda$ and $i / \lambda$ with probability at least $1/2$ under $\mu^i$, and that $i / \lambda$ as well as $\lambda i$ are a.s.  not on $\gamma (i)$ (when $\lambda$ is fixed).

Let $O_i$ denote the interior of the loop $\gamma (i)$.
We know that
$$ \lim_{\eps \to 0}
\frac { P \left( \lambda i \in O_i \hbox { and } {\gamma(i) \cap C_{\lambda \eps} \not= \emptyset} \right) }{ u (\lambda \eps)} =
 \mu^i ( \lambda i \in O_i) .
$$
On the other hand, the scaling property of the CLE shows that when $\eps \to 0$,
\begin {eqnarray*}
\lefteqn {\frac { P \left( \lambda i  \in O_i  \hbox { and } {\gamma(i) \cap C_{\lambda \eps} \not= \emptyset} \right) }{ u (\lambda \eps)}}\\
&=& \frac { P \left(  i  \in O_{\lambda i}  \hbox { and } {\gamma(\lambda i) \cap C_{\lambda \eps} \not= \emptyset} \right) }{ u (\lambda \eps)} \\
&=&
\frac { P \left( i / \lambda \in O_i  \hbox { and } {\gamma(i) \cap C_{\eps} \not= \emptyset} \right) }{ u ( \eps)} \times
 \frac {u(\eps)}{u(\lambda \eps)}
 \\
 &\sim&
 \mu^i (  i/ \lambda \in O_i) \times   \frac {u(\eps)}{u(\lambda \eps)}
 \end {eqnarray*}
Hence, for  all $\lambda$ sufficiently close to $1$, we conclude that
$$ \lim_{\eps \to 0} \frac {u (\lambda \eps)}{ u (\eps)}
=
\frac {  \mu^i (  i/ \lambda \in O_i) } { \mu^i ( \lambda i \in O_i )}.
$$
If we call $f(\lambda)$ this last quantity, this identity clearly implies that this convergence in fact holds for all positive $\lambda $, and that $f( \lambda \lambda') = f (\lambda) f( \lambda')$. Furthermore, we see that
$f (\lambda) \to 1$ as $\lambda \to 1$. Hence:

\begin {proposition}
There exists a $\beta \ge 0$ ($\beta$ cannot be negative since $\eps \mapsto u(\eps)$ is non-decreasing) such that for all positive $\lambda$,
$$
f(\lambda) = \lim_{\eps \to 0} \frac { u( \lambda \eps)} {u(\eps)} = \lambda^{\beta}.
$$
\end {proposition}
This has the following consequence:
\begin {corollary}
$u( \eps)= \eps^{\beta+ o(1) }$ as $\eps \to 0+$.
\end {corollary}
\begin {proof}
Note that for any $\beta'< \beta < \beta''$, there exists $\eps_0= 2^{-n_0}$ such that for all $\eps \le \eps_0$,
$$ 2^{\beta'} < u (2\eps)/ u (\eps)  < 2^{\beta''}.$$
Hence, for all $n \ge 0$ it follows that
$$u (\eps_0)  2^{-n \beta''}  < u  (\eps_0 2^{-n}) <  u (\eps_0) 2^{-n \beta'}.$$
It follows that $u(\eps_0 2^{-n}) = 2^{-n\beta + o(n)}$ when $n \to \infty$.
Since $\eps \mapsto u(\eps)$ is an non-decreasing function of $\eps$, the corollary follows.
\end {proof}

\medbreak

Because of scaling, we can now define for all $z= i \lambda$, a measure $\mu^z$ on  loops $\gamma (z)$ that surround $z$ and touch the real line at the origin as follows:
$$
\mu^z ( \gamma (z)  \in \A ) = \lambda^{-\beta} \mu^i ( \lambda \gamma(i) \in \A )
$$
for any measurable set $\A$ of loops. This is also the limit of $u( \eps)^{-1}$ times the law of $\gamma (z)$ in an CLE, restricted to the event $\{\gamma (z) \cap C_\eps \not= \emptyset \}$.

\medbreak

Let us now choose any $z$ in the upper half-plane. Let $\psi= \psi_z$ now denote the Moebius transformation from the upper half-plane onto itself with
$\psi ( z) = i$ and $\psi (0) = 0$. Let $\lambda = 1/ \psi' (0)$.
Clearly, for any given $a >1$, for any small enough $\eps$, the image of $C_\eps$ under $\psi$ is ``squeezed'' between the circles
$ C_{ \eps / a\lambda}$ and $C_{a\eps/ \lambda}$. It follows readily (using the fact that $f(a) \to 1$ as $a \to 1$) that the measure
$\mu^z$ defined for all measurable $\A$ by
$$
\mu^z ( \gamma (z) \in \A  ) = \lambda^{-\beta} \mu^i (  \psi^{-1} ( \gamma (i) ) \in \A )
$$
can again be viewed as the limit when $\eps \to 0$ of $u(\eps)^{-1}$ times the distribution of $\gamma (z)$ restricted to $\{ \gamma (z) \cap C_\eps
\not= \emptyset \}$.

\medbreak
Finally, we can now define our measure $\mu$ on pinned loops. It is the measure on simple loops that touch the real line at the origin and otherwise stay in the upper half-plane (this is what we call a pinned loop) such that for all $z \in \HH$, it coincides with $\mu^z$ on the set of loops that surround $z$.
Indeed, the previous limiting procedure shows immediately that for any two points $z$ and $z'$, the two measures $\mu^z$ and $\mu^{z'}$ coincide on the set of loops that surround both $z$ and $z'$. On the other hand, we know that a pinned loop necessarily surrounds a small disc: thus, the requirement that $\mu$ coincides with the $\mu^z$'s (as described above) fully determines $\mu$.

\medbreak

Let us sum up the properties of the pinned measure $\mu$ that we will use in what follows:
\begin {itemize}
\item
For any conformal transformation $\psi$ from the upper half-plane onto itself with
$\psi (0) = 0$, we have
$$ \psi \circ \mu = |\psi' (0)|^{-\beta} \mu.$$
This is the {\em conformal covariance property} of $\mu$.
Note that the maps $z \mapsto -za/(z-a)$ for real $a \not= 0$ satisfy $\psi'(0)=1$ so that $\mu$ is invariant under these transformations.
\item
For each $z$ in the upper half-plane, the mass $\mu ( \{ \gamma :  z \in \interior (\gamma ) \})$ is finite and equal to $\psi'(0)^\beta$, where $\psi$ is the conformal map from $\HH$ onto itself with $\psi (0)=0 $ and $\psi(z)=i$.
\item
For each $z$ in the upper half-plane, the measure $\mu$ restricted to the set of loops that surround $z$ is the limit as $\eps \to 0+$ of $u(\eps)^{-1}$ times the law of $\gamma (z)$ in a CLE restricted to the event $\{ \gamma(z) \cap C_\eps \not= \emptyset \}$. In other words, for any bounded continuous (with respect to the Hausdorff topology, say) function $F$ on the set of loops,
$$
\mu ( 1_{\{z \hbox { is surrounded by } \gamma \}} F ( \gamma ) )
= \lim_{\eps \to 0 } \frac {1}{u(\eps)} E ( 1_{\{\gamma (z) \cap C_\eps \not= \emptyset \}} F ( \gamma (z))).
$$
\end {itemize}

Since this pinned measure is defined in a domain with one marked point, it is also quite natural to consider it in the upper half-plane $\HH$, but to take the marked point at infinity. In other words, one takes the image of $\mu$ under the mapping $z \mapsto -1/z$ from $\HH$ onto itself. This is then a measure $\mu'$ on loops ``pinned at infinity, i.e.\ on double-ended infinite simple curves in the upper-half plane that go from and to infinity. It clearly also has the scaling property with exponent $\beta$, and the invariance under the conformal maps that preserve infinity and the derivative at infinity is just the invariance under horizontal translations.

\subsection {Discrete radial/chordal explorations, heuristics, background}
\label {S43}

We will now (and also later in the paper) use
 the exploration mechanism corresponding to the case where all points $y_n$ are chosen to be equal to $1$, instead of being tailored in order for the exploration to stay in some  a priori chosen set $A$ as in Section \ref {explo}.
Let us describe this discrete radial exploration (this is how we shall refer to it) in the setting of the upper half-plane $\HH$: We fix $\eps >0$, and
we  wish to explore the CLE in the upper half-plane by repeatedly cutting out origin-centered semi-circles $C_\epsilon$ of radius $\eps$ (and the loops they intersect) and applying conformal maps.
The first step is to consider the set of loops that intersect $C_\eps$. Either we have discovered the loop $\gamma (i)$ (i.e., the loop that surrounds $i$), in which case we stop, or we haven't, in which case we define the connected component of the complement of these loops in $\HH \setminus C_\eps$ that contains $i$, and map it back onto the upper half-plane by the conformal map $\varphi_1^\eps$ such that $\varphi_1^\eps (i) = i$ and $(\varphi_1^\eps)'(i) > 0$.
We then start again, and this defines an independent
copy $\varphi_2^\eps$ of $\varphi_1^\eps$. In this way, we define a random geometric number $N=N(\eps)$ of such conformal maps
$\varphi_1^\eps, \ldots , \varphi_N^\eps$. The $N+1$-th map can not be defined because one then discovers the loop that surrounds $i$.
 The probability that $N \ge n$ is equal to $(1- u(\eps))^n$.
The difference with the exploration procedure of Section \ref {explo} is that we do not
try to  explore ``along some prescribed curve'', but we just iterate i.i.d.\ conformal maps in such
a way that the derivative at $i$ remains a positive real, i.e.\ we consistently target the inner point $i$.

Another natural exploration procedure uses a discrete chordal exploration that targets a boundary point: We first consider the set of loops in a CLE in $\HH$ that intersect
$C_\eps$. Now, we consider the unbounded connected component $H_1$ of the complement in $\HH \setminus C_\eps$ of the union of these loops. We map it back onto the upper half-plane using the conformal map $\varphi_1^\eps$ normalized at infinity, i.e.\ $\varphi_1^\eps (z) \sim z + o(1)$ as $z \in \HH$. Then we iterate the procedure, defining an
infinite i.i.d. sequence $\varphi_1^\eps, \varphi_2^\eps, \ldots$ of conformal maps, and a decreasing family of domains $H_n
= ( \varphi_n^\eps  \circ \cdots \circ \varphi_1^\eps)^{-1} ( \HH)$ (unlike the radial exploration, this chordal exploration never stops).

Let us make a little heuristic discussion in order to prepare what follows. When $\eps$ is very small, the law of the sets that one removes at each step in these two exploration mechanisms can be rather well approximated thanks to the measure $\mu$.
Let us for instance consider the chordal exploration. Because of $\mu$'s scaling property, the $\mu$-mass of the set of loops of half-plane capacity (which scales like the square of the radius) larger than $\eta$ should decay as $\eta^{-\beta/2}$. In particular, in the discrete exploration, the average number of exploration steps where one removes a set $\varphi_n^\eps (H_n \setminus H_{n+1})$ with half-plane capacity larger than $x$ is equal (in the $\eps \to 0$ limit) to  $x^{- \beta/2}$ times the average
 number of exploration steps where one removes a set with half-plane capacity larger than one.

Readers familiar with L\'evy processes will probably have recognized that in the $\eps \to 0$ limits, the capacity jumps of the chordal exploration process will be distributed like
the jumps of a $(\beta/2)$-stable subordinator. It is important to stress that we know a priori that the limiting process has to possess macroscopic capacity jumps (corresponding to the discovered macroscopic loops).
Since $\alpha$-stable subordinators exist only for $\alpha \in (0,1)$, we expect that $\beta < 2$. Proving this last fact will be the main goal of this section.

In fact, we shall see that the entire discrete explorations (and not only the process of accumulated half-plane capacities) converge to a continuous ``L\'evy-exploration'' defined using a Poisson point process of pinned loops with intensity $\mu$. We however defer the more precise description  of these L\'evy explorations to Section \ref {S.8}, where we will also make the connection with branching SLE($\kappa, \kappa-6$) processes and show how to reconstruct the law of a CLE using the measure $\mu$ only, i.e., that the measure $\mu$ characterizes the law of the CLE.

In fact, the rest of this part of the paper (until the constructive part using Brownian loop-soups) has two main goals: The first one is to show that the CLE definition yields a description of the pinned measure $\mu$ in terms of SLE$_\kappa$ for some $\kappa \in (8/3, 4]$. The second one is to prove that the pinned measure $\mu$ characterizes the law of the entire CLE and also to make the connection with SLE($\kappa, \kappa-6$).
We choose to start with the SLE-description of $\mu$; in the next subsection, we will therefore only derive those two results that will be needed for this purpose,
leaving the more detailed discussion of continuous chordal explorations for later sections.

For readers who are not so familiar with L\'evy processes or Loewner chains, let us now briefly recall some  basic features of Poisson point processes and
the stability of Loewner chains that will hopefully help making the coming proofs more transparent.

\medbreak
\noindent
{\bf Stability of Loewner chains:}
Let ${\cal S}$ denote the class of conformal maps $\varphi=\varphi_H$ from a subset $H$ of the complex upper half-plane back onto the upper half-plane $\HH$ such that $i \in H$,  $\varphi (i)=i$  and
$\varphi' (i)$ is a positive real.
 Note that (because $H\subset \HH$), $\varphi'(i) \ge 1$.
Let us then define $a(\varphi) = \log \varphi' (i)$. This is a decreasing function of the domain $H$ (the smaller $H$, the larger $a$), which is closely related to the conformal radius of the domain (one can for instance conjugate with $z \mapsto (i-z)/(i+z)$ in order to be in the usual setting of the disc).
It is immediate to check that for some universal constant $C$, for all $r < 1/2$ and for all $\varphi$ such that the diameter of $\HH \setminus H$ is smaller than $r$,
$$ a (\varphi) \le C r^2.$$
On the other hand, for some other universal constant $C$, for all $r < 1/2$ and for all $H$ such that $ir \notin H$,
$$ a ( \varphi ) \ge C' r^2.$$

Suppose now that $\varphi_1, \ldots, \varphi_N$ are $N$ given conformal maps in ${\cal S}$. Define $\Phi =\varphi_N \circ \ldots  \circ \varphi_1$ to be the composition of these conformal maps.
It is an easy fact (that can be readily deduced from simple distortion estimates, or via Loewner's theory to approximate these maps via Loewner chains for instance) that
for any family $(\psi_0^\delta, \ldots , \psi_N^\delta)_{\delta > 0}$ of conformal maps in ${\cal S}$ such that
$$ \lim_{\delta \to 0} a ( \psi_0^\delta) + \ldots + a ( \psi_N^\delta) =  0, $$
the conformal maps
$$\Phi^\delta = \psi_N^\delta \circ \varphi_N \circ \psi_{N-1}^\delta \circ \varphi_{N-1} \circ \ldots \varphi_1 \circ \psi_0^\delta$$
converge in Carath\'eodory topology (viewed from $i$) to $\Phi$ as $\delta \to 0$.
In other words, putting some perturbations of the identity between the iterations of $\varphi$'s  does not change things a lot, as long as the accumulated ``size'' (measured by $a$) of the perturbations is small.

\medbreak
\noindent
{\bf Poisson point processes discrete random sequences:}
We will approximate Poisson point processes via discrete sequences of random variables. We will at some point need some rather trivial facts concerning Poisson random variables, that we now briefly derive.
Suppose that we have a sequence of i.i.d.\ random variables $(X_n, n \ge 1)$ that take their values in some finite set $\{0, 1, \ldots , k\}$ with $P(X_1=0) >0$. Let $N$ denote the smallest $n$ value at which $X_n = 0$. For each $j \in \{ 1, \ldots, k \}$, let $N_j$ denote the cardinality of $\{ n \in \{ 1, \ldots, N-1 \} \ : \ X_n = j \}$. We want to control the joint law of $(N_1, \ldots, N_k)$.

One convenient way to represent this joint law is to consider a Poisson point process $(X_j, T_j)_{j \in J}$ on ${\cal A} \times [0, \infty)$ with intensity ${\cal M} \times dt$, where ${\cal M}$ is some $\sigma$-finite measure on a
 space ${\cal A}$.
Consider disjoint measurable sets $A, A_1, \cdots , A_k$ in ${\cal A}$ with
${\cal M} ( A) = 1$, ${\cal M} (A_1) = a_1 < \infty$, \ldots, ${\cal M} (A_k) = a_k < \infty$.
 We define
$$ T = \inf \{ t > 0 \ : \ \exists j \in J \hbox { such that } T_j \le t \hbox { and } X_j \in A \},$$
i.e., loosely speaking, if we interpret $t$ as a time-variable, $T$ is the first time at which one observes an $X$ in $A$. Since ${\cal M} (A) =1$, the law of $T$ is exponential with parameter $1$.
In particular, $P ( T > 1) = 1/e$.
We now define, for $j=1, \ldots , k$, $$N_j = \# \{ j \ : \  T_j \le T \hbox { and } X_j \in A_j \},$$
i.e., the number of times one has observed an $X$ in $A_j$ before the first time at which one observes an $X$ in $A$.
Then the law of $N_1, \ldots , N_k$ is the same as before, where $P(X_1= j ) = a_j / (1+ a_1 +\ldots + a_j)$.

Because of the independence properties of Poisson point processes, the conditional law of $(N_1, \ldots, N_k)$ given $T$ is that of
$K$ independent Poisson random variables with respective means $T M(A_1), \ldots, T M(A_k)$. Hence, $E(N_j) = E(T) M (A_j) = M(A_j)$.
Furthermore, if we condition on the event  $\{T > 1\}$, the (joint) conditional distribution of $N_1, \ldots , N_k$ ``dominates''
that of $k$ independent Poisson random variables of parameter $a_1, \ldots, a_k$.

 \subsection {A priori estimates for the pinned measure} \label{s.tail}

Let us denote the ``radius''  of a loop $\gamma$ in the upper half-plane by
$$ R (\gamma ) = \max \{ |z|  \ : \ z \in \gamma \} .$$

\begin {lemma}
\label {LR1}
The $\mu$-measure of the set of loops with radius greater than $1$ is finite.
\end {lemma}

\begin {proof}
Let us  use the radial exploration mechanism.
The idea of the proof is to see that if the
 $\mu$-mass of the set of loops of radius greater than $1$ is infinite, then one ``collects'' too many macroscopic loops
 before finding $\gamma (i)$  in the exploration mechanism, which will contradict the local finiteness of the CLE.

Recall that, at this moment, we know that for any given point $z \in \HH$, the measure $\mu^{z}$ is the limit when  $\eps \to 0$ of $u(\eps)^{-1}$ times the law of $\gamma (z)$ restricted to the event that $\gamma (z)$ intersects the $\eps$-neighborhood of the origin. Furthermore, for any given point $z \in \HH$, the distance between $z$ and $\gamma$ is $\mu$-almost always positive.  This implies that for any given finite family of points $z_1, \ldots, z_n$ in $\HH$, the measure $\mu$ restricted to the set
$\A= \A (z_1, \ldots, z_n)$ of loops that surrounds at least one of these points is the limit when $\eps \to 0$ of $u(\eps)^{-1}$ times the (sum of the) laws of loops in $\A$ that intersect the $\eps$-neighborhood of the origin.

Suppose now that $\mu ( \{ \gamma \  : \   R (\gamma) \ge 1 ) \} = \infty$.
This implies clearly that for each $M >1$, one can find a finite set of points $z_1, \ldots, z_n$ at distance greater than one of the origin such that
 $\mu (\A (z_1, \ldots, z_n)) >M$.
Hence, it follows that when $\eps$ is small enough, at each exploration step, the probability to discover a loop in $\A$ is at least $M$ times bigger than $u(\eps)$.
 The number of
exploration steps before $N$ at which this happens is therefore geometric with a mean at least equal to $M$.
It follows readily that with probability at least $c_0$ (for some universal positive $c_0$ that does not depend on $\eps$), this happens at least $M/4$ times.

Note that the harmonic measure in $\HH$ at $i$ of a loop of radius at least $1$ that intersects also $C_{1/2}$ is bounded from below by some universal positive
constant $c_1$. If at some step $j \le N$, the radius of $\HH \setminus (\varphi_j^\eps)^{-1} ( \HH)$ is greater than $1$, then it means that there is a loop
in the CLE in $\HH$ that one explores at the $j$-th step that has a radius at least $1$ and that intersects $C_{1/2}$. Its preimage under
$\varphi_{j-1}^\eps \circ \ldots \circ \varphi_1^\eps$ (which is a loop of the original CLE that one is discovering) has therefore
 also has a harmonic measure (in $\HH$ and at $i$) that is bounded from below by $c_1$.
Hence, we conclude that with probability at least $c_0$, the original CLE has at least $M/4$ different loops such that their harmonic measure seen from $i$ in $\HH$ is bounded from below by some universal constant $c_1$.

This statement holds for all $M$, so that with probability at least $c_0$, there are infinitely many loops in the CLE
such that their harmonic measure seen from $i$ in $\HH$ is bounded from below by $c_1$.

  On the other hand, we know that $\gamma (i)$ is almost surely at positive distance from $i$, and this implies that for some positive $\alpha$, the probability that some loop in the CLE is a distance less than $\alpha$ of $i$ is smaller than $c_0/2$. Hence, with probability at least $c_0/2$, the CLE contains infinitely many loops that are all at distance at least $\alpha$ from $i$ and all have harmonic measure at least $c_1$. A similar statement is therefore true for the CLE in the unit disc if one maps $i$ onto the origin. It is then easy to check that the previous statement contradicts the local finiteness (because the diameter of the conformal image of all these loops is bounded from below).
    Hence, the $\mu$-mass of the set of pinned loops that reach the unit circle is indeed finite.
\end {proof}

Let us now list various consequences of Lemma \ref {LR1}:

\begin {itemize}
 \item
For all $r >0$, let us define $ \A_r := \{ \gamma \ : \ R (\gamma) > r \}$.
Because of scaling, we know that for all $r >0$,
$ \mu (\A_r ) =  r^{-\beta} \mu ( \A_1)$.
Clearly, this cannot be a constant finite function of $r$; this implies that $\beta > 0$.
Also, we get that for each fixed $r$, $\mu ( R( \gamma) = r ) =0$.

\item
We can now define the function $v(\eps)$ as the probability that in the CLE, there exists a loop that intersects $C_\eps$ and $C_1$. We know that
\begin {equation}
\label {vandu}
\lim_{\eps \to 0+} \frac {v(\eps)}{u(\eps)} = \mu ( R( \gamma) \ge 1 )
\end {equation}
(because $\mu ( R(\gamma) = 1) =0$).
Hence, it follows that for any $\delta <1 $, $v (\delta \eps)/ v (\eps) \to \delta^\beta$ as $\eps \to 0$. This will be useful later on.

\item
We can rephrase in a slightly more general way our description of $\mu$ in terms of limits of CLE loops. Note that
$ \A_r
= \cup_{n} \A (z_1, \ldots, z_n)$, where
 $(z_n, n \ge 1)$ is some given dense sequence on $\{ z \ : \ |z|= r \}$.

For each simple loop configuration $\Gamma$, and for each $\eps$, let us define $\tilde \gamma (\eps)$ to be the loop in the configuration $\Gamma$ that intersects the disc of radius $\eps$ and with largest radius (in case there are ties, take any deterministic definition to choose one).
We know from (\ref {vandu}) that the probability that $R ( \tilde \gamma (\eps)) > r $ decays like $u(\eps) \times r^{-\beta} \mu ( \A_1)$ as $\eps \to 0$ (for each fixed $r>0$).

Furthermore, we note that
the probability that there exist two different loops of radius greater than $r$ that intersect the circle $C_\eps$ decays like $o( u(\eps))$ as $\eps \to 0$.
Indeed, otherwise, for some sequence $\eps_n \to 0$, the probability that in our exploration procedure, two macroscopic loops are discovered simultaneously remains
positive and bounded from below, which is easily shown to contradict the fact that almost surely, any two loops in our CLE are at positive distance from each other.

Hence, for each $n$, the measure $\mu$ restricted to $\A (z_1, \ldots, z_n)$ is
the limit when $\eps \to 0$ of $u(\eps)^{-1}$ times the law of $\tilde \gamma (\eps)$, restricted to the event that it surrounds at least one of the points $z_1, \ldots, z_n$.

We conclude that the measure $\mu$ restricted to $\A_r$ can
be viewed as the weak limit when $\eps \to 0$ of $u(\eps)^{-1}$ times the law of $\tilde \gamma (\eps)$ restricted to $\A_r$.
In other words, if we consider the set of pinned loops of strictly positive size (i.e., the loop of zero length is not in this set) endowed with the Hausdorff metric, we can say that $\mu$ is the vague limit of $u(\eps)^{-1}$ times the law of $\tilde \gamma (\eps)$.
\end {itemize}

Let us finally state another consequence of this result, that will turn out to be useful in the loop-soup construction part of the paper.
Consider a CLE in $\HH$ and let us consider the set of loops that intersect the unit circle. Define ${\cal R}$ to be the radius of the smallest disc centered at the origin that contains all these loops. Note that scaling shows that $P ( {\cal R} > x ) = v (1/x)$ for $x \ge 1$.
\begin {corollary}
\label{boundedness}
 If $\beta > 1$, then $E ( {\cal R}^{(1+\beta)/2}) < \infty$.
\end {corollary}
\begin {proof}
Just note that
\begin {eqnarray*}
E ( {\cal R}^{(1+ \beta)/2}) & = &  \int_0^\infty dr P ( {\cal R}^{(1+\beta)/2} \ge r ) \\
& \le&  1 + \int_1^\infty dr \ v( r^{- 2/(1+ \beta)} ) \\
&\le
& 1 + \int_1^\infty \frac {dr}{r^{2\beta/ (1+ \beta) +o(1)}}
\  < \ \infty
\end {eqnarray*}
because $2 \beta > 1 +\beta$.
\end {proof}

With the $8/ \kappa = 1+ \beta$ identification that we will derive later, $\beta >1 $ corresponds to $\kappa < 4$ and then $(1+\beta)/2 = 4/ \kappa$.

\medbreak

The next proposition corresponds to the fact that $\alpha$-stable subordinators exist only for $\alpha \in (0,1)$.

\begin {proposition}
The scaling exponent $\beta$ described above lies in $(0, 2)$.
\end {proposition}

\begin {proof}
We use the radial exploration mechanism again.
Let us now assume that $\beta \ge 2$ and focus on the  contribution of the ``small'' loops that one discovers before actually discovering $\gamma (i)$.

Let $\A_1, \ldots, \A_k$ be $k$ fixed disjoint (measurable) sets of loops that do not surround $i$, with $a_l = \mu ( \A_l) < \infty$. We suppose that all $\A_l$'s for $l \le k$  belong to the algebra of events generated by the events of the type $ \{ \gamma \ : \ \gamma \hbox { surrounds } z \}$.
We let $N_\eps (\A_l)$ denote the number of loops in $\A_l$ that one has discovered in this way before one actually discovers the loop that surrounds $i$.
Our previous results show that when $\eps \to 0$, the joint law of $N_\eps (\A_1), \ldots , N_\eps (\A_k)$ converges to that of the $N_1, \ldots , N_k$ that we described
at the end of the previous subsection.

For each integer $j$, we define $\A_j$ to be  the set of  loops that surround $2^{-j} i$ but that do not surround any point  $2^{-k} i$ for $k < j$.
The scaling property of $\mu$  implies  that for all $j$, $\mu (\A_j) = 2^{j\beta} \mu (\A_0) $ and it is easy to check that $\mu (\A_0) > 0$ (because the loops that surround $i$ have a finite radius so that $\sum_{j \le 0} \mu (\A_j) \ge 1 > 0$). Hence, for all positive $j$,
$$ \mu (\A_j ) \ge 4^j \mu (\A_0).$$
Recall that if $\psi$ is a conformal map in ${\cal S}$
from a simply connected subset $H$ of $\HH$ that does contain $i$ but not $2^{-j_0} i $ for some  $j_0 \ge 1$, then
$$a( \psi) \ge C' 4^{-j_0}.$$
Furthermore, for each fixed $j_0$,  when $\eps$ is small enough, we can compare the number of loops in $\A_1, \ldots , \A_{j_0}$ that have been discovered before $\gamma (i)$ via the chordal exploration mechanism with i.i.d.\ Poisson random variables $N_1, \ldots , N_{j_0}$.
Note also that when one composes conformal maps in ${\cal S}$, the derivatives at $i$ get multiplied and the $a$'s therefore add up.

Hence, it follows immediately that if $\beta \ge 2$, then  for each $j_0$, with a probability that is bounded from below independently of $j_0$, when $\eps$ is small enough,
$$
a (\varphi_1 \circ \cdots \circ \varphi_N) \ge  \frac {C'}{2}  \sum_{j=1}^{j_0} N_{j} 4^{-j}
.$$
Hence, we conclude that there exists $c_2 >0$ such that for each $M > 0$, if one chooses $\eps$ small enough, the probability that
$a  (\varphi_1 \circ \cdots \circ \varphi_N) \ge M$ is at least $c_2$.
But this contradicts the fact that $\gamma (i)$ is at positive distance from $i$ (for instance using Koebe's $1/4$ Theorem).
Hence, we conclude that $\beta$ is indeed smaller than $2$.
\end {proof}

\section {The two-point pinned probability measure}

\subsection {Restriction property of the pinned measure}

We now investigate what sort of ``restriction-type'' property the pinned measure $\mu$ inherits from the CLE. Note that $\mu$ is an infinite measure on single pinned loops (rather than
a probability measure on loop configurations), so the statement will necessarily be a bit different from the restriction property of CLE.

Suppose now that $A$ is a closed bounded set such that $\HH\setminus A$ is simply connected, and that $d (0, A) > 0$. Our goal is to find an alternative description of the infinite measure $\mu$ restricted to the set of loops $\gamma$ that do not intersect $A$.

Suppose that a deterministic pinned (at zero) loop $\gamma$ (in the upper half-plane) is given.
Sample a CLE $\Gamma^\#$ in the upper half-plane.  This defines a random $H^\#$ which is the connected component that has the origin on its boundary of the set obtained by removing from $\HH \setminus A$ all loops of $\Gamma^\#$ that intersect $A$.
 Then, we define a conformal map $\psi^\#$ from $\HH$ onto $H^\#$ such that
$$ \psi^\# (0) = 0 \hbox { and } (\psi^\#)' (0)= 1.$$
In order to fix $\psi^\#$, another normalization is needed. We can for instance take $\psi^\# (\infty) = \infty$, but all of what follows would still hold if one
replaced $\psi^\#$ by the map $G^\#$ such that $G^\# (z)=  z + o(z^2)$ as $z \to 0$ (we will in fact also use this map in the coming sections).
Finally, we define $\gamma^\# =  \psi^\# (\gamma)$. Clearly, this is a pinned loop that stays in $ H^\#$ and therefore avoids $A$ almost surely.

Suppose now that we use the product measure $\mu \otimes P$ on pairs $(\gamma, \Gamma^\#)$ (where $P$ is the law of the same CLE that was used to define the pinned measure $\mu$).
For each pair $(\gamma, \Gamma^\#)$ we define the loop $\gamma^\# = \gamma^\# (\gamma, \Gamma^\#) = \psi^\# (\gamma)$ as before,
and we define $\mu_A$ to be the image measure of $\mu \otimes P$ via this map.  This $\mu_A$ is an infinite measure on the set of loops that do not intersect $A$.

\begin {figure}[htbp]
\begin {center}
\includegraphics [width=5in]{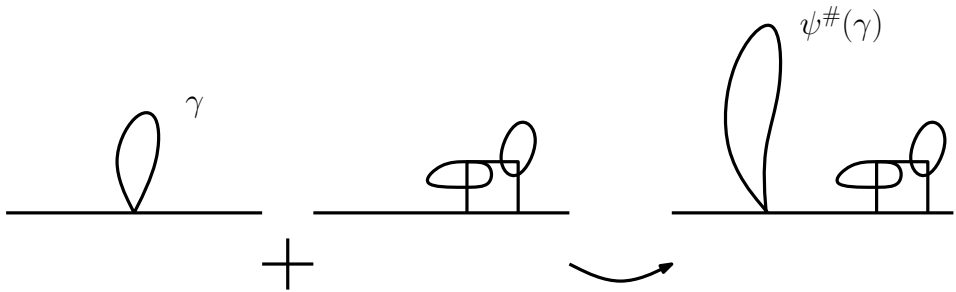}
\caption {Construction of $\mu_A$ (sketch)}
\end {center}
\end {figure}
We are now ready to state the pinned measure's restriction property:

\begin {proposition}[Restriction property of $\mu$]
\label {muA}
 The measure $\mu$ restricted to the set $\{\gamma \cap A = \emptyset\}$ is equal to $\mu_A$.
\end {proposition}

Recall that we constructed the pinned measure by ``exploring'' a CLE in $\U$ until we discover $\gamma_0$.
If we started with a pinned loop --- together with a CLE in the complement of the pinned loop --- we might try to ``explore'' this configuration
until we hit the pinned loop. This would be a way of constructing a natural measure on loops pinned at two points.  We will essentially carry out such a procedure later on, and the above proposition will be relevant.

\begin {proof}
Let us now consider the set $A$ and a CLE $\Gamma^\#$, and define $H^\#$ and the map $\psi^\#$ as before. For each $\eps$, the event that no loop in $\Gamma^\#$
 intersects both $C_\eps$ and $A$ is identical to the event that $C_\eps \subset H^\#$.
Let us now condition on the set $H^\#$ (for a configuration where $C_\eps \subset H^\#$); the restriction property of the CLE tells us that the conditional law of the CLE-loops that stay in $H^\#$ is exactly that of an ``independent CLE'' defined in this random set $H^\#$.
Hence, we get an identity between the following two measures:
\begin {itemize}
 \item The law of $\tilde \gamma (\eps)$  in the CLE (the loop with largest radius that intersects $C_\eps$), restricted to the event that no loop in the CLE intersects both $C_\eps$ and $A$.
\item
Sample first a CLE $\Gamma^\#$, define $H^\#$, restrict ourselves to the event where $C_\eps \subset H^\#$, define $\psi^\#$, consider an independent CLE $\tilde \Gamma$ in the upper half-plane and its image $\psi^\# ( \tilde \Gamma)$, and look at the law of the loop with largest radius in this family that intersects $C_\eps$.
\end {itemize}
Note that the total mass of these two measures is the probability that no loop in the CLE intersects both $C_\eps$ and $A$.

Now, we consider the vague limits when $\eps \to 0$ of $1/u(\eps)$ times these two measures. It follows readily from our previous considerations that:
\begin {itemize}
\item
For the first construction, the limit is just $\mu$ restricted to the set of pinned loops that do not hit $A$.
\item
For the second construction, the limit is just $\mu_A$ (recall that $\psi^\#$ has been chosen in such a way that  $(\psi^\#)' (0) =1$, so that when $\eps$ is small $ \psi^\# (C_\eps)$ is very close to $C_\eps$).
\end {itemize}
\end {proof}

In fact, it will be useful to upgrade the previous result to the space of ``pinned configurations'': We say that $(\bar \gamma, \bar \Gamma)$ is a {\em pinned configuration}
 if $\bar \gamma$ is a pinned loop in the upper half-plane and if $\bar \Gamma$ is a loop-configuration in the unbounded connected component of $\HH \setminus \bar \gamma$.
\begin {figure}[htbp]
\begin {center}
\includegraphics [width=3.5in]{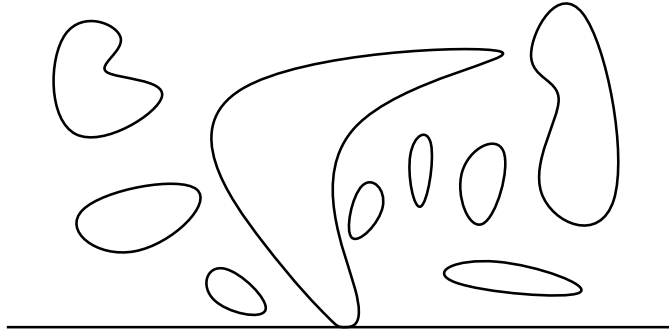}
\caption {A pinned configuration (sketch)}
\end {center}
\end {figure}

Let us define a first natural measure on the space of pinned configurations. Suppose that one is given a pinned loop $\gamma$ and a loop-configuration $\Gamma$ (that is not necessarily disjoint from $\gamma$).
Define $H_\gamma$ the unbounded connected component of the complement of $\gamma$, and $\phi_\gamma$ the conformal map from $\HH$ onto $H_\gamma$ that is normalized at infinity ($\phi_\gamma (z) = z + o(1)$). Then $(\bar \gamma, \bar \Gamma)= ( \gamma, \phi_\gamma ( \Gamma))$ is clearly a pinned loop configuration.
We now let $\bar \mu$ denote the image of the product measure $\mu\otimes P$ under this map $(\gamma, \Gamma) \mapsto (\bar \gamma, \bar \Gamma)$. We call it the pinned CLE configuration measure.
Clearly, the marginal measure of $\bar \gamma$ (under $\bar \mu$) is $\mu$ (because $P$ is a probability measure and $\bar \gamma = \gamma$).

\begin {figure}[htbp]
\begin {center}
\includegraphics [width=5in]{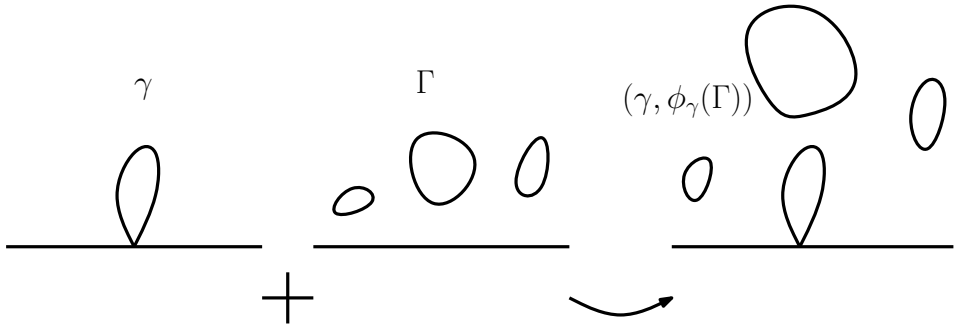}
\caption {Construction of $(\bar \gamma, \bar \Gamma)$ (sketch)}
\end {center}
\end {figure}

Suppose now that $A$ is a set as before.
If $(\bar \gamma , \bar \Gamma )$ is a pinned configuration, then we define:
\begin {itemize}
 \item $\tilde \Gamma_A$ to be the set of  loops of $\bar \Gamma$ that intersect $A$.
\item $\bar H_A$ to be the unbounded connected component of $\HH \setminus (\tilde \Gamma_A \cup A)$
\item $\bar \Gamma_A$ to be the set of loops of $\bar \Gamma$ that stay in $\bar H_A$.
\end {itemize}
 Hence, when $\bar \gamma \cap  A = \emptyset$, we define a triplet $(\bar \gamma, \bar \Gamma_A, \tilde \Gamma_A)$. Let $\bar \mu_A$ denote the image of $\bar \mu$ (restricted to the set  $\bar \gamma \cap  A = \emptyset$) under this transformation.

We now construct another measure that will turn out to be identical to
 $\bar \mu_A$: Start on the one hand with a pinned measure configuration $(\bar \gamma, \bar \Gamma)$ and on the other hand with a loop-configuration in $\HH$ that we denote by $\Gamma^\#$. We define $\psi^\#$ as before (using $A$ and $\Gamma^\#$). We also let $\Gamma_A^\#$ denote the loops in $\Gamma^\#$ that intersect $A$.
Then we consider the triplet $(\psi^\# ( \bar \gamma), \psi^\# ( \bar \Gamma), \Gamma_A^\# )$. This triplet is a function of $((\bar \gamma, \bar \Gamma), \Gamma^\# )$.
We now define $\mu_A^\#$  to be the image of the product measure $\bar \mu \otimes P$ under this mapping.

\begin {figure}[htbp]
\begin {center}
\includegraphics [width=5in]{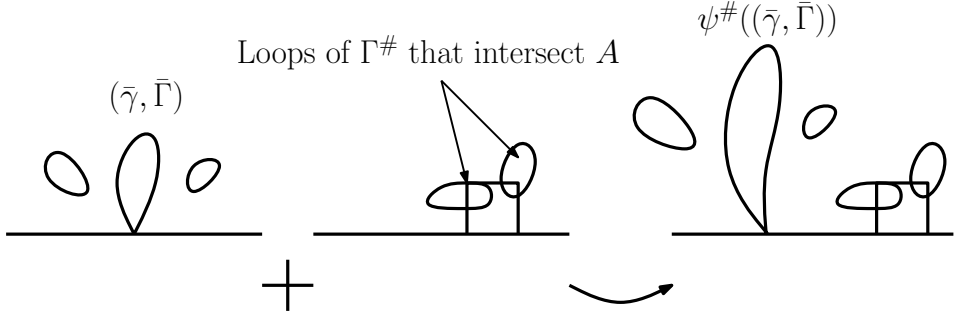}
\caption {The measure $\mu_A^\#$ (sketch)}
\end {center}
\end {figure}

\begin {proposition}[Restriction property for pinned configurations]
\label {P72}
 The two measures $\bar \mu_A$ and $\mu_A^\#$ are identical.
\end {proposition}

If we consider the marginal measures on the pinned loops of $\bar \mu_A$ and $\mu_A^\#$, we recover
Proposition \ref {muA}.
The proof is  basically identical to that of Proposition \ref {muA}: One just needs also
to keep track of the remaining loops, and this in done on the one hand thanks to the CLE's restriction property, and on the other hand (in the limiting procedure) thanks to the fact that loops are disjoint and at positive distance from the origin. In order to control the $\eps \to 0$ limiting procedure applied to loop configurations, one can first derive the result for the law of finitely many loops in the loop configuration (for instance those that surround some given point).

\subsection {Chordal explorations}
As in the case of the CLE, we are going to progressively explore a pinned configuration,
cutting out recursively small (images of) semi-circles, and trying to make use of the
pinned measure's restriction property in order to define a ``two-point pinned measure''. However,  some caution
will be needed in handling ideas involving independence because
$\bar \mu$ is not a probability measure.

Recall that $\bar \mu$ is a measure on configurations in a simply connected domain with one special
marked boundary point (i.e., the origin, where the pinned loop touches the boundary of the domain).
It will therefore be natural to work with chordal $\eps$-admissible explorations instead of radial ones.
This will be rather similar to the radial case, but it is nonetheless useful to describe it in some detail:

\medbreak
It is convenient to first consider the domain to be the upper half-plane and to take $\infty$ as the marked boundary point instead of the origin.
Suppose  that $\Gamma$ is a loop configuration (with no pinned loop) in the upper half-plane $\HH$ and choose some
 bounded closed simply connected set $A \subset \overline \HH$ such that \begin{enumerate}
  \item $\HH \setminus A$ is simply connected,
  \item $A$ is the closure of the interior of $A$,
  \item the interior of $A$ is connected, and
  \item the length of $\partial A \cap \partial \HH$ is positive.
 \end{enumerate}
Suppose furthermore that all $\gamma_j$'s in $\Gamma$ that intersect $\partial A$ also intersect the interior of $A$.

Here, when $x \in \R$, $D(x,\eps)$ will denote the set of points in $\HH$ that are at distance less than $\eps$ from $x$.
We choose some $x_1$ on the real line, such that $D(x_1, \eps) \subset A$. Then, we define the set $H_1$ to be the unbounded connected component of the set obtained when removing from $\HH \setminus D(x_1, \eps)$ all loops of $\Gamma$ that intersect $D(x_1, \eps)$. We also define the conformal map $g_1$ from $H_1$ onto $\HH$ that is normalized at infinity ($g_1 (z) = z + o(1)$ as $z \to \infty$), and we let $A_2 = g_1 ( H_1 \cap A)$.

Then, we proceed inductively: For each $n \ge 1$, we choose (when it is possible) $x_{n+1}$ on the real line such that
$D(x_n, \eps) \subset  A_n$. Then
we consider the set $H_{n+1}$ to be the unbounded connected component of the set obtained when removing from $\HH \setminus D(x_n, \eps)$ all loops of $g_n \circ \ldots \circ g_1 (\Gamma)$ that intersect $D(x_n, \eps)$. We also define the conformal map $g_{n+1}$ from $H_{n+1}$ onto $\HH$ that is normalized at infinity ($g_{n+1} (z) = z + o(1)$ as $z \to \infty$), and we let $A_{n+1} = g_{n+1} ( H_{n+1} \cap A_n)$.

This procedure necessarily has to stop at some step $N$, i.e., at this step $N$, it is not possible to find a point $x$ such that
$D (x, \eps ) \subset H_N$. This is just because of additivity of the half-plane capacity under composition of the conformal maps (and because the half-plane capacity of $A$ is finite). We say that such an exploration is an {\em $\eps$-admissible (chordal) exploration of $(\Gamma, A)$ rooted at infinity}. The results of Subsection \ref {explodet} and their proofs can be immediately adapted to the present setting. It is for instance easy to check that for any given $\Gamma$, for any given loop $\gamma_j$ in $\Gamma$ that intersects such a given $A$, there exists a positive $\eps_0  = \eps_0 ( \Gamma, A, \gamma_j)$ such that any
 $\eps$-admissible chordal exploration of $(\Gamma, A)$ necessarily discovers the loop $\gamma_j$ as soon as
$\eps \le \eps_0$.
Hence, after the last step $N$ of this $\eps$-exploration of $(\Gamma, A)$, for all positive $\alpha$, when $\eps$ is chosen to be sufficiently small,
one has necessarily discovered all loops of diameter greater than $\alpha$ of $\Gamma$ that intersect $A$.

\medbreak

Let us now suppose that $\Gamma$ is a loop configuration and $A$ is a set satisfying the same conditions as before, except that we also assume that it is at positive distance of the origin.
We can now define chordal $\eps$-explorations of $(\Gamma, A)$ that are {\em rooted at $0$}: These are just the image under $z \mapsto -1/z$ of the previous explorations of $(-1/ \Gamma, -1/A )$, where $-1/\Gamma$ denotes the loop configuration obtained when one considers the loops $(-1/\gamma_j)$.
The difference between the exploration rooted at the origin and the exploration rooted at infinity in fact only lies in what we call $D(x, \eps)$, and how one renormalizes the conformal maps $g_n$ at each step. In the exploration rooted at the origin, $D(x,\eps)$ would have to be replaced by the conformal image of $D(x, \eps)$ under the conformal mapping from $\HH$ onto $\HH$ that fixes $x$, has derivative $1$ at $x$ and maps $\infty$ onto $0$, and the conformal map $g_n^{-1}$ would have to be replaced by the conformal map
$G_n$ from $\HH$ onto $H_n$ such that $G_n(z) = z + o (z^2)$ when $z \to 0$.

Let us describe more precisely an iteration step in this case:
Let us define, for each $x \in \R \setminus \{ 0 \}$ and each small $\eps$, $C_\eps (x)$ to be the image of $C_\eps$ under the Moebius transformation
$I_x : \ z \mapsto -x / (xz-1)$ of $\HH$ that maps $0$ onto $x$ and $\infty$ onto $0$. Note that for all $x$ and $y$ in $\R \setminus \{ 0 \}$,
 $I_{x} \circ I_{y}^{-1} (z) \sim z$ when $z \to 0$ and
that $I_x \circ I_y^{-1} ( C_\eps (y)) = C_\eps (x)$.
When $\eps$ is very small, $C_\eps (x)$  is very close to the semi-circle of radius $x^2 \eps$ around $x$ (i.e., it is squeezed between to semi-circles of radii close to $x^2 \eps$) because $I_x'(0) = x^2$.

Suppose that $H_n$, $G_n$ are already defined. The choice of $x_n$ is then said to be $\eps$-admissible if $C_\eps (x_n) \subset G_n^{-1} ( H_n \cap A)$.
We then define $H_{n+1}$ to be the unbounded connected component the domain by removing from $H_n \setminus G_n ( C_\eps (x_n) )$ all the loops that intersect $G_n (C_\eps (x_n))$.

Clearly, after the last step $N$ of this exploration rooted at the origin (when no $\eps$-admissible point can be found), we again have a set $H_N$ that is in some sense close to $H$ (which is the unbounded connected component of the domain obtained when removing from $\HH \setminus A$ all loops that intersect $A$).
One way to make this more precise is to use the Carath\'eodory topology ``seen from the origin''. Even if the origin is a boundary point of the simply connected domains $H_N$ and $H$, we can symmetrize these domains (by considering their union with their symmetric sets with respect to the real axis), and therefore view the conformal maps
$G_N$ and $G$ as conformal maps normalized at the origin, which is now an inner point of the domain. We will implicitely use this topology for domains in the upper half-plane ``viewed from the origin''. Then, just as in Subsection \ref {explodet}, we get that (for each given $\Gamma$ and $A$)  $H_N$ converges to $H$, when $\eps \to 0$, uniformly with respect to all possible $\eps$-admissible explorations.

\subsection {Definition of the two-point pinned measure}

 We know that the $\mu$-mass of the set of pinned loops that intersect the segment $[1, 1+i]$ is positive and finite. By scaling, we can choose $a$ in such a way that the $\mu$-mass of the set of pinned loops that intersect $[a, a(1+i)]$ is equal to $1$. On this set of configurations, $\mu$ can therefore be viewed as a probability measure. Similarly, the pinned configuration measure $\bar \mu$ restricted to the set
$$ \A = \A ([a, a + ia]) = \{ ( \bar \gamma , \bar \Gamma) \ : \  \bar \gamma \cap [a, a+ia] \not= \emptyset \}$$
 is a probability measure that we will denote by $\tilde P_{[a, a +ia]}$.

Let us now consider a pinned configuration $(\bar \gamma, \bar \Gamma )$ in $\A$.
Define $u= \min \{ y \ : \ a+ iy \in \bar \gamma \}$ so that $a+iu$ is the lowest point of $\bar \gamma \cap [a, a+ia]$.
Let $\psi$ denote the conformal map from
 the unbounded connected component $H_\psi$ of the set obtained by removing from $\HH \setminus [0,a+iu]$ all the loops of $\bar \Gamma$ that intersect this segment back onto $\HH$, normalized by $\psi (0)=0$, $\psi' (0)= 1$ and $\psi (a+iu)= 1$.
We then define $\tilde P$ to be the distribution of $\tilde \gamma = \psi (\bar \gamma)$, i.e.\ the image of the probability measure $\tilde P_{[a,a+ia]}$ under the transformation
$(\bar \gamma, \bar\Gamma) \mapsto \psi (\bar \gamma)$.
\begin {figure}[htbp]
\begin {center}
\includegraphics [width=6in]{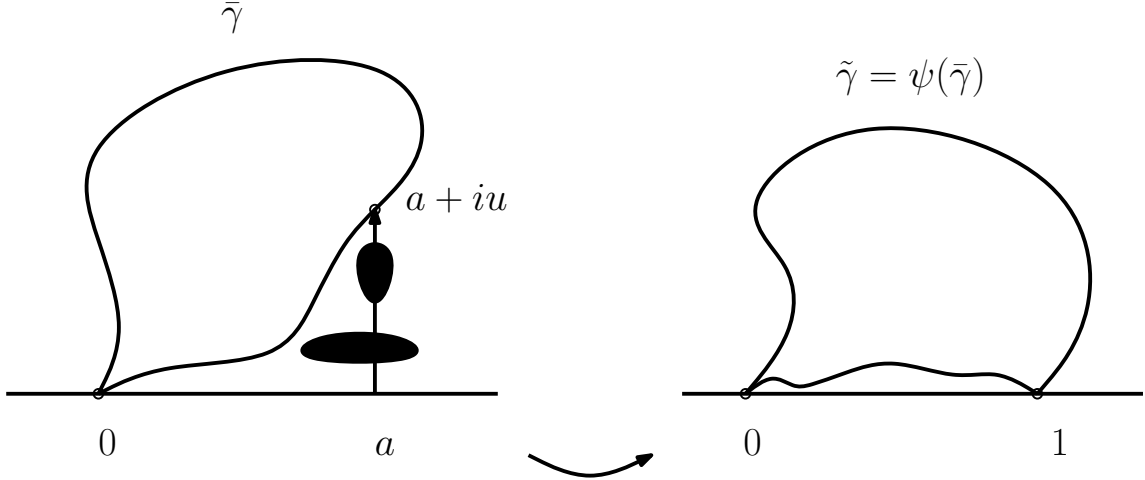}
\caption {The construction of of the two-point pinned loop (sketch)}
\end {center}
\end {figure}
A by-product of the coming arguments will be that in fact:
\begin {proposition}
\label {P81bis}
Under $\tilde P_{[a, a +ia]}$, the loop $\tilde \gamma$ is independent of $\psi$.
\end {proposition}

Define on the other hand the probability measure $\tilde P_\eps$ to be the measure $\mu$ restricted to the set of pinned loops that intersect $C_\eps (1)$, and
renormalized in order to be a probability measure. The main first goal of this subsection is to derive the following proposition (which we will later prove together with Proposition \ref{P81bis}):

\begin {proposition}[Definition of the two-point pinned measure $\tilde P$]
\label {P81}
As $\eps \to 0$, $\tilde P_\eps$ converges to $\tilde P$.
\end {proposition}


As we have already indicated, the basic idea of the proof will be similar in spirit with the construction of the pinned measure itself: We will use discrete explorations
of pinned configurations, and use the restriction property of the pinned measure in order to deduce some independence property between the already discovered loops and the yet-to-be-discovered ones along the exploration; this will enable us to conclude.

\medbreak

Suppose now that $A$ is a set that is at positive distance from  the origin, and that satisfies the conditions 1-4 described in the previous subsection.
Clearly, if we define
$$ \bar \A (A) = \{ ( \bar \gamma, \bar \Gamma) \ : \ \bar \gamma \cap A \not= \emptyset \},$$
then $\bar \mu (\bar \A (A))$ is finite and positive. We can therefore define $\tilde P_{A}$ to be $\bar \mu$ restricted to this set, and renormalized in order to be a probability measure.

Suppose now that $H\subset \HH$ is a simply connected domain with $d ( 0, \HH \setminus H ) > 0$. Define $\phi$ to be the conformal map from $H$ onto $\HH$ such that
$\phi (z) = z + o(z^2)$ when $z \to  0$, and suppose that $C \subset H$ is a set such that $A:= \phi (C)$ satisfies the same conditions as before.
We then define $\tilde P_{C}^H$ to be the image of $\tilde P_{A}$ under $\phi^{-1}$.  Basically, $\tilde P_{C}^H$ is the ``law'' of a pinned configuration in $H$ rooted at the origin and ``conditioned'' on the event that the pinned loop intersects $C$.

Suppose now that $B$ is another set at positive distance of the origin that satisfies the four conditions described in the previous subsection and such that $ A \subset B$.
We are now going to define, for each such $B$ and $A$, the measure $\bar \mu_{B , A}$ to be the measure $\bar \mu$ on pinned
configurations $(\bar \gamma, \bar \Gamma)$, restricted to the set $\bar \A ( B, A) = \bar \A (B) \setminus \bar \A (A)$ of configurations such that $\bar \gamma$ intersects $B$ but not $A$.
When $\bar \mu ( \bar \A (B)) > \bar \mu ( \bar \A (A))$, we then define
$\tilde P_{B, A}$ to be $\bar \mu_{B, A}$ normalized to be a probability measure.

Finally, when we are given $\bar \Gamma$ and $A$, we can define the set $\bar H_{A}$ as before (i.e.,  the unbounded connected component of the set obtained by removing from $\HH \setminus A$ all the loops of $\bar \Gamma$ that intersect $A$). Then, it follows immediately from the pinned measure's restriction property (Proposition \ref {P72})
that:

\begin {corollary}
\label {C555}
 Suppose that $(\bar \gamma, \bar \Gamma)$ is sampled according to $\tilde P_{B,A}$. Then the conditional law of $\bar \gamma$ given $\tilde \Gamma_{A}$ (i.e., the loops of $\bar \Gamma$ that intersect $A$) is $\tilde P^{H}_{B \cap H}$ where $H= \bar H_{A}$.
\end {corollary}

Suppose now that we are given a configuration $(\bar \gamma, \bar \Gamma)$ that is sampled according to $\tilde P_A$. For each given $\eps$, we can perform a ``Markovian'' $\eps$-admissible exploration as before (Markovian can simply mean here that we have chosen some deterministic procedure to choose each $x_n$). When $\bar \gamma$ intersects $A$,  this exploration procedure can discover $\bar \gamma$ at some (random) step that we call $\bar N+1$, and we know that when $\eps \to 0$, the probability that $\bar N $ exists tends to $1$.

Corollary \ref {C555} shows that conditionally on $\bar N=n$, on $H_n$ and on $x_n$, the law of $\bar \gamma$ is $\tilde P_{G_n (C_\eps (x_n))}^{H_n}$.
This implies in particular that the conditional law of $I_1 \circ I_{x_{\bar N}}^{-1} \circ G_{\bar N}^{-1} ( \bar \gamma )$ is $\tilde P_\eps$ (recall that $\psi_{\bar N}:= I_1 \circ I_{x_{\bar N}}^{-1} \circ G_{\bar N}^{-1}$ is just the conformal map from $H_{\bar N}$ onto $\HH$ that maps the origin onto itself, has derivative at the origin equal to $1$, and that maps $G_{\bar N} (x_{\bar N})$ onto $1$).
Hence (on the event ${\bar N} < \infty$), $\psi_{\bar N} ( \bar \gamma )$ is in fact independent of ${\bar N}$ and of $H_{\bar N}$.

It is worth stressing at this point that the law of $\bar N$ is not geometric as in the CLE case, and that the iteration steps are not i.i.d.\ anymore, but this
will not prevent us from now proving Propositions \ref {P81} and \ref {P81bis}:

\begin {proof}
The arguments are again close in spirit to the definition of the pinned measure itself.
Suppose that $K$ is a large integer, that $\delta= 1/K$, and consider the rectangles
$A_{\delta,k} = [a-\delta/2, a+ \delta/2] \times [0, ak \delta]$ for $k=1, \ldots, K$.
Note that because of scale-invariance and translation invariance, for each given $\delta$, up to a set of configurations of zero measure,
all the loops of $(\bar \gamma, \bar \Gamma)$ that intersect some $A_{\delta, k}$ also intersect its interior  (when $\delta$ is fixed).
Furthermore, the probability measure $\tilde P_{A_{\delta, K}}$ converges to $\tilde P_{[a, a (1+i)]}$.

For each fixed $\delta = 1/ K$, we can start to explore each $(\bar \Gamma \cup \{ \bar \gamma\}, A_{\delta, 1})$ by a Markovian chordal $\eps$-admissible exploration rooted at $0$, then continue exploring $(\bar \Gamma \cup \{ \bar \gamma \} , A_{\delta, 2})$ and so on until we either complete a chordal $\eps$-admissible exploration of $A_{\delta , K} = [a - \delta , a +\delta] \times [0, a]$, or we have discovered $\bar \gamma$. When $(\bar \gamma, \bar \Gamma)$ is sampled according to $\tilde P$, the probability to discover $\bar \gamma$ is going to $1$ as $\eps \to 0+$. Let us call (as before) $\bar N +1 $ the (random) step at which $\bar \gamma$ is discovered.

\begin {figure}[htbp]
\begin {center}
\includegraphics [width=3.5in]{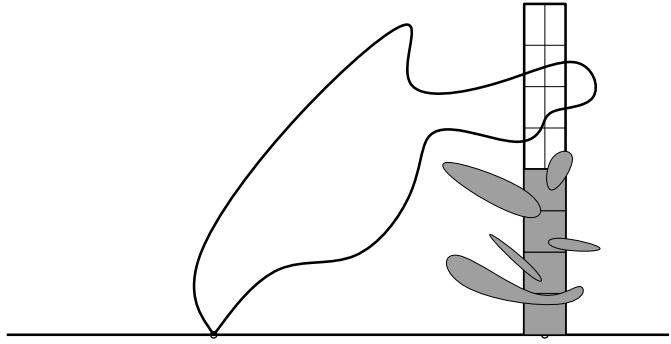}
\caption {Exploration of the thin rectangle (sketch)}
\end {center}
\end {figure}

Furthermore, for each given $\delta$, it is possible to choose $\eps$ small enough, so that the probability that the exploration processes misses no loop of $\bar \Gamma$ of diameter greater than $4\delta$ that intersects $[a, a+ ia]$ is as close to $1$ as we want.

The very same arguments as in the proof of
Proposition \ref {pwellchosen} then show that if $\eps (\delta)$ is chosen to be sufficiently small, then the set $H_{\bar N}$ converges to $H_\psi$ almost surely, and that
$\psi_{\bar N}$ converges almost surely to $\psi$.
We then conclude using the independence between $\psi_{\bar N}$ and $\psi_{\bar N} ( \bar \gamma)$. We safely leave the details to the reader.
\end {proof}

A consequence of Proposition \ref{P81} is that the probability measure $\tilde P$ is invariant under all Moebius transformations of $\HH$ onto itself that
have $0$ and $1$ as fixed points. This follows immediately from the conformal covariance property of $\mu$ and the description of $\tilde P$ as limits of $\tilde P_\eps$
given in the proposition.

\medbreak

In fact, our proof of Proposition \ref {P81} works also if we replace the straight segment $[a, a(1+i)]$ by any given other piecewise linear path starting at $1$ such that the $\mu$ mass of the set of pinned loops that it intersects is finite (we can then just multiply  $\mu$ by a constant to turn it into a probability measure on this set): Suppose that $\eta$ is some finite piecewise linear path that starts on the real line, has no double points, and such that all of its segments are horizontal or vertical. We parametrize it ``from the real axis to the tip''.

For each pinned configuration $(\bar \gamma, \bar \Gamma)$ such that $\bar \gamma \cap \eta \not= \emptyset$, we can define the first meeting time $T$ of $\eta$ with $\bar \gamma$, i.e.\ $T = \min \{ t \ : \ \eta (t) \in \bar \gamma \}$. Let us now consider $\tilde \Gamma_\eta$ to be the set of loops of $\bar \Gamma$ that intersect
$\eta [0,T]$, and let $\psi$ denote the conformal map as before (that maps the set obtained by removing from $\HH \setminus \eta [0,T]$ all the loops of $\bar \Gamma$ that intersect $\eta [0,T]$ back onto $\HH$, normalized by $\psi (0)=0$, $\psi' (0)= 1$ and $\psi (\eta(T))= 1$. Note that $\psi ( \bar \gamma)$
 is a pinned loop (the map $\psi$ is smooth near the origin). Note also that
$\bar \mu ( \bar \gamma \cap \eta \not= \emptyset )$ is positive and finite.
Then:
\begin {proposition}
If we consider the measure $\bar \mu$ on the event $\{ \bar \gamma \cap \eta \not= \emptyset \}$ and if we renormalize it to be a probability measure, then
under this probability measure, $\tilde \gamma = \psi (\bar \gamma)$ is distributed according to the two-point pinned distribution $\tilde P$. Furthermore, $\tilde \gamma$ is independent of the conformal map $\psi$.
\end {proposition}

\subsection {Independence property of two-point pinned loops}

We are now going to derive the key independence property of the two-point pinned measure.
Suppose that $\tilde \gamma$ is a two-point pinned loop in $\HH$ that touches the real line at $0$ and $1$. We can define two simple paths $\gamma^*$ and $\gamma_*$ from $0$ to $1$ in $\HH$ in such a way that their union is $\tilde \gamma$ and that $\gamma_*$ is ``below'' $ \gamma^*$.
We now define $\psi_*$ to be the conformal map from the unbounded connected component $H_*$ of $\HH \setminus  {\gamma}_*$ onto $\HH$ that fixes the three boundary points $0$, $1$ and $\infty$, and finally, we define $U(\tilde \gamma) = \psi_*  ( \gamma^*)$.
\begin {figure}[htbp]
\begin {center}
\includegraphics [width=5in]{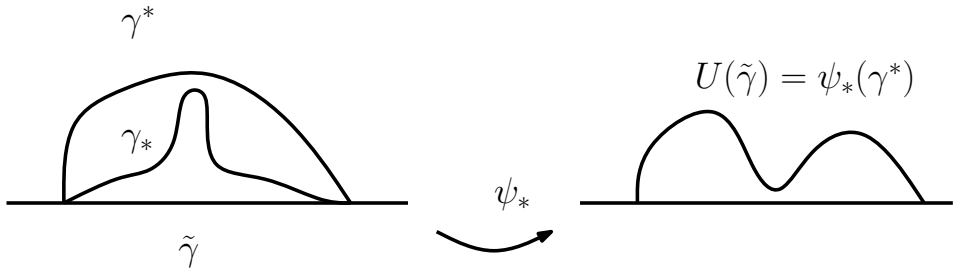}
\caption {Construction of $U (\tilde \gamma)$ (sketch)}
\end {center}
\end {figure}
\begin {proposition}
\label {Pkey}
If $\tilde \gamma$ is chosen according to the two-point pinned measure $\tilde P$, then
 $U(\tilde \gamma)$ and $\gamma_*$ are independent.
\end {proposition}

 Recall on the one hand that $\tilde P$ is the limit of $\tilde P_\eps$ when $\eps \to 0$.
Let us now consider a hull $A \subset \overline \HH$ that is attached to the interval $(0,1)$ that satisfies the conditions 1-4 as before.
Let $B = A \cup C_\eps (1)$ where $\eps$ is sufficiently small, so that $C_\eps (1) \cap A = \emptyset$.
Let us now describe the law of $\tilde \gamma$ conditioned to avoid $A$, trying to use Corollary \ref {C555}. Remark first that it is the limit when $\eps \to 0$
of the pinned measure ``renormalized in order to be a probability'' on the event $\A ( B, A)$ i.e.\ on the event where $\gamma$ intersects $C_\eps (1)$ but not $A$.

But this measure can be described using Corollary \ref {C555}, and we could use this to prove Proposition \ref {Pkey} directly.
Let us however describe first  the ``explicit restriction-type-property'' for the law $\tilde P$ of the two-point pinned loop $\tilde \gamma$.
Suppose that $A$ is such that $A \cap \R \subset (0,1)$, and that $\HH \setminus A$ is simply connected. We would like describe the conditional law of
$\tilde \gamma$ given that it does not intersect $A$.

In order to do so, let us use a CLE $\Gamma^\#$ that is independent of $\tilde \gamma$.
We will use $\tilde E$, $ E^\#$ and $\tilde E^\#$ to denote the expectation with respect to $\tilde \gamma$, to $\Gamma^\#$ and to both.
 We define $H^\#$, the unbounded connected component of the set obtained by removing from $\HH \setminus A$ all the loops of $\Gamma^\#$ that intersect $A$. Then, we define the conformal map $\varphi_\#$ from $H^\#$ onto $\HH$ such that $\varphi_\# (0)= 0$, $\varphi_\#(1) = 1$ and
$\varphi_\#'(1)=1$. It is easy to check that $\varphi_\#'(0) \le 1$.  The loop $\varphi_\#^{-1} ( \tilde \gamma )$ is clearly  two-point pinned loop that avoids $A$ almost surely.

\begin {proposition}
\label {explicit}
 The conditional law of $\tilde \gamma$ given $\tilde \gamma \cap A = \emptyset$ satisfies, for a bounded continuous function $F$ on the set of loops,
$$\tilde E ( F( \tilde \gamma) | \tilde \gamma \cap A = \emptyset ) = \frac{ \tilde E^\# ( \varphi_\#'(0)^\beta F ( \varphi_\#^{-1} ( \tilde  \gamma)  ))}
{E^\# (( \varphi_\#'(0))^\beta)}.$$
\end {proposition}

\begin {proof}
Proposition \ref {P81} shows that
$$
\tilde E ( F(\tilde \gamma) | \tilde \gamma \cap A = \emptyset )
= \lim_{\eps \to 0}
\frac {\mu ( F (\gamma) 1_{ \gamma \cap A = \emptyset } 1_{ \gamma \cap C_\eps (1) \not= \emptyset} ) }
{ \mu (  1_{ \gamma \cap A = \emptyset } 1_{ \gamma \cap C_\eps (1) \not= \emptyset} ) }.
$$
The measure $\mu$ restricted to the set of
loops that do not intersect $A$, is equal to $\mu_A$ by Proposition \ref {muA}. Recall that
in order to define $\mu_A$, we use another CLE $\Gamma^\#$, that we consider the set $H^\#$ and the conformal map $\psi^\#$ from $\HH$ onto $H^\#$ such that
$\psi^\# (0)= 0$, $(\psi^\#)' (0) = 1$ and $\psi (\infty)= \infty$, and that $\mu_A$ is the image of the product measure $\mu \otimes P$ under the
mapping
$(\gamma, \Gamma^\#) \mapsto \psi^\# (\gamma)$.
Hence,
$$
\tilde E ( F(\tilde \gamma) | \tilde \gamma \cap A = \emptyset )
= \lim_{\eps \to 0}
\frac {E^\# ( \mu ( F(\psi^\#( \gamma))  1_{\psi^\# (\gamma) \cap C_\eps (1) \not= \emptyset} )) }
{  E^\# ( \mu (  1_{\psi^\# (\gamma) \cap C_\eps (1) \not= \emptyset} ))}.
$$
But when $\eps$ is very small, the $\eps$-neighborhood of $1$ is not much distorted by $\varphi_\#$. Furthermore, $\psi^\# \circ \varphi_\#$ is a conformal mapping from the upper half-plane into itself, that depends on $\Gamma^\#$ only, that maps the origin onto itself and its derivative at the origin is equal to $\varphi_\#'(0)$.
 Recall also that $\mu$ satisfies conformal covariance with the
exponent $\beta$. It follows easily by dominated convergence that
\begin {eqnarray*}
\lefteqn {
 \tilde E (  F( \tilde \gamma) | \gamma \cap A = \emptyset  ) } \\
&=& \lim_{\eps \to 0} \frac { E^\#( \varphi_\#'(0)^\beta \mu ( F (\varphi_\#^{-1} (\gamma) )  1_{\gamma \cap C_\eps (1) \not= \emptyset} ) )}
{ E^\# (   \varphi_\#'(0)^\beta \mu ( (   1_{\gamma \cap C_\eps (1) \not= \emptyset} ) ))} \\
&=& \lim_{\eps \to 0}
\left(
\frac { E^\#( \varphi_\#'(0)^\beta \mu ( F (\varphi_\#^{-1} (\gamma) )  1_{\gamma \cap C_\eps (1) \not= \emptyset} ) )}
{\mu ( 1_{\gamma \cap C_\eps (1) \not= \emptyset} ) }
\times \frac { \mu ( 1_{\gamma \cap C_\eps (1) \not= \emptyset }) }
{ E^\# (   \varphi_\#'(0)^\beta \mu (    1_{\gamma \cap C_\eps (1) \not= \emptyset} )) } \right) \\
&=& \frac { E^\# ( \varphi_\#' (0)^\beta \tilde E ( F ( \varphi_\#^{-1} ( \tilde \gamma ) )))}
{ E^\# ( \varphi_\#' (0)^\beta ) }.
\end {eqnarray*}
\end {proof}

We now use this to derive Proposition \ref {Pkey}:

\begin {proof}
Let us consider an independent CLE $\Gamma^\#$ as before, and
for each $A$ as before, such that $A \cap \R \subset (0,1)$,  define $\varphi_\#$. It is important to observe that if $\tilde \gamma$ is a two-point pinned loop, then
 $\varphi_\#^{-1} (\tilde \gamma)$ is always a two-point pinned loop, and that $U( \varphi_\#^{-1} (\tilde \gamma))= U(\tilde \gamma)$.
Furthermore, note that the event $A \cap \tilde \gamma =\emptyset$ is identical to the event that $A \cap  \gamma_*  = \emptyset$.

\begin {figure}[htbp]
\begin {center}
\includegraphics [width=4.5in]{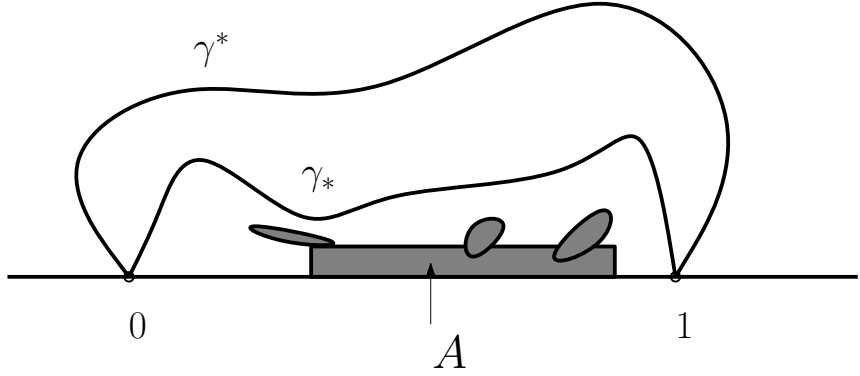}
\caption {Proof of independence (sketch)}
\end {center}
\end {figure}

Hence, we get that for all continuous bounded function $F$ on the space of loops (recall that $\gamma$ and $\varphi_\#$ are independent)
\begin {eqnarray*}
 \tilde E( F( U(\tilde \gamma))  | \gamma_*  \cap A = \emptyset )
&=& E^\# (\varphi_\#'(0)^\beta )^{-1} \tilde E^\# ( \varphi_\#'(0)^\beta F (U ( \varphi_\#^{-1} (\tilde \gamma)) ))\\
&=& E^\# (\varphi_\#'(0)^\beta )^{-1} \tilde E^\# (\varphi_\#'(0)^\beta  F (U(\tilde \gamma))) \\
&=& E^\# (\varphi_\#'(0)^\beta )^{-1} E^\# (\varphi_\#'(0)^\beta ) \tilde E ( F (U(\tilde \gamma))) \\
&=& \tilde E ( F( U(\tilde \gamma))).
\end {eqnarray*}
Since this is true for all such $A$, it follows that $U(\tilde \gamma)$ and $\gamma_*$ are indeed independent.
\end {proof}

In the sequel, we call $P^*$ the law of $U(\tilde \gamma)$. This is a probability measure on paths joining $1$ to $0$ in the upper half-plane.
Note here that $P^*$ is invariant under the Moebius transformations from $\HH$ onto itself that preserve $0$ and $1$. This follows easily from the corollary and from
the fact that the same is true for the two-point pinned measure.

Let us define $P^{**}$ to be the image of $P^*$ under the conformal map $z \mapsto 1 - (1/z)$ of $\HH$ onto itself (that maps $1$ to $0$ and $0$ to $\infty$). $P^{**}$ is then a probability measure on simple paths in $\HH$ from $0$ to infinity, that is scale-invariant (because the multiplications by positive constants are the Moebius transformations in $\HH$ that leave $0$ and $\infty$ invariant).

\section {Pinned loops and SLE excursions}

\subsection {Two-point pinned measure and SLE paths}

The goal of this section is show that the pinned measure $\mu$ associated to a CLE is necessarily one of the  ``SLE excursion measures''.
Let us first combine the results of the previous section and reformulate them into a single tractable
statement: Consider the pinned configuration measure $\bar \mu$ and some
 finite piecewise linear path $\eta$ in $\HH$ that starts on the positive real half-line, has no double points, and such that all of its segments are horizontal or vertical. We parametrize it ``from the real axis to the tip''.
We also suppose that
$$
\mu ( \{ \gamma \ : \ \gamma \cap \eta \not= \emptyset \} ) = 1 .
$$
so that we can view $\mu$ (and $\bar \mu$ on the corresponding set of configurations) as a probability measure.

For each pinned configuration $(\bar \gamma, \bar \Gamma)$ such that $\bar \gamma \cap \eta \not= \emptyset$, we can define the first meeting time $T$ of $\eta$ with $\bar \gamma$, i.e.\ $T = \min \{ t \ : \ \eta (t) \in \bar \gamma \}$. Let us now consider $\tilde \Gamma_\eta$ to be the set of loops of $\bar \Gamma$ that intersect
$\eta [0,T]$. We also call $\gamma_-$ the part of $\gamma$ (when oriented counterclockwise) between $0$ and $\eta (T)$, and $\gamma^+$ the other part.
Let us now define the set $H_-$ to be the unbounded connected component of the set obtained by removing from $\HH$ the union of $\gamma_-$, $\eta [0,T]$ and the loops of
$\tilde \Gamma_\eta$.
We let $\psi_-$ denote the conformal map from $H_-$ onto $\HH$ such that $\psi_- (0-) = 0$, $\psi_- ( \eta_T) = 1$ and $\psi_- (\infty)= \infty$.

Note that $\psi_-$, $\gamma_-$ and $\tilde \Gamma_\eta$ are all deterministic functions of the triple $( \bar \gamma, \bar \Gamma, \eta)$. We will sometimes write
$\psi_{-, \eta}$ and $\gamma_{-, \eta}$, $\gamma^{ +, \eta}$ to indicate the dependence in $\eta$.

\begin {figure}[htbp]
\begin {center}
\includegraphics [width=4in]{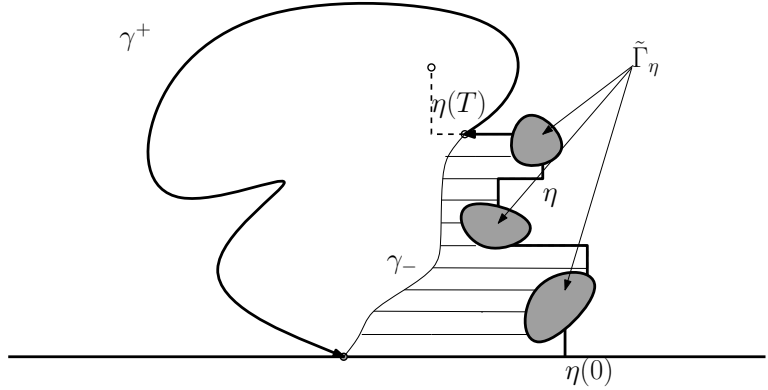}
\caption {Use of independence (sketch)}
\end {center}
\end {figure}
\begin {proposition}
\label {propo}
$\psi_- ( \gamma^+)$ is independent of $\psi_-$ and $\gamma_-$, and its law is $P^*$.
\end {proposition}

\begin {proof}
 This is a direct combination of our previous results. Let $\psi_\eta$ denote the conformal map from $\HH \setminus ( \eta [0,T] \cup \tilde \Gamma_\eta)$ onto $\HH$
with $\psi_\eta (0)= 0$, $\psi_\eta (\eta (T)) = 1$ and $\psi_\eta ( \infty) = \infty$. Conformal invariance of $\tilde P$ under the maps that preserve $0$ and $1$, and the definition of the two-point pinned measure shows that $\psi_\eta (\gamma)$ is independent of $\psi_\eta$ and that its law is the two-point pinned measure $\tilde P$.

If we now define $\gamma_*= \psi_\eta ( \gamma_-)$ and $\gamma^* = \psi_\eta (\gamma^+)$, and $\psi_*$ (which depends on $\gamma_*$ alone) as before, we know from Proposition \ref {Pkey} that
$\psi_* (\gamma^*)$ is independent of $\gamma_*$, and its law is $P^*$. Hence, $\psi_* ( \gamma^* ) $ is independent of $(\gamma_*, \psi_*, \psi_\eta)$.

We can now conclude, noting that $\psi_* ( \gamma^*) = \psi_- ( \gamma^+)$ and that $\gamma_- = \psi_\eta^{-1} ( \gamma_*)$.
\end {proof}

Suppose now that $\eta : [0, l] \to \HH \cup \eta (0)$ is a finite piecewise linear path as before, except that
$\mu ( \{ \gamma \ : \ \gamma \cap \eta [0,l] \not= \emptyset \} ) $ is not equal to one. We can use scaling in order to find some constant $\lambda$, in such a way that we can apply the previous statement to $\lambda \eta$, and it therefore follows that Proposition \ref {propo} still holds if we consider the measure $\mu$ on the event
$ \{ \gamma \ : \ \gamma \cap \eta [0,l] \not= \emptyset \} $, and renormalized in order to be a probability measure on this event.
Since this is true for all $l$, it also follows that:

\begin {corollary}
\label {finalco}
 The same statement holds
 if we consider the measure $\mu$ restricted to the event that $\gamma$ intersects only the last segment of $\eta$ (and not the previous ones), and if we renormalize the measure $\mu$ to be a probability measure on this event.
\end {corollary}

We now want to deduce from this corollary that
$P^*$ is the law of some SLE$_\kappa$ from $1$ to $0$ in $\HH$ i.e., that $P^{**}$ is the law of some chordal SLE$_\kappa$
in $\HH$.
We will use Oded Schramm's conformal Markov property characterization of SLE that we now briefly recall.
Suppose that one is given a continuous simple curve $\xi$ in $\HH \cup \{ 0 \}$ starting at $0$. Let us parametrize $\xi$ according to its half-plane capacity, i.e., in such a way that for each $t$, there
exists a conformal map $g_t$ from $H_t := \HH \setminus \xi [0,t]$ onto $\HH$ such that $g_t (z)  = z + 2t / z + o (1/z)$ as $z \to \infty$. Suppose that the half-plane capacity of $\xi$ is not bounded (i.e., that $\xi$ is defined for all $t \ge 0$), and define $f_t (z) = g_t (z) - g_t ( \gamma_t)$, i.e.\ the conformal map from $H_t$ onto $\HH$ with $f_t  (\xi_t ) =0$ and $f_t (z) \sim z $ at infinity.
Then:
\begin {lemma}[Schramm \cite {Sch}]
\label {schmarkov}
 {If} the law of $\xi$ is scale-invariant in distribution (i.e., the law of $(\lambda^{-1} \xi_{\lambda^2 t}, t \ge 0)$ does not depend on $\lambda$),
 and if for all $t \ge 0$, the conditional law of $f_t ( \xi [t, \infty))$ given $\xi [0,t]$ is identical to the initial law of $\xi$, then $\xi$ is an SLE curve.
\end {lemma}
 Recall that this result can be easily understood using Loewner's theory of slit mappings that shows that
the curve $\xi$ is characterized by the function $ t \mapsto W_t := g_t (\xi_t)$. Indeed,
 the previous conditions imply on the one hand that the random function $t \mapsto W_t$ has the Brownian scaling property (because $\xi$ is scale-invariant in distribution) and on the other hand that it is a continuous process with independent increments.
These two fact imply that $W$
it is a multiple of a standard Brownian motion (and one can call this multiplicative constant $\sqrt {\kappa}$). As the function $W$ fully characterizes $\xi$, this determines the law of the path $\xi$ (up to this one-parameter choice). Furthermore, SLE computations (see \cite {Sch}) show that the fact that the curve is simple implies that $\kappa \le 4$.

\medbreak

There are several natural ways that one can use in order to parametrize a pinned loop. We are going to choose one that is tailored for our purpose, i.e., to recognize SLE
excursions.
Let us first define $r_0$ in such a way that
$$ \mu ( R (\gamma) \ge r_0 ) = 1.$$
Then, the measure $\mu$ restricted to the set $\mathcal Q= \{ \gamma \ : \ R(\gamma ) \ge r_0 \}$ is a probability measure that we will call $Q$ in the present section. Until the rest of the present subsection, we will assume that $\gamma$ is sampled from this probability measure.

Suppose now that
 $\gamma \in \mathcal Q$ and let us orient it ``anti-clockwise'', i.e., $\gamma$ starts at $0$, makes a simple anti-clockwise loop in $\HH$ and comes back to $0$.
On the way, there is the first intersection point $z_0$ of $\gamma$ with $\{ z \ : \ |z| = r_0 \}$.
We  define $b_0 := b_0(\gamma)$ the beginning part of the loop between $0$ and $z_0$ (with the anti-clockwise orientation) and we call $e_0:= e_0(\gamma)$ the other (end-)part
of the loop $\gamma$ between $z_0$ and $0$.

Let $h_0$ denote the conformal mapping from $ \HH \setminus b_0$ onto $\HH$  normalized by $h_0 ( z_0) = 0$, $h_0( \infty) = 1$ and $h_0(0-) = \infty$ (where $h_0 (0-)$ denotes the image of the left-limit of $0$ in $\HH \setminus b_0 $). The end-part $e_0= e_0 (\gamma)$ of the loop between $z_0$ and $0$ is a continuous simple path from $z_0$ to $0-$ in $(H \setminus b_0)  \cup \{ z_0, 0- \}$. Hence, its image under $h_0$ is a simple path from $0$ to infinity in the upper half-plane that we now call $\xi$. We
 parametrize $\xi$ according to its half-plane capacity as before (and define the conformal maps $g_t$ and $f_t$). This therefore defines (via the map $h_0^{-1}$) a parametrization of the path $e_0$.

\begin {figure}[htbp]
\begin {center}
\includegraphics [width=4.2in]{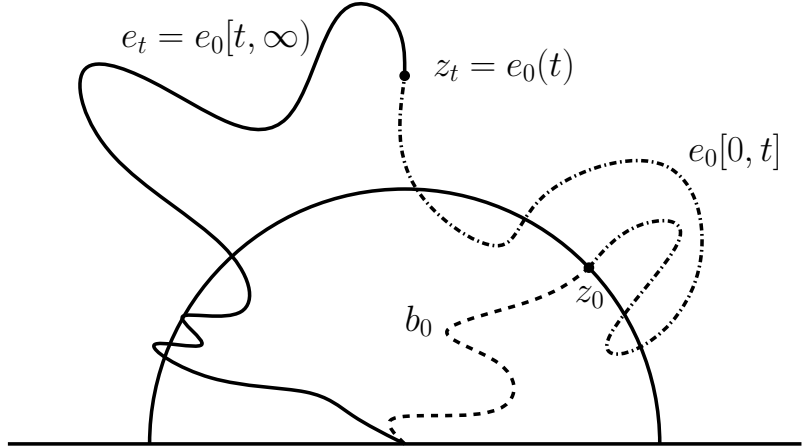}
\caption {Definition of $e_t(\gamma)$ and $b_t(\gamma)$ (sketch)}
\end {center}
\end {figure}

For each $t \ge 0$, we now define $z_t = e_0 (t)$,
$$b_t := b_t (\gamma) = b_0 ( \gamma) \cup e_0 [0,t] \hbox { and } e_t := e_t (\gamma) = e_0 [t, \infty).$$
Let us insist on the fact that $b_t$ and $e_t$ are paths (and not points).
$h_t$ will denote the conformal map from $\HH \setminus b_t( \gamma)$  with
$h_t ( z_t) = 0$, $h_t (0-) =\infty$ and $h_t (\infty)= 1$.
We will now prove the following statement that we will then combine with Lemma \ref {schmarkov}:
\begin {lemma}
 Under the probability measure $Q$, for each given $t \ge 0$, $h_t ( e_t (\gamma)) = f_t ( \xi [t, \infty))$ is independent of $b_t(\gamma)$.
Furthermore, its law is $P^{**}$.
\end {lemma}

\begin {proof}
Let us fix $t \ge 0$. Define $\psi_t$ the conformal map from $\HH \setminus b_t( \gamma)$ on $\HH$ with $\psi_t ( z_t )= 1$, $\psi_t (\infty) = \infty$,
$\psi_t ( 0-) = 0$. Then, the lemma is clearly equivalent to the fact that
$\psi_t ( e_t)$ is independent of $b_t(\gamma)$, and that its law is $P^{*}$.

 Define $b_t (\gamma)$ as before (when $\gamma$ is defined under the probability measure $Q$).
The proof will be based on the independence property of the two-point pinned measure.
In order to use this, we will need to discover the pinned loop via some grid-based path that turns out to be close to $b_t$.  Take $n \ge 1$. Let us try to associate to each $b=b_t( \gamma)$ a grid approximation of $b_t$ ``from the right side'', that we will call $\beta_n$.

First of all, let us note that the diameter of $b_{t+2/n} \setminus b_{t-1/n}$ goes to $0$ almost surely when $n \to \infty$
and that we can find some sequence $\eps_n \to 0$ such that
$$ \lim_{n \to \infty} Q ( d( b_{t-1/2n}, e_t) \ge \eps_n ) = 1 $$
 (this follows easily from the fact that $\gamma$ is $Q$-almost surely a simple curve), where here $d$ denotes the usual Euclidean distance.
Then, when $\delta$ is a mesh-size, we try to define a particular grid-approximation of $b_{t-1/n}$ on the grid $(\delta \N) \times (\delta \Z)$ as follows:
Delete all the edges from the grid that are distance less than $\delta$ from $b_{t-1/n}$, and consider the unbounded connected component $C$ of the graph that one obtains. Let $f$ denote the first edge of $C$ that $\gamma$ hits (it is therefore after ``time'' $t-1/n$). It is clearly on the ``boundary of $C$ seen from $b_{t-1/n}$'', and one can find the simple nearest-neighbor path from the positive real axis to $f$ (where $f$ is its last edge) on $C$ that is ``closest'' to $b_{t-1/n}$. We call this path
$\beta = \beta ( b, t, n ,\delta)$.

\begin {figure}[htbp]
\begin {center}
\includegraphics [width=4in]{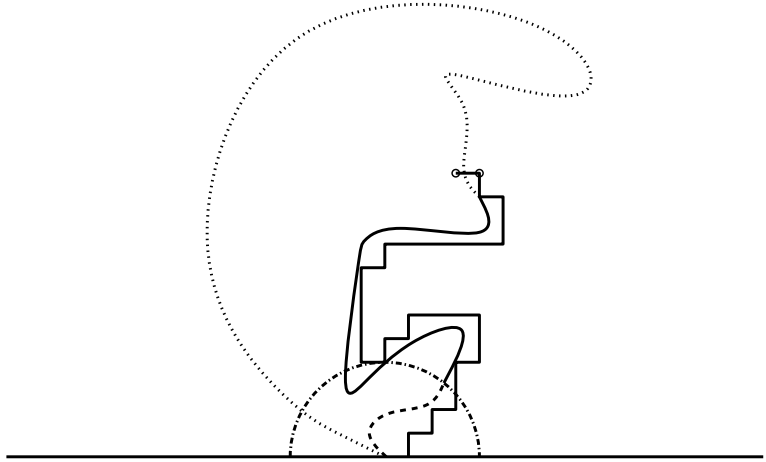}
\caption {Grid-approximation of some $b_t(\gamma)$ (sketch)}
\end {center}
\end {figure}

Clearly, when $t$, $n$ and $b$ are fixed, then when $\delta \to 0$, the (Hausdorff) distance between $\beta$ and  $b$ is going to $0$.
Hence, we can choose $\delta_n$ in such a way that
$$ \lim_{n \to \infty} Q ( d_H ( \beta, b_{t-1/n} ) \le \eps_n^2 ) =1 $$
(where $d_H$ denotes the Hausdorff distance).

Once these sequences $\delta_n$ and $\eps_n$ are chosen, we denote for each given $n$,
the grid approximation $ \beta ( b, t, n , \delta_n)$ by $\beta_n$.
The previous estimates also ensure that
$$
 \lim_{n \to \infty}
Q ( (\gamma \cap  \beta_n) \subset  b_{t+1/n} \setminus b_{t-1/n}) = 1.
$$

Furthermore, we can observe that if we are given $b_t$ and a finite grid-path $\eta$ on $\delta_n \Z \times \delta_n \N$, it suffices to look at the path $b_t$ only until the moment it hits $\eta$ for the first time to check whether $\beta_n = \eta$ or not.

Suppose now that $\eta$ is some given grid-path that has a positive probability to be equal to $\beta_n$.
Let us consider a pinned configuration $(\gamma, \bar \Gamma)$ such that $\gamma$ intersects $\eta_T$,
and let us define $\psi_- = \psi_{-, \eta}$, $\gamma_-= \gamma_{-, \eta}$ and $\gamma^+ = \gamma^{+, \eta}$ as before (recall that $\gamma_-$ and $\gamma^+$ are the two parts of $\gamma$, before and after it hits $\eta (T)$, which is the first point of  $\eta$ that intersects $\gamma$, and that $\psi_-$ corresponds to the domain where
one from $\HH \setminus \eta[0,T]$  the loops of  $\bar \Gamma$ that  intersect $\eta [0,T)$).
Then, we know on the one hand that $\psi_- (\gamma^+)$ is independent of $\psi_-$ and $\gamma_-$, and that its law is $P^*$.
On the other hand, we have just argued that the event $\{\beta_n = \eta \}$ is measurable with respect to $\gamma_-$ (this is because $\gamma_-$ contains the part of $\gamma$ up to the first time it hits $\eta$), so that the conditional law of
$\psi_- ( \gamma^+)$ given $\{ \beta_n= \eta \}$ is also $P^*$.

When $\beta_n= \eta$, we define $\psi_{n, -}$ to be this map $\psi_-$, and we also define $\gamma^{n, +}$ and
 $\gamma_{n, -}$ to be these $\gamma^+$ and $\gamma_-$.
Note that when $n$ is very large,
the probability that $\beta_n$ intersects some macroscopic loop of $\bar \Gamma$ is very small (this is because they are all at positive distance of $\bar \gamma$, and that $\beta_n$ is very close to $\gamma$). This, and our definition of $\beta_n$ ensures readily that $\gamma_{n,-}$ converges (in Hausdorff topology) to $b_t$, and that $\psi_{n,-} ( \gamma^{n,+})$ converges to
$\psi_t (e_t)$.

We are now almost ready to conclude:
For any continuous bounded functions $F$ and $G$ with respect to the Hausdorff topology, we get that:
\begin {eqnarray*}
 Q ( F ( b_t ( \gamma)) G ( \psi_t ( e_t )))
&=&
\lim_{n \to \infty }
Q ( F ( \gamma_{n,-}) G ( \psi_{n , -} ( \gamma^{n, +}))) \\
&=&
\lim_{n \to \infty}
\sum_\eta Q ( 1_{\{ \beta_n = \eta \}} F ( \gamma_{-, \eta }) G ( \psi_{ -, \eta } ( \gamma^{+, \eta })) ) \\
&=&
\lim_{n \to \infty}
\sum_\eta
 Q ( 1_{\{ \beta_n = \eta \}} F ( \gamma_{-, \eta } )) \times P^* ( G) \\
&=& \lim_{n \to \infty} P^* ( G) \times Q (F ( \gamma_{n,-})) \\
&=& P^* ( G) \times Q ( F (b_t (\gamma)))
\end {eqnarray*}
where we have use the independence property of Corollary \ref {finalco} between the second and the third line.
\end {proof}

This lemma has a number of consequences:
It implies (taking $t=0$) that the law of $\xi$ itself is  $P^{**}$. We know also that $P^{**}$ is scale-invariant.
Hence, we get that $\xi$ is a continuous simple curve from $0$ to infinity in the upper half-plane, that is scale-invariant in law, and such that for all $t$, the conditional law of
$f_t (\xi [t, \infty))$ given $\xi [0,t]$ is the same as the law of $\xi$ itself. The previous characterization of SLE therefore implies that:
\begin {corollary}
\label {recognize}
 The curve $\xi$ is an SLE$_\kappa$ for some $\kappa \in [0, 4]$. Furthermore, $\xi$ is independent of $b_0 (\gamma)$.
\end {corollary}
The independence between $b_0(\gamma)$ and $\xi$ shows that the probability measure $Q$ can be constructed as follows: First sample $b_0(\gamma)$; this defines its tip
 $z_0$. Then draw an SLE$_\kappa$ (for some given value of $\kappa$) from $z_0$ to $0-$ in $H\setminus b_0$.
Standard SLE computations (see \cite {SchPerc})
 show that the probability that an SLE$_\kappa$ from $1$ to $0$ in $\HH$ hits the (semi-)circle of radius $\rho > 1$ around the origin decays like
$\rho^{- (8/\kappa) + 1 + o (1)}$ when $\rho \to \infty$.
It follows immediately that if $\kappa$ is the value associated to the pinned measure $\mu$,
$$ \mu ( \{ \gamma \ : \ R (\gamma) \ge \rho \} ) = Q ( \{\gamma \ : \  R( \gamma) \ge \rho \}  )
= \rho^{ - (8/ \kappa) +1 + o(1)}$$
 as $\rho \to \infty$. If we compare this with the scaling property of $\mu$, i.e., the fact that
$$ \mu ( \{ \gamma \ :\ R (\gamma ) \ge \rho ) = \rho^{-\beta} \mu ( \{ \gamma \ : \ R( \gamma) \ge 1\} ),$$
we get that:
\begin {corollary}
One necessarily has $\kappa = 8 / (1+ \beta)$, $\beta \in [1,2)$ and $\kappa \in (8/3, 4]$.
\end {corollary}
The last two statements in the corollary are just due to the fact that we know already that $\beta < 2$ (so that $\kappa > 8/3$) and also that $\kappa \le 4$ (because $\gamma$ does not touch the boundary elsewhere than at the origin), which implies that $\beta \ge 1$.

\subsection {Pinned measure and SLE excursion measure}
\label {pinnedSLE}

In order to show that Corollary \ref {recognize} in fact characterizes the measure $\mu$,
 let us first recall the following rather standard facts concerning SLE processes and excursions of Bessel processes:

\begin {lemma}
\label {l.besexc}
Suppose that $\kappa \in (8/3, 4]$, and that $P^\eps$ denotes the law of an SLE$_\kappa$ from $\eps$ to $0$ in the upper half-plane. Then, the measures
$\eps^{1- (8/\kappa)} P^\eps$ converge (vaguely) to an infinite measure $\nu$ on the set of pinned loops, that we call the SLE$_\kappa$ excursion measure.
\end {lemma}
Vague convergence here means that for all $r>0$, when restricted on the set of paths with diameter greater than $r$, $\eps^{1-(8/\kappa)} P^\eps$ converges weakly (with respect to the Hausdorff topology for instance).

\begin {proof}
Let us fix a positive $\eps$ and consider an SLE$_\kappa$ curve (for $\kappa \in (8/3,  4]$) from $\eps$ to $0$ in $\HH$.
It is possible to parametrize this curve as a Loewner chain in the upper half-plane, with time measured via the half-plane capacity. If we use this parametrization, then the SLE path is defined up to some (random) time $\tau$, and one can define for each $t \ge 0$, the conformal map $g_t$ from $\HH \setminus \gamma [0,t]$ normalized at infinity, and it satisfies Loewner's equation $\partial_t g_t (z) = 2 / (g_t (z) - U_t)$ for all $t \le \tau$ and $z \notin \gamma [0,t]$, where $U_t= g_t (\gamma_t)$.
Note also that $\tau$ is the hitting time of $0$ by $U_t - g_t (0)$ (because $\gamma$ is almost surely a simple curve).

It is quite easy to argue that the process $X_t := U_t - g_t (0)$ is necessarily a Markov process with the same scaling property as Brownian motion, and therefore the multiple of some Bessel process. In fact, direct computations (see for instance \cite {SchW}) shows that $\gamma$ is a so-called SLE($\kappa, \kappa-6$) process stopped at its first swallowing time of $0$, i.e., that
$$ dU_t = \sqrt {\kappa } d B_t + \frac {\kappa - 6}{U_t - g_t (0)} dt$$
and
$$ dX_t = \sqrt {\kappa} d B_t + \frac {\kappa -4}{X_t} dt$$
where $B$ is a standard Brownian motion. Note that $U$ can be recovered from $X$ because
$$ U_t - U_0 = X_t - X_0  + \int_0^t 2 dt / X_t .$$

An important feature here is (and this follows immediately from It\^o's formula) that for $\beta = 8/\kappa - 1$,
$(X_t)^\beta$ is a local martingale up to its first hitting time of $0$.
For instance, when $X_0 = \eps$, then the probability that $X$ hits $\eps' > \eps$ before hitting $0$ is $(\eps/ \eps')^\beta$.
This makes it possible (in the standard procedure to define Bessel excursions, see for instance \cite {PY})
to define a measure $\nu$ on the space of one-dimensional excursions $(Y_t, t \le \tau)$ (i.e., such that $Y_0=Y_\tau=0$ and $Y (0,\tau) \subset (0, \infty)$) in such a way that the $\nu$-mass of the  set of paths that reach level $\eps$ (at some time $\tau_\eps$) is $\eps^{-\beta}$, and that the $(\eps^{\beta} \nu)$-law of
$( Y_{t + \tau_\eps}, t \in [0, \tau - \tau_\eps])$ is the law of $X$ as before started from level $\eps$ and stopped at its first hitting of $0$. Furthermore, one can check (for instance via scaling considerations) that for $\nu$-almost all $Y$, $\int_0^\tau dt / Y_t < \infty$.

Each such excursion $Y$ clearly defines (via Loewner's equation, replacing $X$ by $Y$ and using the fact that $\int_0^\tau dt /Y_t $ is finite) a two-dimensional pinned loop $\gamma$. So,
 we can view the measure $\nu$ as a measure on the set of loops, and it satisfies the following two statement:
\begin {itemize}
 \item The $\nu$-mass of the set of pinned loops $\gamma$ such that $Y$ hits $\eps$ is $\eps^{-\beta}$.
\item On this set, the measure $\nu^\eps = \eps^{\beta} \nu$ is a probability measure,
and one can sample $\gamma$ by first sampling $\gamma [0, \tau_\eps]$, and then
finishing the curve by an SLE$_\kappa$ in $\HH \setminus \gamma [0, \tau_\eps]$ from $\gamma (\tau_\eps)$ to $0$.
\end {itemize}
It follows that the difference between the two probability measures $\nu^\eps$ and $P^\eps$ is just due to the mapping by the random conformal mapping normalized at infinity from
$\HH \setminus \gamma [0, \tau_\eps]$ onto $\HH$. But, when $\eps \to 0$,  this mapping clearly converges to the identity away from the origin, which proves the claim.
\end {proof}

\begin{proposition} \label{muisSLEexcursion}
The pinned measure $\mu$ is the multiple of the SLE$_\kappa$-excursion measure $\nu$ for $\kappa = 8 /(1+\beta)$ (where $\beta$ is the ``exponent'' associated to $\mu$), that is normalized in such a way that $ \mu ( \gamma \hbox { surrounds } i ) = 1$.
\end{proposition}

 A consequence of this proposition is that two CLE's with the same exponent $\beta$ have the same pinned measure.

\begin{proof}
 Recall that we have the following description of $\mu$ (this is just the combination of the scaling property of $\mu$ and of the definition of $r_0$): For each $\eps \ge 0$, the $\mu$-mass of the set of loops with radius at least $\eps r_0$ is
$\eps^{-\beta}$. If we restrict ourselves to this set of loops, and renormalize $\mu$ so that it is a probability measure on this set, we can first sample the part of
$\gamma$ up to the point at which it reaches the circle of radius $\eps$, and then complete it with an SLE$_\kappa$ back to the origin in the remaining domain.

We can note that for some absolute constant $K_0$, if we consider a loop that does not reach the circle of radius $r_0$, then the corresponding Loewner-type function $Y$ as before does not reach $K_0$.
Let us now sample the beginning of the loop $\gamma$ as before up to the first hitting point of the circle of radius $\eps r_0$. Then continue the SLE up to the first time (if it exists) at which the function $Y$ associated to the loop reaches $K_0 \eps$. A fixed positive fraction $K_1$ (independent of $\eps$ because of scale-invariance)
 of these SLEs succeed in doing so. Then, after this point, we still are continuing with an SLE$_\kappa$ in the remaining domain, so that $\gamma$ is close to a sample of $P^\eps$.

Hence, we conclude readily that $\mu$ is a constant multiple of $\nu$ (where the constant is given in terms of $K_0$, $K_1$ and $\beta$).
Another way to describe this constant is to recall that $ \mu ( \gamma \hbox { surrounds } i ) = 1$.
\end{proof}

Let us now make the following comment: Suppose that a CLE in $\HH$ is given. This defines a random collection of loops $(\gamma_j)$. If we now define the
symmetry $S$ with respect to the imaginary axis (i.e.\ the map $z \mapsto -  \bar z$), then it is clear from our CLE axioms that the family $(S(\gamma_j))$ is also a CLE in $\HH$.
We have just seen that each CLE defines a pinned measure $\mu$ that happens to be the (multiple of an) SLE$_\kappa$ excursion measure. But, by construction, the pinned measure associated to the CLE $(S(\gamma_j))$ is simply the image of $\mu$ by the map $S$. Therefore (noting that both these CLEs correspond to the same exponent $\beta$),
the CLEs $(S(\gamma_j))$ and $(\gamma_j)$ correspond to the same SLE excursion measure. Hence:

\begin {corollary}
\label {reversibility}
 If the pinned measure of a CLE is the SLE$_\kappa$ excursion measure for some $\kappa$, then it implies that the law of the corresponding
 SLE$_\kappa$ bubble (as defined in the introduction) is reversible:
The trace of SLE$_\kappa$ bubble traced clockwise has the same law as that of an SLE$_\kappa$ bubble traced anti-clockwise.
\end {corollary}

It has been proved by Dapeng Zhan that SLE$_\kappa$ for $\kappa \le 4$ is indeed reversible (see \cite {Zh}), which in fact implies this last statement.
The present set-up therefore provides an alternative approach to reversibility of SLE paths for these values of $\kappa$.

\section {Reconstructing a CLE from the pinned measure}
\label {S.8}

The goal of the present section is to show that if one knows $\mu$, then one can recover the law of the initial CLE.
This will conclude the proof of Theorem \ref {t2}.

The rough idea is to use the radial exploration and to use the same idea as in the proof of the fact that $\beta < 2$, but ``backwards'':
The fact that $\beta < 2$ will ensure that this exploration (described in Section \ref{s.tail}) can be approximated by keeping only those steps that discover ``large'' loops, because the cumulative contribution of the small ones vanish. Furthermore, these large loops can be described using a Poisson point process of intensity $\mu$. This will lead to the description of a ``continuous exploration mecanism''  that we will relate to SLE($\kappa, \kappa-6$) processes in the next section.

\subsection {The law of $\gamma (i)$}
\label {onepoint}

It is useful to first show how one can recover the law of the loop in the CLE that contains $i$.
Recall that the set-up is the following: We suppose that we are given the law $P$ of a CLE.
The previous sections enable us to define its pinned measure $\mu$ and we have seen that it is necessarily equal (up to a multiplicative constant) to
the SLE$_\kappa$ excursion measure for some $\kappa$ in $(8/3, 4]$.

\begin {lemma} \label{l.mugivesgammailaw}
 The pinned measure $\mu$ characterizes the law of $\gamma(i)$ under $P$.
In other words, if two (laws of) CLEs define the same pinned measure, then $\gamma(i)$ is distributed in the same way for both CLEs.
\end {lemma}

\begin {proof}
Let us use on the one hand the radial exploration mechanism in $\HH$. This defines (for each $\eps$) the geometric number $N$ of exploration steps, the conformal maps
$ \varphi_1^\eps, \ldots , \varphi_{N}^\eps$ that are normalized at $i$ (i.e., $\varphi_n^\eps (i) = i$ and the derivative at $i$ is positive).
 We also define for all $n \le N$,
$$\Phi_n^\eps= \varphi_{n}^\eps \circ \cdots \circ \varphi_1^\eps $$
and $\Phi^\eps = \Phi_N^\eps$.
Recall that:
\begin {itemize}
 \item The random variable $N$ is geometric with mean $1/ u(\eps)$.
\item Conditionally on the value of $N$, the maps $\varphi_1^\eps, \ldots , \varphi_{N}^\eps$ are i.i.d.\ (and their law does not depend on $N$ --- these are the maps corresponding to the CLE exploration conditioned on the fact that $\gamma (i)$ does not intersect $C_\eps$).
\item What one ``discovers'' at the $(N+1)$-th step is independent of the value of $N$ and of the maps $\varphi_1^\eps, \ldots , \varphi_{N}^\eps$. It is just a CLE conditioned by the event that the loop surrounding $i$ intersects $C_\eps$.
\end {itemize}

On the other hand, define a Poisson point process of pinned loops $(\bar \gamma_t, t \ge 0)$ with intensity $\mu$, and  let
$T$ be the first time at which one $\bar \gamma_t$ surrounds $i$. The $\mu$-mass of the set of loops that surround $i$ is equal to $1$,
so that $T$ is an exponential random variable of parameter $1$ (see for instance \cite {RY} for background on Poisson point processes).
For each of the countably many  $t$ in $(0,T)$ such that $\bar \gamma_t$ exists, we denote by $f_t$ the conformal map normalized at $i$ by $f_t (i)=i$ and $f_t'(i) >0$, from the unbounded connected component of $\HH \setminus \bar \gamma_t$ onto $\HH$.
The fact that $\beta < 2$ shows that
\begin {eqnarray*}
 E \left ( \sum_{t < T} a ( f_t) 1_{R (\bar \gamma_t) < 1/2} \right) &\le&  E(T) \mu ( a(f) 1_{R(\gamma) < 1/2 }) \\
&
\le & C \mu ( R(\gamma)^2 1_{R(\gamma) < 1/2}) \ \le \  C' \int_0^{1/2} \frac {x^2 dx}{x^{1+\beta}} \  < \  \infty .
\end {eqnarray*}
Since $\mu ( R ( \gamma ) \ge 1/2 ) < \infty$, the number of times $t$ before $T$  at which $R(\bar \gamma_t) \ge 1/2$ is almost surely finite, so that
almost surely, $\sum_{t < T} a (f_t) $ is finite
(here and in this section, we shall re-use notations and arguments that have been presented in the ``stability of Loewner chains'' paragraph of Section \ref {S43}).
Hence, if for each $r > 0$, we define the iteration $\Psi^r$ of the finitely many $f_t$'s for $t <T$ (in their order of appearance, let us call the corresponding times $t_1(r), \ldots , t_k(r)$) that correspond to pinned loops $\bar \gamma_{t_j}$'s of radius greater than
$r$, we get that as $r \to 0+$, the maps $\Psi^r$ converge (in Carath\'eodory topology, with respect to the marked point $i$) to some conformal map $\Psi$ that can be interpreted as
the iteration of all the conformal maps $(f_t, t < T)$ in their order of appearance.
Note that $\Psi$ maps some open set onto the upper half-plane, but that we have no information at this point on the regularity of the boundary of this open set (whether it is a curve or not, etc.).

Our goal now is to prove that when $\eps$ is small, then
$\Psi$ and $\Phi^\eps$ are very close (in law -- with respect to the Carath\'eodory topology from $i$). This will ensure that the law of
$\gamma (i)$ in the CLE is identical to that of $\Psi^{-1} ( \bar \gamma_T)$, i.e.\ the image under $\Psi^{-1}$ of an independent sample of $\mu^i$.
In order to do so, we are going to introduce the cut-off $\Psi^r$ as before. For each $r > 0$, we define
$n_1(r), \ldots, n _k(r)$ the steps before $N$ at which
 one discovers a loop intersecting $C_r$ (note that this number of steps $k$ and the values of the $n_j$'s are random, and depend on $\eps$ as well).
 The description of $\mu$ shows that for each fixed $r>0$,
the joint law of
$(\varphi_{n_1(r)}^\eps, \ldots , \varphi_{n_{k(\eps)} (r)}^\eps)$ converges precisely to that of the corresponding $(f_{t_1(r)}, \dots , f_{t_k(r)})$;
the law of their composition $\varphi_{n_k(r)}^\eps \circ \cdots \circ \varphi_{n_1 (r)}^\eps$ therefore converges to that of $\Psi^r$.

It now just remains to control
$$ \sum_{n \le N} 1_{n \notin \{ n_1 (r), \ldots, n_k (r)\}} a (\varphi_n^\eps) $$
as $r \to 0$, uniformly with respect to $\eps$.
We will bound the expected value of this quantity. Recall that $\beta < 2$, that the half-plane capacity of a set contained in a disc of radius smaller than $r$ around the origin is bounded by a constant $C$ times $r^2$, that $N$ follows a geometric random variable of mean $1/ u(\eps)$, and that conditionally on
$N$, the maps $\varphi_1^\eps, \ldots , \varphi_N^\eps$ are i.i.d.\
It follows that
$$
E ( \sum_{n \le N} 1_{n \notin \{ n_1 (r), \ldots, n_k (r)\} } a (\varphi_n^\eps) )
\le E ( N) E ( a (\varphi_1^\eps) 1_{ n_1 (r) \not= 1 })
$$
and that
$$
E ( a (\varphi_1^\eps) 1_{ n_1 (r) \not= 1} ) \le C \times E ( R(\tilde \gamma (\eps))^2 1_{R (\tilde \gamma(\eps)) \le r }).
$$
This last expectation will be controlled thanks to the bound $\beta < 2$:
Let us first fix $\delta =1/2$ and choose some $a \in (0,  (2 - \beta)/2)$, for instance $a= (2 -\beta)/4$. We know that
$$ \lim_{\eta \to 0 } \frac {v (\delta \eta) }{v (\eta) } \ge \delta^{2-2a} . $$
 In particular, for some $\eta_0 >0$, we get that for all $\eta< \eta_0$,
$$\frac {v( \delta \eta ) }{v (\eta)} \ge \delta^{2-a}.$$
Hence, there exists a constant $C'$ such that for all $m \ge n \ge 1$
$$\frac { v (\delta^m ) }{ v (\delta^n) } \ge C' (\delta^{m-n})^{2-a}.$$
Hence, if $\delta^{n_0+1} \le \eps \le \delta^{n_0} \le \delta^{n_1+1} \le r \le \delta^{n_1}$, we get
that
\begin {eqnarray*}
\lefteqn {E ( \sum_{n \le N} 1_{n \notin \{ n_1 (r), \ldots, n_k (r)\}} a (\varphi_n^\eps) ) }\\
& \le & {E(N) \times E ( R(\tilde \gamma (\eps))^2 1_{R (\tilde \gamma(\eps))^2 \le r })}
\\
& \le & \frac {C}{ u(\delta^{n_0})} \sum_{j \in [n_1, n_0]}  \delta^{2j} P ( R (\tilde \gamma ( \delta^{n_0} )) \ge \delta^{j+1} )
\\
&\le & C '' \sum_{j \in [n_1, n_0 ]} \delta^{2j} \frac {v( \delta^{n_0-j-1})}{ u( \delta^{n_0}) } \\
& \le & C''' \sum_{j \ge n_1} \delta^{2j} (\delta^{j})^{-2+a} \\
& \le & C''' r^a
\end {eqnarray*}
where all the constants do not depend on $r$ and $\eps$, so that this quantity
converges to $0$ as $r \to 0+$ independently of $\eps$; this completes the proof of the fact that
the law of $\Phi^\eps$ converges to $\Psi$ (in Carath\'eodory topology).

Recall that for each $\eps >0$, the law of $\gamma(i)$ is described as follows: We consider on the one hand a sample of
 the conformal map $\Phi^\eps= \Phi_N^\eps$, and on the other hand, an independent sample of the loop  $\bar \gamma^\eps$ that surrounds $i$ in a CLE in $\HH$ conditioned to intersect $C_\eps$. Then, $\gamma(i)$ is distributed exactly as $(\Phi^\eps)^{-1} ( \bar \gamma^\eps)$.

We have just seen that $\Phi^\eps$ converges in law to $\Psi$ and we have also proved that the (conditional) law of $\bar \gamma^\eps$ converges to $\mu^i$.
 It follows therefore that the law of $\gamma (i)$ is described as the image under $\Psi^{-1}$ of an independent sample of $\mu^i$. This description is based solely on the measure $\mu$, so that the law of $\gamma (i)$ can indeed be fully recovered from $\mu$.
\end {proof}

We may observe that the previous proof also shows that if $(z_1, \ldots, z_m)$ are other points in $\HH$, if we restrict ourselves to the event where all these $m$ other points are surrounded by some macroscopic loops that one discovered in this radial exploration mechanism before (or at) the $N$-th step, the joint distribution of these $m$ loops (some of which may be the same if some loops surround several of these points) in the $\eps \to 0$ limit is also described via the Poisson point process of pinned loops i.e.\ via $\mu$.

\subsection {Several points} \label{s.severalpoints}

We now want to prove the more general result:

\begin {proposition}
Two CLE probability measures that define the same pinned measure $\mu$ are necessarily equal.
\end {proposition}

In other words, the pinned measure $\mu$ characterizes the law of the corresponding CLE. This clearly implies Theorem \ref {t2}.

\begin{proof}

Note that the law of the CLE is characterized by its ``finite-dimensional marginals'', i.e., by the joint distribution of 
$$ ( \gamma (i), \gamma (z_1) ,  \gamma (z_2), \ldots, \gamma (z_m))$$
for any finite set of points $z_1, \ldots, z_m$ in $\HH$ (if we consider the CLE defined in the upper half-plane).

Let us first focus on the law of $( \gamma (i), \gamma (z))$ for some given $z \in \HH$ (the general case will be very similar).
In order to motivate what follows, let us informally describe what could happen if we simply try to use exactly the same procedure as in the previous
subsection.
That is, we consider the conformal maps
$ \varphi_n^\eps$,
$\Phi^\eps_n$ for $n \le N$ as before.
 Let $K$ denote the largest $k \le N$ such that $z \in (\Phi^\eps_k)^{-1} ( \HH)$ (and write $K = N$ if $z$ is not swallowed during the first $N$ steps).
We want to keep track of what happens to $z$, so we define  $a_k = \Phi_k^\epsilon(z)$
(which will be defined up to $k \le K$). Let us now fix some small positive $\delta$ and consider $\eps < \delta$.

We will also introduce another step $K' \le K$, which is the first step $k$ at which either $| a_k | \le \delta$ or $k=K$.
Note that if $K' < K$, this means that after the $K'$-th step, neither $\gamma (z)$ nor $\gamma (i)$ have been discovered, but that the points $z$ and $i$ are ``conformally'' quite far away from each other in the domain
$H_{K'}=(\Phi_{K'}^\eps)^{-1} ( \HH)$ (i.e., the Green's function $G_{H_{K'}}(z,i)$ is  small --- recall that $\delta$ is small). This can happen if the exploration gets close to one of the two points $i$ or $z$ in $\HH$, or if it almost disconnects one from the other in $\HH$ (and this is in fact a scenario that might really occur with positive probability, even in the $\delta \to 0$ limit).

We can subdivide the case  $K' = K$ into three possibilities: The loop $\gamma (z)$ has been discovered before the $N$-th step (in other words: $K <N$), the loop $\gamma (z)$ is discovered at the $N$-th step (then with high probability $\gamma (i) = \gamma (z)$ when $\eps \to 0$), or it has still to be discovered after then $N$-th step. In all these three cases, the arguments of the previous subsection allow us to conclude that the joint law of $(\gamma (i), \gamma (z))$ on the event that $K' = K$ can be described via the Poisson point process of pinned loops, i.e.\ thanks to the knowledge of $\mu$.

When $K' < K$, we could just try to continue the exploration process after $K'$, but at some later step $k$, it could happen that the exploration procedure captures $z$ because $|a_k| \le \eps$. Or more generally, the loop $\gamma (z)$ could be discovered in the discrete exploration process at some step $k$ via some loop that intersects $C_\eps$ and that has a very small radius (this does not necessarily contradict the fact the $\gamma (z)$ is at positive distance of $z$). In such a  case, the previous argument will clearly not work, as the exploration step corresponding to the discovery of $\gamma (z)$ would not be apparent in the limiting Poisson point process of ``macroscopic'' pinned loops.

Here is a simple way of fixing this: At the step $K'$, instead of doing an exploration using the semi-circle $C_\eps$, we use the semi-circle $C_{\sqrt {\delta}}$ of radius $\sqrt {\delta}$.  Note that we are only exploring the loops that hit the semicircle (not all the loops hitting the semi-disc) and that $| a_{K'} | \le \delta$. The harmonic measure of this semi-circle in $\HH$ seen from $a_{K'}$ or from $i$ is small (if $\delta$ is small). After this $K'$-th step, $z$ and $i$ are then in two different domains. Furthermore, the probability to discover a loop that intersects also $C_{\delta^{3/4}}$ or $C_{\delta^{1/4}}$ at this $K'$-th step is very small.

\begin {figure}[htbp]
\begin {center}
\includegraphics [width=5in]{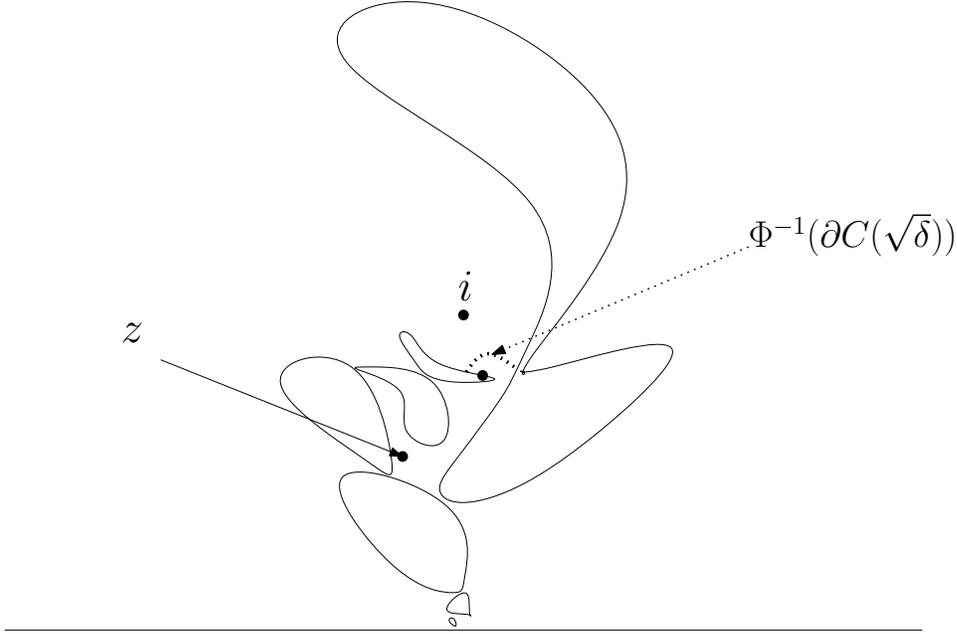}
\caption {Near separation of $z$ from $i$ (sketch).}
\end {center}
\label{imx9}
\end {figure}

The same  arguments as above show convergence (as $\epsilon \to 0$) of the laws of the $\Phi^\epsilon_{K'}$ (in the Carath\'eodory sense, with normalization at $i$).  With high probability, the $K'$-th exploration step changes the log conformal radius viewed from either $z$ or $i$ by a fraction that goes to $0$ when $\delta \to 0$.  If we keep in mind that the restriction property still holds for non-simply connected subsets of $\HH$ (this is the final observation of Section \ref {S.2}), we see that we can then continue the discrete radial exploration (aiming towards $i$) with $\eps$-semi-circles in the connected component containing $i$ after the $K'$-th step, and independently another discrete radial exploration (aiming at $z$) in the domain containing $z$. In both cases, we continue until we discover $\gamma (i)$ and $\gamma (z)$, and the same arguments as above allow to approximate the laws of these two loops in the two respective domains via two Poison point processes of loops.

We can therefore  conclude that up to an error that tends to $0$ with $\delta$, the joint law of $(\gamma(i), \gamma(z))$ on the event where $K' < K$ is also fully described via a procedure that is based on the knowledge of $\mu$ only. Hence, the joint law of $(\gamma (i), \gamma (z))$ is fully determined by the pinned measure $\mu$.

Let us now consider the case where there is more than one additional point, i.e.,
when we look at the joint law of $(\gamma (i), \gamma (z_1), \ldots, \gamma (z_m))$.
 We then use the same argument as above, stopping at the first time that {\em at least one} of the images of the other points gets to a distance less than $\delta$ from the origin, i.e., at the step
$$
K' = \min \{ k \ : \ \min ( | a_k^1 | , \ldots, | a_k^m | )  \le \delta \},
$$
where $a_k^j = \Phi^\eps_k ( z_j) $. Note that at this step, it could happen that several $|a_k^j|$ are in fact quite small simultaneously.
 If we would explore as before by cutting out the semi-circle
$ C_{\sqrt \delta}$, we might be unlucky and have some $a_k^j$ that lies very close to this semi-circle.
However, it is clear that for at least one semi-circle $ C_{\delta^{1/j}}$ out of the $m$ semi-circles, $ C_{\delta^{1/2}}$, $ C_{\delta^{1/3}}$, \ldots, $C_{\delta^{1/m+1}}$,
the harmonic measure of $C_{\delta^{1/j}}$ in $\HH$ seen from {\em any} of the $m$ points $a_k^1, \ldots, a_k^m$ is very small (i.e., tends uniformly to $0$ when $\delta \to 0$).
We therefore choose to use this particular semi-circle for the $(K' + 1)$-th exploration step. This reduces the problem to two independent explorations in two new domains as before, and each of the two domains contains strictly less than $m+1$ points. We can then inductively continue the discrete exploration procedure.

For each fixed $\delta$, this shows that up to a small error (that vanishes when $\delta \to 0$), the $\eps \to 0$ limit of this exploration procedure is well-approximated via Poisson point processes with intensity $\mu$. This enables to conclude that the joint law of the loops
$(\gamma (i), \gamma (z_1), \ldots, \gamma (z_m))$ can be fully described in terms of $\mu$.
\end {proof}

\section {Relation to CLE$_\kappa$'s}
\label {S.8bis}

It now remains to make the connection with the CLE$_\kappa$ families defined in \cite {Sh} via SLE($\kappa, \kappa-6$) branching trees.
This will prove of Theorem \ref {t3}.

\subsection{Background about SLE($\kappa, \kappa -6$) processes} \label{s.slekrbackground}
We first  briefly review
some of the properties of SLE($\kappa, \kappa-6$), in the case where $\kappa \le 4$ that we are focusing on in the present paper. We refer to \cite {Sh} for the more precise statements and their proofs.
We have already encountered the SLE($\kappa, \kappa -6$) process in Section \ref {pinnedSLE}. Recall that a process $(X_t, t \ge 0)$ started from $X_0 = x \not= 0 $ is called a Bessel process of dimension $\delta \ge 0$ if it is the solution to the ordinary differential equation
\begin {equation}
\label {SDE}
 dX_t = d B_t + \frac {\delta -1}{2 X_t} dt
\end {equation}
and that this process is well-defined up to its first hitting time of the origin (which is almost surely finite if $\delta < 2$).
When $\delta > 0$, it is possible to define it (uniquely) also after it hits the origin, in such a way that $X_t \ge 0$ when $t \ge 0$ and that the Lebesgue measure of the time spent by $X$ at the origin is 0. This is the ``instantaneously reflected Bessel process'', that can also be defined by concatenating a Poisson point process $(e_\ell)_{\ell \ge 0}$ of Bessel excursions (see for instance \cite {RY} for background on Bessel processes).
We will also use a variation of this process obtained when tossing a fair coin for each excursion of the Bessel process in order to decide if it is positive or negative (this process is therefore defined on a larger probability space than the Brownian motion $B$), which is the ``symmetrized Bessel process''.

To each excursion $e$ of the Bessel process of dimension $\delta$, we associate $J(e)$ the integral of $dt / e(t)$ on this excursion. Recall that $(X_t)^{\beta}$ is a local martingale as long as $X_t$ stays away from the origin (for $\beta = 2 - \delta$); it follows from a simple scaling argument that the $J(e)$'s are finite and
that the intensity of the  Poisson point process $(J(e_\ell), \ell \ge 0)$ is a multiple of $dx / x^{1+\beta/2}$.

It is already apparent in Section \ref {pinnedSLE} that it will be essential to try to make sense of a quantity like
 $\int_0^t ds / X_s$ when $X$ is a Bessel process. In fact, we will need to focus on the case where $\delta = 3 - 8/ \kappa   \in (0,1]$ and $\beta \in [1, 2)$, where
this integral is infinite as soon as $X$ starts touching the origin (this is because $\int_0 dx / x^{1+(\beta/2)}$
 diverges so that almost surely, $\sum_{\ell \le 1} J(e_\ell)
= \infty$). There are two ways around it this problem. The first one involves the notion of ``principal values'' and is described in \cite {Sh}. The other one, which works for all $\delta \in (0,1]$, is simply to consider the symmetrized Bessel process that chooses its signs at random according to independent fair coin tosses on each excursion. In this case, the integral $\int_0^t ds / X_s$ is not absolutely convergent, but the process
$t \mapsto I_t := \int_0^t ds / X_s$  can be nevertheless easily defined (this corresponds
exactly to the existence of symmetric $\alpha$-stable processes for $\alpha \in [1,2)$, even if $\alpha$-stable subordinators do not exist when $\alpha \ge 1$). We refer to \cite {Sh} for a more detailed description of the possible options.
What we will need here is that in all these cases, one defines a continuous Markov process $(X_t, I_t)_{t \ge 0}$ such that:
\begin {itemize}
\item $X$ is a solution to the SDE (\ref {SDE}) when it is away from the origin.
\item The Lebesgue measure spent by $X$ at the origin is $0$.
\item On the set of times where $X_t$ is away from the origin, $t \mapsto I_t$ is differentiable and its derivative is $1/X_t$.
\item $(X_t, I_t)_{t \ge 0}$ satisfies the Brownian scaling property.
\end {itemize}

The way to define an SLE($\kappa, \kappa-6$) process out of this couple $(X_t, I_t)_{t \ge 0}$ goes as follows: Define for all $t \ge 0$, the continuous process
$$ U_t = \sqrt {\kappa} X_t + 2 I_t / \sqrt {\kappa} $$
and construct the chordal Loewner chain driven by this function, i.e., solve the Loewner differential equation for all $z \in \HH$
$$ \partial g_t (z) = 2 / (g_t (z) - U_t)$$
started from $g_0(z)=z$
in order to define a random Loewner chain. If we define $O_t= I_t / \sqrt {\kappa}$, then when $X_t$ is away from the origin, i.e., when $U_t \not= O_t$, one has
$ \partial_t O_t = 2 / (O_t - U_t)$, i.e., $O_t$ follows the image of some boundary point under the Loewner flow.
Hence, putting our various equations together (definitions of $X_t$, $I_t$, $O_t$, and the relation between $\kappa$ and $\delta$), we get that
 on the set of times where $O_t \not= U_t$,
$$ dU_t = \sqrt {\kappa} dB_t + (\kappa -6)  \frac {dt}{O_t -U_t} \hbox  { and } dO_t = \frac {2 dt} {O_t - U_t }.$$
This continuous driving function $t \mapsto U_t$ defines a Loewner chain, but it is not clear whether this chain is almost surely generated by a path.
However, on those time-intervals where $X_t \not= 0$, the law of $t \mapsto U_t / \sqrt {\kappa}$ is (locally) absolutely continuous to that of a standard Brownian motion, and the Loewner chain is therefore tracing streches of simple paths during these intervals. In fact, each of these intervals correspond to a {\em simple loop} traced by the SLE($\kappa, \kappa-6$) Loewner chain. We can say that the positive excursions of $X$ correspond to anti-clockwise loops while the negative ones correspond to clockwise loops.

We have already mentioned that such chordal SLE($\kappa, \kappa-6$) chains are of particular interest because of the following
properties \cite{SchW}: They are target-independent, and if $X_t > 0$, then the process locally continues exactly as an SLE$_\kappa$ targeting  $g_t^{-1} (O_t)$.
This makes it is possible (see \cite {Sh}) to construct a branching  SLE($\kappa, \kappa - 6$) process whose law is invariant under conformal automorphisms of $\HH$ that fix the origin. In fact, it is also possible to define a radial version of the SLE($\kappa, \kappa-6$) process (see \cite {SchW, Sh}) in order to
also explore the connected components that are ``cut out'' without having a point on the boundary of $\HH$. This can be either defined using a ``radial-chordal'' equivalence (like that of SLE(6)) or  viewed as a rather direct consequence of the target-independence of the (chordal) SLE($\kappa, \kappa-6$) (just interpreting the past of the SLE as new boundary points).
The obtained set is then dense in the upper half-plane, and it is therefore possible to define the family of all loops that are traced in this way: these are exactly the CLE$_\kappa$ defined in \cite {Sh}. In particular, for any given choice of the root (say at the origin i.e., the SLE($\kappa, \kappa-6$) is started from $(0, 0)$), any given $z$ in the upper half-plane is almost surely surrounded by an (outermost) loop in this CLE$_\kappa$, and we will denote it by $\hat \gamma (z)$.

Let us make the following simple observation that will be useful later on: Suppose that $z_1$ and $z_2$ are two points in the upper half-plane, and that the (random) time $\tau$ at which this (branching) SLE($\kappa, \kappa-6$) process separates $z_1$ from $z_2$ is finite. Then, there are two possible ways in which this can happen:
\begin {itemize}
 \item Either $\tau$ is the end-time of an excursion of $X$, where the SLE traced a closed loop that surrounds one of the two points but not the other.
 \item Or the SLE has not yet traced a closed loop that surrounds one of the two points and not the other, but the Loewner chain traced by the paths nevertheless
separates the two points from one another. This can for instance happen if the SLE($\kappa, \kappa-6$) touches the boundary of the domain at a limit time that occurs after the tracing of infinitely many small loops (recall that the SLE($\kappa, \kappa-6$) is {not} a simple path).
\end {itemize}

\subsection {SLE($\kappa, \kappa-6$) and Poisson point process of pinned loops}

The definition of SLE($\kappa, \kappa-6$) clearly shows that it is possible to reconstruct the process $(X_t, t \ge 0)$ starting from a Poisson point process of (positive) Bessel excursions $(e_\ell)_{\ell \ge 0}$ of the corresponding dimension. More precisely (for the symmetrized Bessel process), one defines also an i.i.d.\ family $( \eps_\ell)_{ \ell \ge 0}$
 (with $P( \eps_\ell= 1) = P ( \eps_\ell = -1) = 1/2$) of ``random signs'' and then defines $X$ by concatenating the excursions $( \eps_\ell e_\ell, \ell \ge 0)$.
Each individual excursion corresponds to a pinned loop $\bar \gamma_\ell$ in the upper-half plane. In fact, this pinned loop is an SLE excursion, and one has tossed the
 fair coin in order to decide whether to trace it clockwise or anti-clockwise. Hence, an SLE($\kappa, \kappa-6$)
defines the same Poisson point process of
pinned loops $(\bar \gamma_\ell)_{\ell \ge 0}$ with intensity $\mu$ as in Section \ref {onepoint}.

But there is a difference in the way this Poisson point process is used in order to construct the ``composition'' of the corresponding conformal maps $f_\ell$.
Here, one does not compose the conformal maps $f_\ell$ that are normalized to have real derivative at $i$.  Instead, one keeps track of the previous location of the tip of the curve and continues growing from there.
More precisely, suppose that a pinned loop $\bar \gamma$ is given, and that one decides to trace it clockwise. Consider the unbounded connected component $H$ of its complement in $\HH$ and let $0+$ denote the boundary point of $H$ defined as ``the origin seen from the right''. Define the conformal map $\tilde f$ from $H$ onto $\HH$ such that
$\tilde f(0+) = 0$ and $\tilde f (i) = i$. If the loop was traced counterclockwise, just replace $0+$ by $0-$ in this definition of $\tilde f$.

Then, if one focuses on the radial SLE($\kappa, \kappa-6$) only at those times $t$ at which $X_t = 0$ and before it draws a loop surrounding $i$, we have exactly  the iteration of the conformal maps $\tilde f_\ell$.

In other words, let us define $c_\ell \in ( - \pi, \pi]$ in such a way that $\tilde f_\ell' (i) = e^{i c_\ell}$; then the difference between the radial SLE exploration procedure and the continuous iteration of maps $(f_\ell)$ described in Section \ref {onepoint} is this ``Moebius rotation''  at the end of each discovered loop. More precisely, let $\theta_c$ denote the Moebius transformation of $\HH$ onto itself with
$\theta_c (i)=i$ and $\theta_c'(i)= e^{i c}$; then $\tilde f_\ell = \theta_{c_\ell} \circ f_\ell$.

This gives a motivation to modify the discrete radial exploration mechanism of a CLE that we used in Lemma \ref {l.mugivesgammailaw} as follows: After each step, instead of
normalizing the maps $\varphi_n^\eps$ at $i$ by $\varphi_n^\eps (i) = i $ and $(\varphi_n^\eps)' (i) > 0$, we use the modified map $\tilde \varphi_n^\eps$ defined by
$\tilde \varphi_n^\eps (i) = i$ and by tossing a fair coin in order to decide whether $\tilde \varphi_n^\eps (-\eps) = 0$ or $\tilde \varphi_n(\eps) = 0$.
Because the law of a CLE in $\HH$ is invariant under Moebius transformations, this does not change the fact that the $(\tilde \varphi_n^\eps)$'s are i.i.d., that $N$ is a geometric random variable with mean $1/ u( \eps)$, and that what one discovers at the $(N+1)$-th step is independent of the value of $N$ and of the maps
$\tilde \varphi_1^\eps, \ldots, \tilde \varphi_N^\eps$.

We would now like to see that this discrete exploration mechanism converges (as $\eps \to 0$) exactly to the radial SLE($\kappa,\kappa-6$) procedure described before.
We already know from the proof of Lemma \ref {l.mugivesgammailaw} that if we focus on the discovered pinned loops of radius larger than $r$ (for each fixed $r$), then the corresponding discovered loops converge in law to the corresponding SLE loops (and that the laws of the corresponding conformal maps therefore converge too).
The point is therefore only to check that the effect of all these additional ``Moebius rotations'' remains under control (in other words, that it tends to $0$ when $r \to 0$, uniformly with respect to $\eps$). But as we shall now see, this follows from the same arguments that allow us to define symmetrized Bessel processes (respectively symmetric $\alpha$-stable processes) as almost sure limits of processes obtained by the cut-off of all Bessel excursions of small height or length (resp. the cut-off of all jumps of small size) i.e., it is a straightforward consequence of the fact that $\beta < 2$:
For each $n \le N$, we define $\varphi_n^\eps$ and $\tilde \varphi_n^\eps$ as before, and choose $c_n^\eps$ in such a way that
 $\tilde \varphi_n^\eps = \theta_{c_n^\eps} \circ \varphi_n^\eps$.

Note that there exists a universal constant such that as soon as $R(\tilde \gamma (\eps)) \le 1/2$, the corresponding $c^\eps$ satisfies
$ | c^\eps | \le C \times  R ( \tilde \gamma (\eps))$
(because harmonic measure scales like the diameter).
It follows, using the independence between the exploration steps, and the symmetry (i.e., the fact that  $c_1^\eps$ and $-c_1^\eps$ have the same law) that
$$
E \Big( ( \sum_{n \le N} 1_{n \notin \{ n_1 (r), \ldots, n_k (r)\} } c_n^\eps )^2 \Big)
= E ( N) E ( (c_1^\eps)^2 1_{1 \not= n_1 (r)}) \le \frac {C}{ u(\eps)} E ( R(\tilde \gamma ( \eps))^2 1_{R(\tilde \gamma (\eps)) < r })
.$$
We have already shown in the proof of Lemma \ref {l.mugivesgammailaw} that this last quantity goes to $0$ as $r \to 0$, uniformly with respect to $\eps$.
Note also that
$m \mapsto \sum_{n \le \min (m,N)} 1_{n \notin \{ n_1 (r), \ldots, n_k (r)\} } c_n^\eps$ is a martingale. Hence, Doob's inequality ensures that uniformly with respect to
$\eps$,
$$ \lim_{r \to 0+} E \Big( \sup_{ m \le N}  (\sum_{n \le m } 1_{n \notin \{ n_1 (r), \ldots, n_k (r)\} } c_n^\eps )^2 \Big) \to 0 .$$
From this, it follows that forgetting or adding these ``rotations'' to the exploration procedure does not affect it much (in Carath\'eodory sense, seen from $i$); one can for instance use the Loewner chain interpretation of this exploration procedures, and the above bound on the supremum of the cumulative rotations implies that the radial Loewner driving functions (with or without these rotations) are very close.

\subsection{CLE and the SLE ($\kappa, \kappa -6$) exploration ``tree''}

In this section, we finally complete the proof of Theorem \ref{t3} by showing that indeed the loops in CLE necessarily have the same law as the set of loops generated by an SLE($\kappa, \kappa -6$) (as constructed in \cite {Sh}).  As before, it is enough to prove that for every finite set $\{z_1, \ldots, z_m\}$ of points in $\HH$, the joint law of the $\gamma(z_j)$ is the same for the CLE as for the CLE$_\kappa$ loops (more precisely, as the outermost loops surrounding these points in the CLE$_\kappa$).

Fix the points $\{z_1, \ldots, z_m\}$ in $\HH$ and consider a radial SLE($\kappa, \kappa-6$) targeting $z_1$, say.  We do not know whether SLE($\kappa, \kappa-6$) is a continuous curve, but we recall that it is continuous at times in the interior of the loop-tracing intervals (during the excursions of the Bessel process).
For each $t >0$, the hull of the process $K_t \subset \HH$ is compact and its complement $H_t = \HH \setminus K_t$ contains $z_1$.  The chain $(K_t)$ also traces loops as explained before. We define $\tau$ to be the first time at which at least
one of the points $z_2, \ldots, z_m$ is ``swallowed'' by $K_t$, i.e.\ separated from $z_1$.
We will use $K_{t-}$ to denote $\cup_{s < t} K_s$.
Note that by the target invariance property, the law of $(K_t, t < \tau)$ does not depend on the fact that we singled out $z_1$ i.e., $\tau$ is just the first time at which $K_{t-}$ separates the set of $m$ points into at least two parts.
 As mentioned earlier, the time $\tau$ can occur in two different ways: either the SLE has traced a loop surrounding some $z_j$ or it has simply disconnected the domain into two parts.

\begin{lemma} \label{l.SLECLEmix}
Consider the following method of generating a random loop $\tilde \gamma(z_j)$ surrounding each $z_j$.
\begin{enumerate}
\item For each $j$, if $z_j$ is surrounded by one of the loops traced by $(K_t, t \le \tau)$, then we let $\tilde \gamma(z_j)$ be that loop.
\item In each component of $\HH \setminus K_{\tau-}$ that is not surrounded by a loop traced by $(K_t, t \le \tau)$,
we then construct an independent copy of the CLE (conformally mapped to that ensemble), and for each $z_j$ in that component, we let $\tilde \gamma(z_j)$ be the loop that surrounds $z_j$ in this CLE.
\end{enumerate}
Then the set $\{\tilde \gamma(z_1), \tilde \gamma(z_2), \ldots \tilde \gamma(z_m)\}$ agrees in law with $\{\gamma(z_1), \gamma(z_2), \ldots \gamma(z_m)\}$.
\end{lemma}

Given this lemma, Theorem \ref{t3} follows immediately by induction. Indeed, suppose that the collections $\{\hat \gamma(z_j'), j \le m-1\}$ and $\{\gamma(z_j'), j \le m-1 \}$ agree in law for all collections of $(z_j', j \le m-1)$.
By construction, each component of $\HH \setminus K_{\tau-}$ has fewer than $m$ points; moreover, conditioned on $(K_t, t\le \tau)$,
 the remainder of the branching radial SLE ($\kappa, \kappa -6$) consists of an independent radial SLE($(\kappa, \kappa -6$) in each component of $\HH \setminus K_{\tau-}$ not surrounded by a loop---by inductive hypothesis, the loops traced by such a process agree in law with the $\tilde \gamma(z_j)$.

\medbreak

 We now proceed to prove Lemma \ref{l.SLECLEmix}.

\begin{proof}
First, we claim that it is enough to prove (for each given $\delta>0$) the statement of Lemma \ref{l.SLECLEmix} where the time $\tau$
is replaced by $\tau_\delta$ defined as follows. Let first $\sigma_\delta$ denote the first time at which at least one of the images of the $z_j$ gets within distance
$\delta$ of $U_t$. Note that typically, at time $\sigma_\delta$, the Bessel process $X$ will be in the middle of some excursion away from $0$.
  Then, let $\tau_\delta$ be the first time $t$ at which either:
\begin{enumerate}
 \item The SLE($\kappa, \kappa-6$) completes a  loop that surrounds some $z_j$.
 \item The SLE completes the loop it is tracing at $\sigma_\delta$ (i.e., $t \geq \sigma_\delta$ and $X_t =0$).
\end{enumerate}
Indeed,  we know that each of the loops $\gamma(z_j)$ is at  finite distance from the $z_j$. Hence, almost surely, the conformal radius of $H_{\tau_\delta-}$ from each of the $z_j$'s remains bounded as $\delta \to 0$; thus, in the Carath\'eodory sense seen from $z_j$, the domain surrounding $H_{\tau_\delta-}$  tends to a limit as $\delta \to 0$ almost surely (note that the $\tau_\delta$ increase as $\delta$ decreases).  It is therefore clear that the law of the $\tilde \gamma$'s (conditioned on $K_{\tau_\delta-}$)
tends to a limit as $\delta \to 0$ and also that (by the arguments of Section \ref{s.severalpoints}) this limit is indeed that of an independent CLE in each
of the component of $H_{\tau_\delta-}$.

Now it remains only to prove the claim for these $\tau_\delta$.  To do this, it suffices to return to the modified radial $\epsilon$-exploration scheme (the one including the $c_n^\eps$ ``rotations'') that we have just defined and studied, and use the fact that for each fixed $\delta$, up to $\tau_\delta$, it approximates the continuous mechanism corresponding to the SLE($\kappa, \kappa-6$) excursions.
\end{proof}

\vfill 
\eject

\part*{Part two: construction via loop-soups}

\addcontentsline{toc}{part}{Part two: construction via loop-soups}

\section{Loop-soup percolation}
\label {S.2bis}
We now begin the second part of the paper, focusing on properties of clusters of Brownian loops.
The next three sections are structured as follows. We first study some properties of the Brownian loop-soup and of the clusters it defines. The main result of the present section is that when $c$ is not too large (i.e.\ is subcritical), the outer boundaries of outermost Brownian loop-soup clusters form a random collection of disjoint
simple loops that does indeed satisfy the conformal restriction axioms.
By the main result of the first part, this implies that they are CLE$_\kappa$ ensembles for some $\kappa$. In Section \ref {S.3}, we compare how loop-soups and SLE$_\kappa$ curves behave when one changes the domain that they are defined in, and we deduce from this the relation between $\kappa$ and $c$ in this subcritical phase.
In Section \ref {S.4}, we show that if the size of the clusters in a Brownian loop-soup satisfy a certain decay rate property, then the corresponding $c$ is necessary {\em strictly} subcritical. This enables to show that the loop-soup corresponding to $\kappa= 4$ is the only possible critical one, and completes the identification of all CLE$_\kappa$'s for $\kappa \in (8/3, 4]$ as loop-soup cluster boundaries.

We will in fact directly use SLE results on only three distinct occasions: we use the standard SLE restriction properties from \cite {LSWr} and the description of CLE in terms of SLE excursions in Section \ref {S.3}, and we use an estimate about the size of an SLE$_\kappa$ excursions
 in Section \ref {S.4}.

Recall that we consider a Brownian loop-soup $\Gammabis$ in $\U$ with intensity $c$ (which is in fact a random countable collection of simple loops because we take the outer boundaries of Brownian loops). Note that almost surely, for any two loops in the loop-soup, either the two loops are disjoint or their interiors are not disjoint. We say that two loops $l$ and $l'$ in $\Gammabis$ are in the same cluster of loops if one can find a finite chain of loops $l_0, \ldots, l_n$ in $\Gammabis$ such that $l_0= l$, $l_n=l'$ and $l_j \cap l_{j-1} \not= \emptyset$ for all $j \in \{1, \ldots, n \}$. We then define $\overline \Gamm$ to be the family of all closures of loop-clusters. Finally, we let $\Gamma$ denote the family of all outer boundaries of outermost elements of $\overline \Gamm$ (i.e. elements of $\overline \Gamm$ that are surrounded by no other element of $\overline \Gamm$).

The goal of this section is to prove the following proposition:

\begin{proposition} \label{p.satisfiesconformal} There exists a positive constant $c_0$ such that
 for all $c$ in $(0, c_0)$, the set $\Gamma$ satisfies the conformal restriction axioms, whereas when $c$ is (strictly) greater than $c_0$, $\Gammabis$ has only one cluster almost surely.
\end{proposition}

Throughout this section, we will use neither SLE-type results nor results derived earlier in the paper.  We will focus on properties of the collection $\overline \Gamm$ of (closures of) the clusters defined by the loop-soup $\Gammabis$. The proposition will follow immediately from a sequence of six lemmas that we now state and prove.

It is easy to see (and we will justify this in a moment) that when $c$ is very large, there almost surely
exists just one cluster, and that this cluster is dense in $\U$, i.e., that
almost surely, $\overline \Gamm = \{ \overline \U \}$.

\begin {lemma} \label{l.confrest}
Suppose that $\mathbb P_c (\overline \Gamm = \{ \overline \U \}) < 1 $.
Let $U \subset \U$ denote some open subset of $\U$, and define $U^*$ to be the set obtained by removing from $U$ all the
(closures of) loop-soup clusters $\overline C$ that do not stay in $U$. Then, conditionally on
$U^*$ (with $U^* \not= \emptyset$), the set of loops of $\Gammabis$ that do stay in $U^*$ is distributed like a Brownian loop-soup in $U^*$.
\end {lemma}
\begin {figure}[htbp]
\begin {center}
\includegraphics [width=2in]{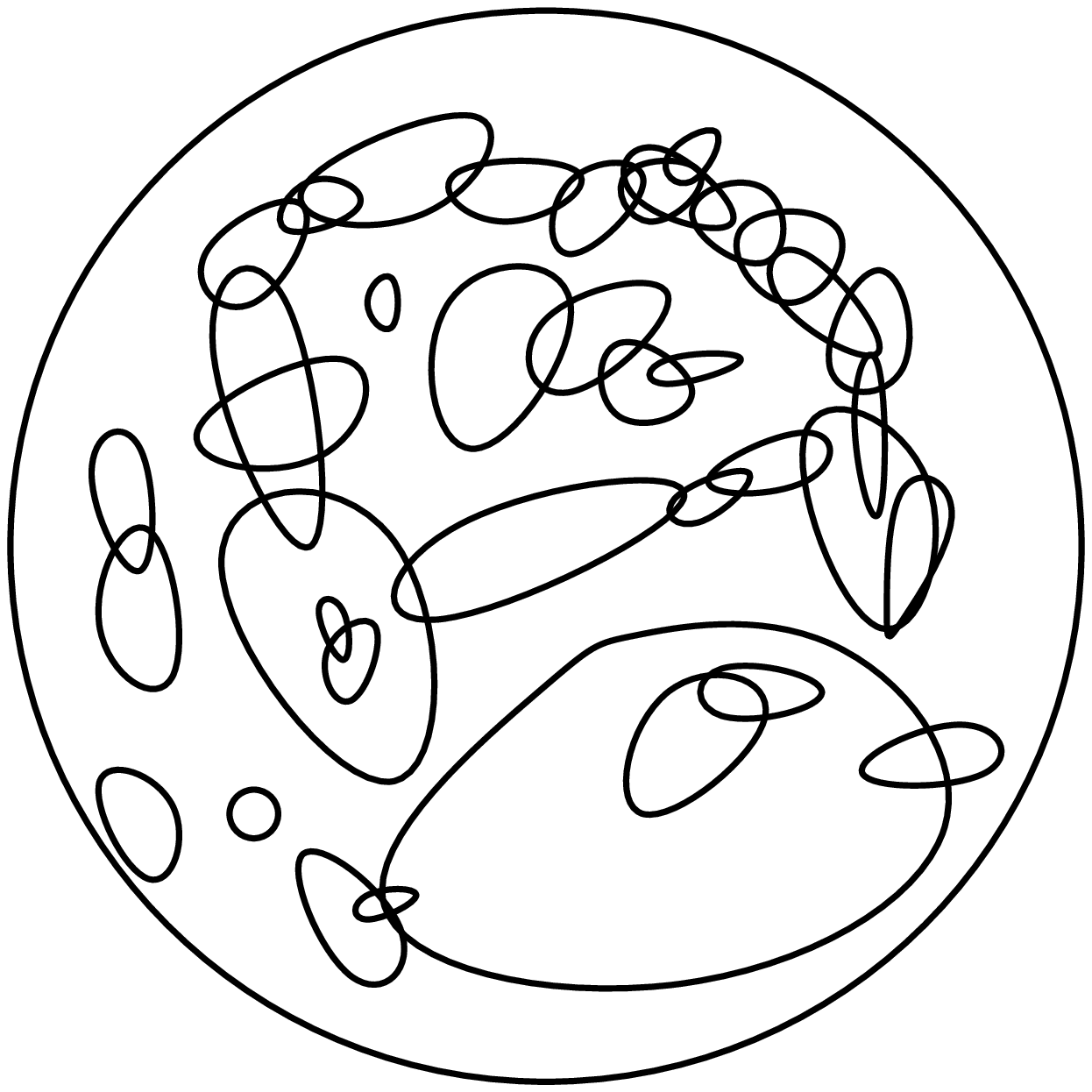}
\includegraphics [width=2in]{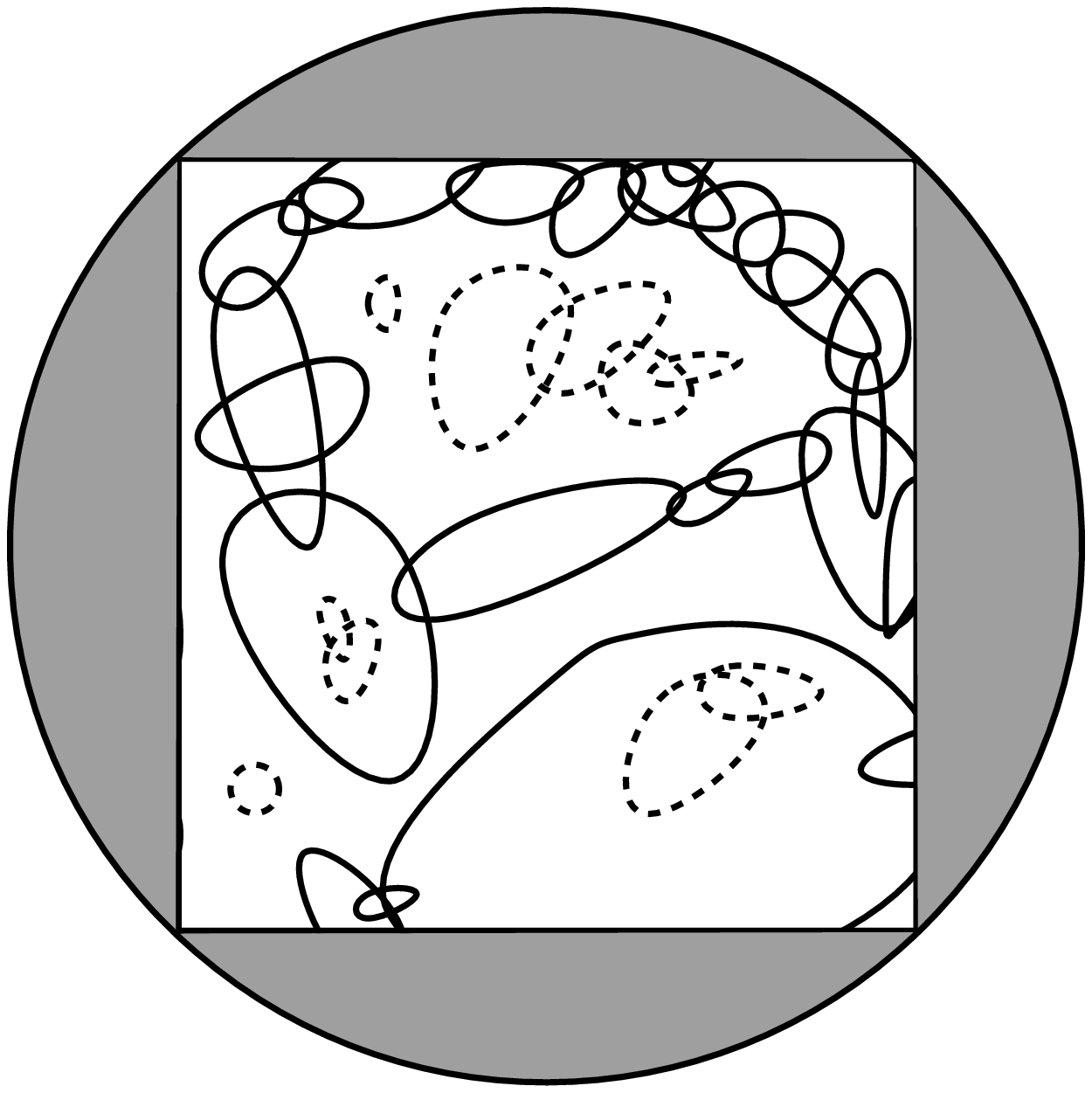}
\caption {Loop-soup clusters that stay in the rectangle $U$ are dashed (sketch)}
\end {center}
\end {figure}
Note that we have not yet proved at this point that in this case, $\Gamma$ is a locally finite collection of disjoint simple loops (this fact will be proved later in this section).


\begin {proof}
Let us define for any $n \ge 1$, the set $U_n^*=U_n^* (U^*)$ to be the
 largest union of dyadic squares of side-lengths $2^{-n}$ that is contained in $U^*$ (note that this is a deterministic function of $U^*$).
For each $n \ge 1$, and for each union  $V_n$ of such dyadic squares the loop-soup restricted to $V_n$ is independent of the event $\{ U_n^* = V_n \}$.
It implies immediately that conditionally on $U^*$, the set of loops that do stay in $U_n^*$ is distributed like a Brownian loop-soup in $U_n^*$.
Since this holds for all $n$, the statement of the lemma follows.
\end {proof}

\begin{lemma} \label{l.biggerc-onecluster}
Suppose that $\mathbb P_c ( \overline \Gamm = \{ \overline \U \} ) < 1$ and that there is a $\mathbb P_c$ positive probability that $\overline \Gamm$ contains an element intersecting
the boundary of $\U$.  Then for all positive $c'$, $\mathbb P_{c+c'} ( \overline \Gamm = \{ \overline \U \} ) = 1$.
\end{lemma}

\begin{proof}
Assume the hypotheses of the lemma, and
let $A_1$ be the union of all elements of $\overline \Gamm$ that intersect some prescribed boundary arc $\partial$ of $\U$ of positive length.
By invariance under rotation, $\mathbb P_c ( A_1 \not= \emptyset ) > 0$.
Using the same
 argument as in the previous lemma, we get that if $U$ is a fixed open set such that $\overline U \subset \U$, then conditioned on the event $\overline U \cap \overline A_1 = \emptyset$,
the law of the set of loops in $\Gammabis$ that are contained in $U$ is the same as its original law (since changing the set of loops within $U$ has no effect on $A_1$).
Since this holds for any $U$, conformal invariance of the loop soup implies that conditioned on $A_1$, the law
of the elements of $\Gammabis$ that do not intersect $\overline A_1$ is that of independent copies of $\Gammabis$ conformally mapped to
each component of $\U \setminus \overline A_1$.
Note that this in fact implies that the event that $A_1$ is empty is independent of $\Gammabis$, and hence has probability zero or one (but we will not really need this).

The conformal radius $\rho_1$ of $\U \setminus \overline A_1$ seen from $0$ has a strictly positive probability to be smaller than one.
We now iterate the previous procedure: We let $U_2$ denote the connected component of $\U \setminus \overline A$ that contains the origin. Note that
the harmonic measure of $\partial_2:= \partial \cup A_1$ at $0$ in $U_1$ is clearly not smaller than the harmonic measure of $\partial$ in $\U$ at $0$ (a Brownian motion started at the origin that exits $\U$ in $\partial$ with necessarily exit $U_1$ through $\partial_2$). We now consider the loop-soup in this domain $U_2$, and we let $A_2$ denote the union of all loop-soup clusters that touch $\partial_2$. We then iterate the procedure, and note that the conformal radius of $U_n$ (from $0$) is dominated by a product of i.i.d. copies of $\rho_1$.

This shows  that for any positive $\delta$  one can almost surely find in $\Gamma$ a finite sequence of clusters $\overline C_1, \ldots, \overline C_k$, such that
$d (C_j , C_{j+1}) = 0$ for all $j < k$,  such that $C_1$ touches $\partial$ and $d ( C_k, 0) \le \delta$.
By conformal invariance, it is easy to check that the same is true if we replace the origin by any fixed point $z$.
Hence, the statement holds almost surely, simultaneously for all
points $z$ with rational coordinates, for all rational $\delta$, and all boundary arcs of $\partial \U$ of positive length.

Note that almost surely, each loop of the loop soup surrounds some point with rational coordinates. We can therefore conclude (see Figure \ref {loops1000}) that almost surely,
any two clusters in $\overline \Gamm$ are ``connected'' via a finite sequence of adjacent clusters in $\overline \Gamm$.
\begin {figure}[htbp]
\begin {center}
\includegraphics [width=2.5in]{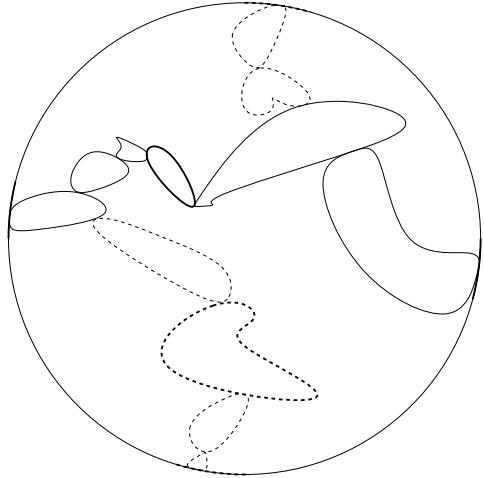}
\caption {The dark solid loop is part of a crossing of (light solid) loops from the left boundary segment to the right boundary segment.  The dark dotted loop is part of a crossing of loops from the lower boundary segment to the upper boundary segment.}
\label{loops1000}
\end {center}
\end {figure}

If we now augment $\Gammabis$ by adding an independent Brownian loop-soup $\Gammabis'$ of
intensity $c'$ for any given positive $c'$, the new loops will almost surely join together any two adjacent clusters of $\overline \Gamm$ into
a single cluster (this is just because any given point $z' \in \U$ -- for instance one chosen contact point between two adjacent clusters of $\overline \Gamm$ -- will almost surely be surrounded by infinitely many small loops of $\Gammabis'$).
Furthermore, for an analogous reason, almost surely, any loop of $\Gammabis'$ intersects some loop of $\Gammabis$.
It follows that for all $c'>0$, $\mathbb P_{c+c'}$ almost surely, there exists just one single cluster, i.e., $\overline \Gamm = \{ \overline \U \}$.
\end{proof}

\begin{lemma} \label{l.firstpcdef}
There is a critical constant $c_0 \in [0, \infty]$ such that
\begin{enumerate}
 \item If  $c > c_0$, then $\mathbb P_c (\overline \Gamm = \{ \overline \U \}) = 1$.
 \item If  $c  \in (0, c_0)$, then $\mathbb P_c$ almost surely \begin{enumerate}
 \item[(a)] $\overline \Gamm$ has infinitely many elements.
 \item[(b)] No element of $\overline \Gamm$ intersects the boundary of $\U$.
 \item[(c)] No two elements of $\overline \Gamm$ intersect each other. \end{enumerate}
\end{enumerate}
\end{lemma}

\begin{proof}
Suppose that $\mathbb P_c ( \overline \Gamm = \{ \overline \U \} ) <1$;
if there is a $\mathbb P_c$ positive probability that two elements of $\overline \Gamm$ intersect each other, then (applying Lemma \ref {l.confrest} to some  $U$
that contains one but not the other with positive probability)  we find that there is a positive probability that an element of $\overline \Gamm$ intersects $\partial \U$.

Also, if $c>0$ and $\overline \Gamm$ has only finitely many elements with positive probability, then (with the same probability) at least one of these elements must intersect $\partial \U$ (since the loops of $\Gammabis$ are dense in $\U$ a.s.).

Thus, Lemma \ref{l.biggerc-onecluster} implies that if $c_0$ is the supremum of $c$ for which (a), (b), and (c) hold almost surely, then $\Gammabis$ has only one cluster $\mathbb P_c$ almost surely whenever $c>c_0$.
\end{proof}

We say $c$ is {\em subcritical} if the (a), (b), and (c) of Lemma \ref{l.firstpcdef} hold $\mathbb P_c$ almost surely. We will later show that $c_0$ is subcritical,
but we have not established that yet. We remark that the proof of Lemma \ref {l.firstpcdef} shows that in order to prove that some $c$ is subcritical, it suffices to check (b).

\begin{lemma}
\label{pc}
The $c_0$ of Lemma \ref{l.firstpcdef} lies in $(0,\infty)$.  Moreover, when $c>0$ is small enough, there are $\mathbb P_c$
almost surely continuous paths and loops in $D$ that
intersect no element of $\Gammabis$.
\end{lemma}

\begin{proof}
We first prove the latter statement: for small $c>0$ there exist almost surely simple paths crossing $\U$ that intersect no element of $\Gammabis$.  This
will also imply $c_0 >0$.  To this end we will couple the loop-soup with a well-known fractal percolation model.
The argument is similar in spirit to the one for multi-scale Poisson percolation in \cite{MR1409145}, Chapter 8.

Consider the unit square $D=(0,1)^2$ instead of $\U$. For each $n$, we will divide it into $4^n$ (closed) dyadic squares
of side-length $2^{-n}$. To each such square $C$, associate a Bernoulli random variable $X(C)$ equal to
one with probability $p$. We assume that the $X(C)$ are independent.
Then, define
\begin{equation} \label{e.mdef}
 M= [0,1]^2 \setminus \bigcup_{C \ : \ X(C)=0} C .
\end{equation}
This is the fractal percolation model introduced by Mandelbrot in \cite{MR665254} (see also the book \cite{MR1409145}).
It is very easy to see that the area of $M$ is almost surely zero as soon as $p<1$.
Chayes, Chayes and Durrett \cite{MR931500}
have shown that this model exhibits a phase transition:
There exists a $p_c$, such that for all $p \ge p_c$, $M$ connects the left and right sides of the
unit square with positive probability, whereas for all $p < p_c$, this is a.s.\ not the case
(note that in fact, if $p \le 1/4$, then $M$ is almost surely empty by a standard martingale argument).
Here, we will only use the fact that for $p$ large enough (but less than one), $M$ connects the two opposite sides of the
unit square with positive probability.  We remark that the proof in \cite{MR931500} actually gives (for large $p$) a positive probability that there
exists a continuous path from the left to right side of the unit square in $M$ that can be parameterized as $t \to (x(t),y(t))$ where $t \in [0,1]$
and $x(t)$ is non-decreasing. It also shows (modulo a straightforward FKG-type argument) that $M$ contains loops with positive probability.

Now, let us consider a loop-soup with intensity $c$ in the unit square.
For each loop $l$, let $d(l) \in (0,1)$ denote its $L^1$-diameter (i.e., the maximal variation of the $x$-coordinate or of the $y$-coordinate), and define $n(l) \ge 0$
 in such a way that $d(l) \in [2^{-n-1}, 2^{-n})$.
Note that $l$ can intersect at most 4 different dyadic squares with side-length $2^{-n}$. We
can therefore associate
in a deterministic (scale-invariant and translation-invariant)
manner to each loop $l \subset (0,1)^2$, a dyadic site $s = (j2^{-n(l)}, j'2^{-n(l)}) \in (0,1)^2$  such that
$l$ is contained in the square $S_l$ with side-length $2 \times 2^{-n(l)}$ and bottom-left corner at $s$. Note that $S_l \subset (0,2)^2$.

We are first going to replace all loops $l$ in the loop-soup by the squares $S_l$.
This clearly enlarges the obtained clusters. By the scale-invariance and the Poissonian character of
the loop-soup, for each fixed square
\begin{equation}
\label{e.stype} S=[j2^{-m}, (j+2) 2^{-m}] \times [j' 2^{-m}, (j'+2) 2^{-m}] \subset [0,2]^2,
\end{equation}
the probability that there exists no loop $l$ in the loop-soup such that
$S_l=S$ is (at least) equal to $\exp ( - bc)$ for some positive constant $b$ (that is independent of $m$).
Furthermore, all these events (when one lets $S$ vary) are independent.

Hence, we see that the loop-soup percolation process is dominated by a variant of Mandelbrot's fractal percolation model:
let $\tilde X$ denote an independent percolation on squares of type \eqref{e.stype},
with each $\tilde X(S)$ equal to $1$ with probability $\tilde p = \exp (-bc)$ and define the random compact set
\begin{equation}
\tilde M= [0,2]^2 \setminus \bigcup_{S \ : \ \tilde X(S)=0} S.
\end{equation}
Note that in the coupling between the loop-soup and $\tilde M$ described above, the distance between $\tilde M$ and each fixed loop in the loop-soup
is (strictly) positive almost surely.  In particular, $\tilde M$ is contained in the complement of the union of the all the loops (and their interiors).

We now claim that this variant of the percolation model is dominated by Mandelbrot's original percolation model with a larger (but still less than 1) value of $p = p(\tilde p)$ that we will choose in a moment.
To see this, let $\hat X$ be a $\hat p$-percolation (with $\hat p = p^{1/4}$) on the set of $(C,S)$ pairs with $C$ a dyadic square and $S$ a square comprised of $C$ and three
of its neighbors.  Take
$$X(C) = \min_{S\ : \ C \subset S} \hat X(C,S) \hbox { and } \tilde X(S) = \max_{C \ : \ C \subset S} \hat X(C,S).$$
Clearly $X(C)$ is a Bernoulli percolation with parameter $p = \hat p^4$ and $\tilde X(S)$ is a Bernoulli percolation with parameter $1-(1-\hat p)^4$.
Let us now choose $p$ in such a way that $\tilde p  = 1 - (1- \hat p)^4$. Note that $p$ tends monotonically to $1$ as $\tilde p$ tends to $1$.
Hence, by taking $\tilde p$ sufficiently close to $1$, i.e.\ $c$ sufficiently small, we can ensure that $p$ is as close to $1$ as we want (so that $M$
contains paths and loops with positive probability). But
in our coupling, by construction we have
$$X(C) \leq \min_{C \subset S} \tilde X(S),$$
 and thus $M \subset \tilde M$.

We have now shown that $c_0 > 0$, but we still have to show that $c_0 < \infty$.  We use a similar coupling with the fractal percolation
model. For any dyadic square that does not touch the boundary of $[0,1]^2$, we let $X(C)$ be $0$ if $C$ is surrounded by a loop in $\Gammabis$ that is contained in the set of eight neighboring dyadic squares to $C$ (of the same size).
The $X(C)$ are i.i.d.\ (for all $C$ whose eight neighbors are contained in $D$), and have a small probability (say smaller than $1/4$) of being $1$ when $c$ is taken sufficiently large.
We now use the fact, mentioned above, that if $p \le 1/4$, then $M$ is almost surely empty; from this we conclude easily that almost surely, for each $C$ (whose incident neighbors
are in $D$) every point in $C$ is surrounded by a loop in $\Gammabis$ almost surely.  It follows immediately that almost surely, every point in $D$ is surrounded by a loop in $\Gammabis$.
This implies that almost surely all loops of $\Gammabis$ belong to the same cluster (since otherwise there would be a point on the boundary of a cluster that was not surrounded by a loop).
\end{proof}

\begin{lemma} \label{exponentialcrossingdecay}
If $c$ is sub-critical, the probability that there are $k$
disjoint chains of loops in $\overline \Gamm$ crossing from the inside to the outside of a fixed annulus decays exponentially in $k$.
\end{lemma}

\begin{proof}
We know that $c$ is subcritical, so that for all $r<1$, the probability that there exists a crossing of the annulus $\{ z \ : \ r < | z| < 1 -1 /n \}$
by a chain of loops goes to $0$ as $n\to \infty$. Hence, there exists $r_1 \in (r, 1)$ such that the probability that there is a single crossing of $\{ z \ : \ r < | z | < r_1 \}$ is strictly smaller than one.
Hence, it follows easily that if we consider any given annulus, $\{ z \ : \  r < | z - z_0 | < r' \} \subset \U$, there exists a positive probability that no cluster of the loop-soup crosses the annulus (just consider the two independent events of positive probability that the loop-soup restricted to $\{ z \ : \ |z -z_0 | < r'/r_1\}$ contains no crossing of the annulus, and that no loop intersects both the circles of radius $r'$ and $r'/r_1$ around $z_0$).  In other words, the probability that there exists a crossing of the annulus is strictly smaller than one.
 The result then follows from the BK inequality for Poisson point processes (see for instance \cite {vdB} and the references therein).
\end{proof}

As a consequence (letting $k \to \infty$), we see that for each fixed annulus, the probability that it is crossed by infinitely many disjoint chains of loops is zero.
We are now ready to state and prove the final lemma of this section:

\begin{lemma} \label{l.simpleloops}
If $c$ is sub-critical, the set $\overline \Gamm$ is almost surely locally finite. Moreover, when $c$ is sub-critical, then almost surely, all elements of $\Gamma$ are continuous simple loops.
\end{lemma}
\begin{proof}
If the  set $\overline \Gamm$ is not locally finite (i.e. if there exists some $\epsilon$ such that there are infinitely many clusters of diameter at least $\epsilon$), then for some positive $\epsilon$ there exists a point $z \in \overline \U$ and a sequence $\overline C_n$ of
elements of $\overline \Gamm$ of size greater than $\epsilon$ such that $d (z, C_n) \to 0$ as $n \to \infty$.
 Hence any annulus with outer radius less than $\epsilon/2$ that surrounds $z$ will have infinitely many crossings by disjoint chains of loops.
By Lemma \ref{exponentialcrossingdecay},
the probability that such a point exists is zero (this result just follows from the lemma by considering a countable collection of annuli such that each point in $\U$ is surrounded by arbitrarily small annuli from this collection).

Note that by construction, no element of $\overline \Gamm$  can have a cut point (recall that almost surely no two loops of $\Gammabis$ intersect at only a single point --
and that the elements of $\Gammabis$ are all {simple} loops).

It therefore remains only to check that the outer boundary of a cluster is a continuous loop. Several approaches are possible to justify this. Let us first show that
the outer boundary of a cluster can be viewed as the outer boundary of a single continuous (non-simple) loop.
Let $(\eta_k)$ be any given countable collection of simple loops in $\U$ such that the graph whose vertices are the loops (with two loops connected if they intersect) is connected.  Assume that there are at most finitely many loops above any given diameter, and at most finitely many disjoint crossings by chains of loops of any given annulus (note that all this holds almost surely when $(\eta_k)$ is the family of loops in a given sub-critical loop-soup cluster).  We would like to define a single continuous loop that traces exactly through all of the points on the closure of the union of the $\eta_k$.  It is not so surprising that this is possible, but requires some justification:

Let $T$ be a spanning tree of the adjacency graph on the $\eta_k$ and fix a root vertex.  Now, relabeling appropriately if necessary (for instance by choosing $\eta_1$ to be the root loop and then inductively choose $\eta_k$ to be the largest child in $T$ of $\eta_1, \ldots, \eta_{k-1}$), assume that $\eta_1, \eta_2, \ldots$ are enumerated in such a way that for every $k>1$, the loop $\eta_k$ has one of $\eta_1, \eta_2, \ldots, \eta_{k-1}$ as a parent in $T$.  We can inductively define a parameterization of each of the $\eta_k$, starting and ending at a point $x_k$ on $\eta_k$ and taking $2^{-k}$ time, as follows. Pick a point on $\eta_1$ and traverse the loop in time $2^{-1}$ starting and ending at an arbitrary point $x_1$ on the loop.  Given the parameterizations up to $k-1$, let $x_k$ be the location of the first place that the parameterization of the parent loop of $\eta_k$ hits $\eta_k$.  Then choose an arbitrary parameterization of $\eta_k$ starting and ending at $x_k$ and taking $2^{-k}$ time.

Now we can define a single loop $L$ that traverses all of the $\eta_j$ in a total time $1 = \sum_{j=1}^\infty 2^{-j}$:  We define a sequence of loops $L_i$ converging to $L$ uniformly.  First $L_1$ is the loop that traverses $\eta_1$ in order except that each time it hits a point which is an $x_k$ for some $k>1$ it waits at that point for a time equal to the sum of all the lengths of $\eta_k$ and all of its descendants.  Thus $L_1$ traverses $\eta_1$ in unit time.  We define $L_2$ the same way except that at each one of the $x_k$ that the $L_1$ loop paused at we actually traverse the loop $\eta_k$ (pausing at each $x_\ell$ on that loop for an amount of time equal to the sum of the lengths of $\eta_\ell$ and all of its descendants). Inductively we similarly define $L_k$, which traverses all loops $\eta_1, \ldots, \eta_k$, and note that the finite diameter and finite number of annulus crossing conditions imply that $L_1, L_2, \ldots$ converge uniformly to a limiting continuous loop $L$ whose range contains the range of each of $L_1, L_2, \ldots$, hence each $\eta_j$. The range of $L$ is therefore exactly the closure of $\cup_k \eta_k$.

Now, it is easy to see that the outer boundary of a continuous loop in the plane is necessarily a continuous loop (this is for instance explained and used the proof of Theorem 1.5(ii) in \cite {BurdzyLawler}, page 1003).
In the case of our loop-soup clusters, we know that this loop is almost surely simple because the cluster has no cut-point.
\end{proof}

Proposition \ref{p.satisfiesconformal} now follows from the results proved in this section: Lemma \ref{pc} gives the existence of $c_0$, Lemma \ref{l.simpleloops} implies that the loops are simple and locally finite a.s. (which in particular also implies the existence of outermost clusters), and Lemma \ref{l.confrest} yields the restriction property.

\section {Relation between $c$ and $\kappa$}
\label {S.3}

The main result of our Markovian characterization, combined with Proposition \ref {p.satisfiesconformal}
now implies the following:

\begin{corollary} \label{CLEcor}
If $c$ is sub-critical, then the set $\Gamma$ is a CLE$_\kappa$ for some $\kappa \in (8/3, 4]$.
In other words, for such $c$, the $\Gamma$
is equivalent in law to the loop ensemble constructed in \cite{Sh} via branching SLE($\kappa, \kappa-6$).
\end{corollary}

We will now identify $\kappa$ in terms of $c$.
\begin{proposition} \label{p.ckappa} For all subcritical $c$, the set $\Gamma$ is in fact a CLE$_\kappa$ with $c=(3\kappa-8)(6-\kappa)/2\kappa$.
\end{proposition}

Note that in our proof, we will not really use the description of CLE$_\kappa$ via branching SLE($\kappa, \kappa -6$). We will only use the description of the conformal loop-ensemble via its pinned loop measure, as described in the earlier sections:
Suppose that $\Gamma$ is a random loop ensemble that satisfies the conformal restriction axioms. Consider its version in the upper half-plane $\H$, and consider the loop
$\gamma (i)$ of $\Gamma$ that surrounds $i$. Let us now consider the law of $\gamma(i)$ conditioned on the event $\{d( 0, \gamma (i)) \le \eps \}$. We have seen that this law converges when $\eps \to 0$ to some limit
$P^i$, and furthermore that for some $\kappa \in (8/3, 4]$, $P^i$ is equal to an SLE-excursion law $P^{i, \kappa}$ that we will describe in the next paragraph. Furthermore, when this is the case, it turns out that the entire family $\Gamma$ is a CLE$_\kappa$ for this value of $\kappa$.

Consider an SLE$_\kappa$ in the upper half-plane, started from the point $\eps >0$ on the real axis to the point $0$ (on the real axis as well). Such an SLE path will typically be very small when $\eps$ is small. However, one can show that the limit when $\eps \to 0$ of (the law of) this SLE, conditioned on disconnecting $i$ from $\infty$ in the upper half-plane exists. This limit (i.e., its law) is what we call $P^{i, \kappa}$.

Conformal invariance of SLE$_\kappa$ enables to define an analogous measure in other simply connected domains. Suppose for instance that
$H = \{ z \in \H \ : \ | z| < 3 \}$. We can again consider the limit of the law of SLE$_\kappa$ in $H$ from $\eps$ to $0$, and conditioned to disconnect
$i$ from $3i$. This limit is a probability measure $P_H^{i, \kappa}$ that can also be viewed as the image of $P^{i , \kappa}$ under the conformal map
$\Phi$ from $\H$ onto $H$ that keeps the points $0$ and $i$ invariant.  Note that the same argument holds for other choices of $H$, but choosing this particular one will be enough $H$ to identify the relationship between $c$ and $\kappa$.

One can use SLE techniques to compare ``directly'' the laws of an SLE $\gamma$ from $\eps$ to $0$ in $\H$ and of an SLE $\gamma'$ also from $\eps$ to $0$ in $H$.
More precisely, the SLE martingale derived in \cite {LSWr} show that the Radon-Nikodym of
the former with respect to the latter is a multiple of
$$ \exp  ( - c L ( \gamma, \H \setminus H ; \H ))$$
on the set of loops $\gamma$ that stay in $H$ (recall that $L(A,A';D)$ denotes the $\mu$-mass of the set of Brownian loops in $D$ that intersect both $A$ and $A'$),
where
$$  c = c (\kappa) =\frac {(3\kappa-8)(6-\kappa)}{2\kappa}.$$
This absolute continuity relation is valid for all $\eps$, and it therefore follows that it still holds after passing to the previous limit $\eps \to 0$, i.e., for the
two probability measures $P^{i, \kappa}$ and $P^{i, \kappa}_H$. This can be viewed as a property of $P^{i, \kappa}$ itself because $P^{i, \kappa}_H$ is the conformal image of $P^{i, \kappa}$ under $\Phi$.

Note that the function $\kappa \mapsto c(\kappa)$ is strictly increasing on the interval $(8/3, 4]$. Only one value of $\kappa$ corresponds to each value of $c \in (0,1]$.
Hence, in order to identify the value $\kappa$ associated to a family $\Gamma$ of loops satisfying the conformal restriction axioms, it suffices to check that the probability measure $P^i$ satisfies the corresponding absolute continuity relation for the corresponding value of $c$.

\begin{proof}
Suppose that $c$ is subcritical, that $\Gammabis$ is a loop-soup with intensity $c$ in $\H$, and that $\Gamma$ is the corresponding family of disjoint loops (i.e., of outer boundaries of outermost clusters of loops of $\Gammabis$). For the semi-circle $H$ defined above, we denote by $\Gammabis'$ the loops of $\Gammabis$ that stay in $H$, and we denote by $\Gamma'$ the corresponding family of disjoint loops (i.e. outermost boundaries of outermost clusters). Note that $\Gammabis'$ is a loopsoup in $H$.
Let $\gamma(i)$ denote the loop in $\Gamma$ that surrounds $i$, and let $\gamma'(i)$ denote the loop in $\Gamma'$ that surrounds $i$.
Note that if $\gamma(i) \not= \gamma'(i)$, then necessarily $\gamma(i) \notin H$. Furthermore, in order for
$\gamma (i)$ and $\gamma'(i)$ to be equal it suffices that:
\begin {itemize}
 \item $\Gammabis \setminus \Gammabis'$ contains no loop that intersects $\gamma'(i)$ (note that the probability of this event-- conditionally on $\gamma'(i)$ --  is
$\exp ( - c L ( \gamma'(i), \H \setminus H ; \H )$).
\item One did not create another disjoint cluster that goes ``around'' $\gamma'(i)$ by adding to $\Gammabis'$ the loops of $\Gammabis$ that do not stay in $H$. When $\gamma'(i)$ already intersects the disk of radius $\eps$, the conditional probability that one creates such an additional cluster goes to $0$ as $\eps \to 0$ (using the BK inequality for Poisson point processes, as in the proof of Lemma \ref{exponentialcrossingdecay}).
\end {itemize}
If we now condition on the event that $\gamma'(i)$ intersects the disk of radius $\eps$ and let $\eps \to 0$, it follows that under the limiting law $P^i$ satisfies the absolute continuity relation that we are after: On the set of curves $\gamma$ that stay in $H$, the Radon-Nikodym derivative of $P^i$ with respect to the measure defined directly  in $H$ instead of $\H$ is $\exp ( - c L ( \gamma, \H \setminus H ; \H )$).
Hence, one necessarily has $P^i = P^{ i, \kappa}$ for $c=c(\kappa)$.
\end {proof}

This identification allows us to give a short proof of the following fact, which will be instrumental in the next section:

\begin {proposition}
 \label {p.cosubcritical}
$c_0$ is subcritical.
\end {proposition}

Note that the standard arguments developed in the context of Mandelbrot percolation
(see \cite{MR1409145,MR931500}) can be easily adapted to the present setting in order to prove that $c_0$ is subcritical, but not in the sense we have defined it. It shows for instance easily that at $c_0$, there exist paths and loops that intersect no loop in the loop-soup, but non-trivial additional work would then be required in order to deduce that no cluster touches the boundary of the domain. Since we have this identification via SLE$_\kappa$ loops at our disposal, it is natural to prove this result in the following way:

\begin {proof}
Let $C_c$ be the outermost cluster surrounding $i$ if the loop-soup (in $\H$) has intensity $c$.
If we take the usual coupled Poisson point of view in which the set $\Gammabis = \Gammabis(c)$ is increasing in $c$ (with loops ``appearing'' at random times up to time $c$) then we have by definition that almost surely $C_{c_0} = \cup_{c < c_0} C_c$ (this is simply because, almost surely, no loop appears exactly at time $c_0$).  Let $d(c)$ denote the
Euclidean distance between $\overline C_c$ and the segment $[1, 2]$.
Clearly $d(c) > 0 $ almost surely for each $c < c_0$ and $d(c_0) = \lim_{c \to c_0-} d(c)$.

By the remark after Lemma \ref{l.firstpcdef}, we know that in order to prove that $c_0$ is subcritical, it suffices to show that $d(c_0) > 0$ almost surely (by Moebius invariance, this will imply that almost surely, no cluster touches the boundary of $\H$).
Note that $d$ is a non-increasing function of the loop-soup configuration (i.e., adding more loops to a configuration can not increase the corresponding distance $d$).
Similarly, the event $E_\eps$ that the outermost cluster surrounding $i$ does intersect the $\eps$-neighborhood of the origin is an increasing event (i.e., adding more loops to a configuration can only help this event to hold).  Hence it follows that for each $\eps$, the random variable $d(c)$ is negatively correlated with the event $E_\eps$.
Letting $\eps \to 0$, we get that (for subcritical $c$),
 the law of $d(c)$ is ``bounded from below'' by the law of the distance between the curve $\gamma$ (defined under the probability measure $P^{i, \kappa}$) and $[1,2]$, in the sense that for any positive $x$,
$$ P ( d(c) \ge  x ) \ \ge \  P^{i, \kappa} ( d ( \gamma, [1,2]) \ge x ).$$
But we also know that $c_0 \le 1$, so that $\kappa_0 := \lim_{c \to c_0-} \kappa (c) \le 4$.
It follows readily that for all $c < c_0$ and for all $x$,
$$ \lim_{c \to c_0-} P ( d(c) \ge  x ) \ \ge\  \lim_{\kappa \to \kappa_0-} P^{i, \kappa} ( d( \gamma, [1,2]) \ge  x )
\ \ge\  P^{i, \kappa_0} ( d(\gamma, [1,2] ) \ge 2x).$$
But we know that for any $\kappa \le 4$, the SLE excursion $\gamma$ stays almost surely away from $[1,2]$.
Putting the pieces together, we get indeed that
\begin {eqnarray*}
P ( d (c_0)  > 0 )  &= & \lim_{x \to 0+} P ( d(c_0) \ge  x )
\ \ge \  \lim_{x \to 0+} \lim_{c \to c_0-} P ( d (c) \ge x )  \\
&\ge&  \lim_{x \to 0+} P^{i, \kappa_0} ( d ( \gamma, [1,2] ) \ge 2x)
\ \ge \  P^{i, \kappa_0} ( d( \gamma, [1,2] ) > 0 )
 \  =   \  1 .
\end {eqnarray*}
\end {proof}

\section{Identifying the critical intensity $c_0$}
\label {S.4}

\subsection {Statement and comments}

The statements of our main results on Brownian loop-soup cluster, Theorems \ref{loopsoupsatisfiesaxioms} and
\ref{kappacorrespondence}, are mostly contained in the results of the previous section:
Corollary \ref{CLEcor}, Proposition \ref{p.ckappa}, Proposition \ref{p.cosubcritical}.  It remains only to prove the following statement:

\begin{proposition} \label{criticaltheorem}
The critical constant of Lemma \ref{l.firstpcdef} is $c_0 = 1$ (which corresponds to $\kappa = 4$, by Proposition \ref{p.ckappa}).
\end{proposition}

Propositions \ref{p.satisfiesconformal}, \ref {p.ckappa} and our Markovian characterization results already imply that we cannot have $c_0 > 1$,
since in this case the $\Gamma$  corresponding to $c \in (1, c_0)$ would give additional random loop collections satisfying the conformal axioms (beyond the CLE$_\kappa$ with $\kappa \in (8/3,4]$), which was ruled out in the first part of the paper.  It remains only to rule out the possibility that $c_0 <1$.

 The proof of this fact is not straightforward, and it requires some new notation and several additional lemmas. Let us first outline the very rough idea. Suppose that $c_0 < 1$. This means that at $c_0$, the loop-soup cluster boundaries are described with $SLE_\kappa$-type loops for $\kappa = \kappa ( c_0) < 4$. Certain estimates on SLE show that SLE curves for $\kappa < 4$ have a probability to get $\epsilon$ close to a boundary arc of the domain that decays quickly in $\epsilon$, and that this fails to hold for SLE$_4$.  In fact, we will use Corollary \ref {boundedness}
roughly shows that the probability that the diameter of the set of loops intersecting a small ball on the boundary of $\H$ is large,  decays quickly with the size of this ball when $\kappa < 4$.
  Hence, one can intuitively guess that when $\kappa < 4$, two big clusters will be unlikely to be very close, i.e., in some sense, there is ``some space'' between the clusters. Therefore, adding a loop-soup with a very small intensity $c'$
on top of the loop-soup with intensity $c_0$ might not be sufficient to make all clusters percolate, and this would contradict the fact that $c_0$ is critical.

We find it convenient in this section to work with a loop-soup in the upper half plane $\H$ instead of the unit disk.
To show that there are distinct clusters in such a union of two CLEs $\Gamma$ and $\Gamma'$,
we will start with the semi-disk $A_1$ of radius $1$ centered at the origin.
We then add all the loops in $\Gamma$ that hit $A_1$, add the loops in ${\Gamma}'$ that hit
those loops, add the loops in $\Gamma$ that hit those loops, etc. and try to show that in some sense the size of this
growing set remains bounded almost surely. The key to the proof is to find the right way to describe this ``size'', as the usual quantities
such as harmonic measure or capacity turn out not to be well-suited here.

\subsection {An intermediate way to measure the size of sets}

We will now define a generalization of the usual half-plane capacity.
Suppose that $\alpha \in (0,1]$, and that $A$ is a bounded closed subset of the upper half-plane $\H$. We define
\begin{equation} \label{hcappsi} M(A) = M_\alpha (A) := \lim_{s \to \infty} s \E ( (\Im B^{is}_{\tau(A)})^\alpha),\end{equation}
where $B^{is}$ is a Brownian motion started at $is$ stopped
at the first time $\tau(A)$ that it hits $A \cup \R$.  Note that $M_1= \hcap$ is just the
usual half-plane capacity used in the context of chordal Loewner chains, whereas $\lim_{\alpha \to 0+} M_\alpha (A)$ is the harmonic
measure of $A \cap \HH$ ``viewed from infinity''.
Recall that standard properties of planar Brownian motion imply that
the limit in (\ref {hcappsi})
 necessarily exists, that it is finite, and that for some universal constant $C_0$ and for any $r$ such that $A$ is a subset of the disk of radius $r$ centered at the origin, the limit is equal to
$C_0 r^{-1} \times \E ( (\Im B_{\tau(A)})^\alpha )$
where the Brownian motion $B$ starts at a random point $r e^{i \theta}$, where $\theta$ distributed according to the density $\sin (\theta)  d \theta / 2$ on $[0, \pi]$.

A {\em hull} is defined to be a bounded closed subset of $\H$ whose complement in $\H$ is simply connected.
The union of two hulls $A$ and $A'$ is not necessarily a hull, but we denote by $A \underline \cup A'$ the hull whose
complement is the unbounded component of $\H \setminus (A \cup A')$.

When $A$ is a hull, let us denote by $\Phi_A: \H \setminus A \to \H$ the conformal map normalized at infinity by  $\lim_{z \to \infty} \Phi_A(z) - z = 0$.
 Recall that for all $z \in \HH \setminus A$, $\Im ( \Phi_A (z) ) \le \Im ( z)$.
Then, when $A'$ is another hull, the set $\Phi_A ( A' \cap ( \H \setminus A))$ is not necessarily a hull. But we can still define the unbounded connected component of its complement in the upper half-plane, and take its complement. We call it $\Phi_{A} (A')$ (by a slight abuse of notation).

It is well-known and follows immediately from the definition of half-plane capacity that for any bounded closed $A$ and any positive $a$,
$\hcap (aA) = a^2 \hcap(A)$. Similarly, for any two $A$ and $A'$, $\hcap(A \cup A') \leq \hcap(A) +\hcap(A')$.
Furthermore $\hcap$ is increasing with respect to $A$, and behaves additively with repect to composition of conformal maps for hulls.

We will now collect easy generalizations of some of these four facts.
 Observe first that
for any positive $a$ we have
\begin{equation}\label{Mscaling} M(a A) = a^{\alpha+1} M(A). \end{equation}
Similarly, we have that for any two (bounded closed) $A$ and $A'$,
\begin {equation}
 \label{Munionsubadditive}
M(A  \cup A') \leq M(A) + M(A').
\end {equation}
This follows from the definition \eqref{hcappsi} and the fact that for each sample of the Brownian motion we have
$$(\Im B_{\tau(A  \cup A')})^\alpha = (\Im B_{\tau(A) \wedge \tau(A')})^\alpha \leq  (\Im B_{\tau(A)})^\alpha +  (\Im B_{\tau(A')})^\alpha.$$
Applying the optional stopping time theorem to the local supermartingale $(\Im B_t)^\alpha$, we know that
\begin{equation} \label{Msubset}
A \subset A' \hbox{ implies }  M(A) \leq M(A'),
\end{equation}
since $$\E\bigl((\Im B_{\tau(A) })^\alpha | B_{\tau(A')} \bigr) \leq (\Im B_{\tau(A')})^\alpha.$$
We next claim that
for any three given hulls $A'$, $A_1$ and $A_2$, we have
\begin {equation} \label{Mtriple}
 M(\Phi_{A_1 \underline \cup A_2} (A')) \le M (\Phi_{A_1} (A')).
\end {equation}
To verify the claim, note first that for $A_3= \Phi_{A_1} (A_2)$, we have $\Phi_{A_3} \circ \Phi_{A_1} = \Phi_{A_1 \underline \cup A_2}$.
Recall that $\Im \Phi_{A_3} (z) \leq \Im(z)$ for all $z \in \H$, so that in particular, $\Im (\Phi_{A_3} (B_{\sigma})) \leq \Im (B_{\sigma})$ where
$\sigma$ is the first hitting time of $\Phi_{A_1} (A')$ by the Brownian motion.  We let the starting point of the Brownian motion tend to infinity as before, and the claim follows.

\medbreak

It will be useful to compare $M(A)$ with some quantities involving dyadic squares and rectangles that $A$ intersects.  (This is similar in spirit to the estimates for half-plane capacity in terms of hyperbolic geometry given in \cite{LLN}.)  We will consider the usual hyperbolic tiling of $\H$ by squares of the form $[a 2^j, (a+1) 2^j] \times [2^j, 2^{j+1}]$, for integers $a,j$.
Let $\mathcal S$ be the set of all such squares.
For each hull $A$, we define $\mathcal S (A)$ to be the set of squares in $\mathcal S$ that $A$ intersects, and we let $\hat A$ be the union of these squares, i.e.
$$ \hat A = \cup_{S \in \mathcal S (A)} S.$$
\begin {figure}[htbp]
\begin {center}
\includegraphics [width=5in]{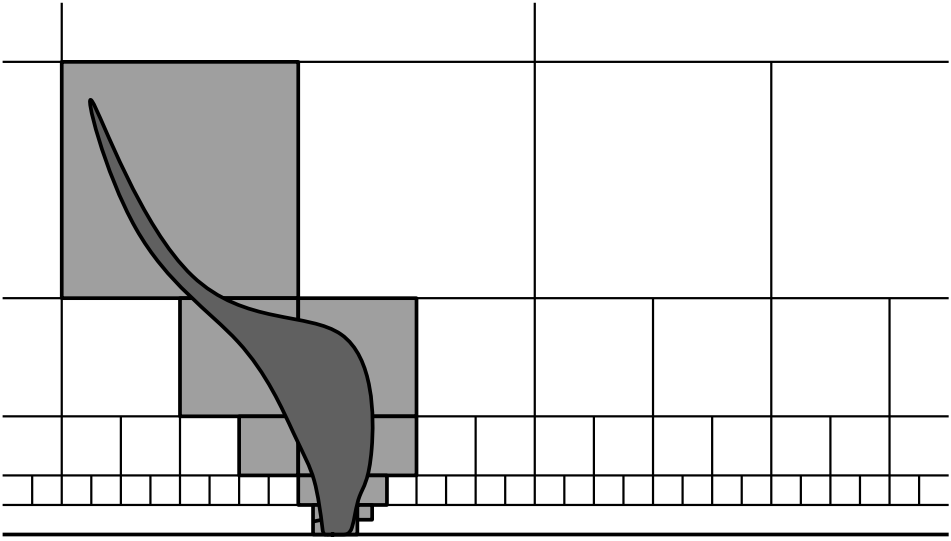}
\caption {A hull $A$ and its corresponding $\hat A$}
\end {center}
\end {figure}

\begin{lemma}\label{MhatMagree}
There exists a universal positive constant $C$ such that for any hull $A$,
$$ C M ( \hat A) \le M(A) \le M( \hat A).$$
\end{lemma}

\begin{proof}
Clearly, $M (\hat A) \ge M (A)$ by \eqref{Msubset} since $A \subset \hat A$. On the other hand,
if we stop a Brownian motion at the first time it hits $\hat A$ i.e.\ a square $S$ of
$\mathcal S(A)$, then it has a bounded probability of later hitting $A$ at a point of about the same height, up to constant factor: This can
be seen, for example, by bounding below the probability that (after this hitting time of $\hat A$) the Brownian motion makes a loop around $S$ before it hits any square
of $\mathcal S$ that
is not adjacent to $S$, which would in particular imply that it hits $A$ during that time. This probability is universal, and the lemma follows.
\end{proof}

\begin{lemma}\label{mhatsum}
There exists a universal positive constant $C'$ such that for any hull $A$,
$$
C' \sum_{S \in {\cal S}(A)} M (S)
\le
M (A)
\le
\sum_{S \in {\cal S}(A)} M (S) .
$$
\end{lemma}

\begin{proof} The right-hand inequality is obvious by (\ref{Munionsubadditive}) and Lemma \ref {MhatMagree}.
By Lemma \ref {MhatMagree}, it is sufficient to prove the result in the case where $A = \hat A$ is the union of squares in ${\cal S}$.

For each $j$ in $\Z$, we will call $\mathcal S_j$ the set of squares that are at height between $2^{j}$ and $2^{j+1}$.
We will say that a square $S = [a 2^j, (a+1) 2^j] \times [2^j, 2^{j+1}]$ in $\mathcal S_j$ is even (respectively odd) if $a$ is even (resp. odd).
We know that
\begin {equation}
\label {eqsum}
\sum_{S \in {\cal S}(A) } M (S) \le  \lim_{s \to \infty} s \E \left( \sum_{S \in {\cal S} (A) } (\Im B^{is}_{\tau(S)})^\alpha \right).
\end {equation}
To bound this expectation, we note that for each $j \in \mathbb Z$, for each even square $S \in \mathcal S_j$, and for each $z \in S$, the probability that a Brownian motion started from $z$
hits the real line before hitting any other even square in $\mathcal S_j$ is bounded from below independently from $z$, $j$ and $S$.
Hence, the strong Markov property implies that the total number of even squares in $\mathcal S_j$ hit by a Brownian motion before hitting the real line is stochastically dominated by a geometric random variable with finite universal mean (independently of the starting point of the Brownian motion) that we call $K/2$. The same is true for odd squares.

Note also that if the starting point $z$ of $B^z$ is in $S \in {\mathcal S}_j$ and if $k \ge 1$, then the probability that the imaginary part of $B$ reaches $2^{j+1+k}$ before $B$ hits the real line is not larger than $2^{-k}$. It follows from the strong Markov property that the expected number of squares in ${\mathcal S}_{j+k+1}$ that $B$ visits before exiting
$\H$ is bounded by $K\times 2^{-k}$.
It follows that for a Brownian motion started from $z$ with $2^j \le \Im (z) \le 2^{j+1}$
$$
\E \left( \sum_{S \in {\cal S} } (\Im B^z_{\tau(S)})^\alpha \right)
 \le
 \sum_{k \le 1} K (2^{j+k+1})^\alpha  +  \sum_{k \ge 2} K 2^{-k} (2^{j+k+1})^\alpha
\le  C 2^{j \alpha} \le C (\Im z)^\alpha
$$
for some universal constant $C$ (bear in mind that $\alpha < 1$ and that for $S \in {\cal S}_{j+k}$, $(\Im B^x_{\tau (S)}) \le 2^{(j+k+1)}$).
If we now apply this statement to the Brownian motion $B^{is}$ after its first hitting time $\tau (A)$ of $A \cup \R = \hat A \cup \R$, we get that for all large $s$ and for some universal positive constant $C'$,
$$
\E \left( ( \Im B^{is}_{\tau (A)})^\alpha \right)
\ge
C'
\E \left( \sum_{S \in {\cal S} (A) } ( \Im B^{is}_{\tau(S)})^\alpha \right) .
$$
Combining this with (\ref {eqsum}) concludes the proof.
\end{proof}

For each square $S = [a 2^j, (a+1) 2^j] \times [2^j, 2^{j+1}]$ of $\mathcal S$, we can define the union $R(S)$ of $S$ with all the squares of ${\mathcal S}$ that lie strictly {\em under} $S$, i.e.
$R(S) = [a 2^j, (a+1) 2^j] \times [0, 2^{j+1}]$.
Note that that scaling shows immediately that for some universal constant $C''$ and for all $S \in \mathcal S$,
 \begin{equation} \label{MR} M(R(S)) = C'' M(S). \end{equation}

\subsection {Estimates for loop-soup clusters}
Let us now use these quantities to study our random loop-ensembles.
Suppose that $\Gamma$ is the conformal loop ensemble corresponding to any given $c \in (0,c_0]$.
Given a hull $A$ we denote by $\tilde A = \tilde A ( A, \Gamma)$ the random hull whose complement is the unbounded component of the set obtained by removing from $\H \setminus A$ all the loops of $\Gamma$ that intersect $A$.  Local finiteness implies that $\tilde A$ is itself
a hull almost surely.

\begin {figure}[htbp]
\begin {center}
\includegraphics [width=6in]{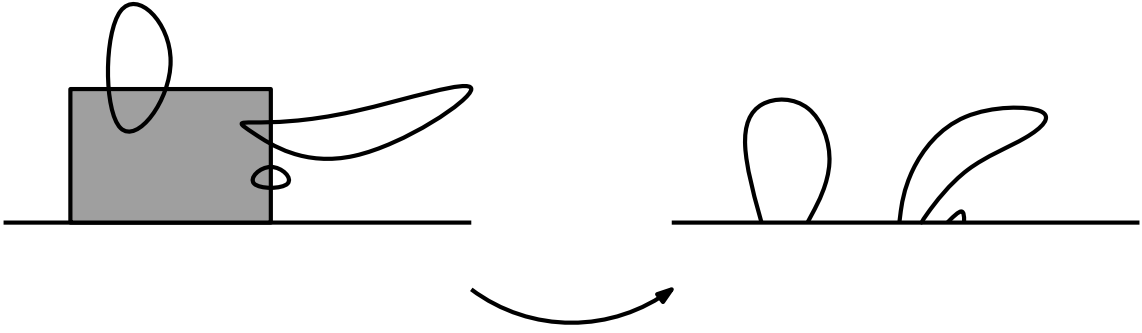}
\caption {Construction of $\Phi_A ( \tilde A)$ (sketch)}
\end {center}
\end {figure}
Now define
$$
N(A) = N_\kappa(A) := \E (M(\Phi_A(\tilde A))).
$$
Recall that if $c < 1$ and $c \le c_0$ then $\mathbb P_c$ defines a CLE$_\kappa$ with $\kappa < 4$. We can therefore reformulate Corollay \ref {boundedness} in terms of $c$  as follows (this corresponds intuitively to the statement that SLE is unlikely to be very close to a boundary arc when $\kappa < 4$):
\begin{proposition} \label{p.swbound}
If $c <1$ and $c \le c_0$, then there is an $\alpha(c)  \in (0,1)$ such that, if we denote by $\diam(A)$ the diameter of
$A$, then we have $\E ( \diam( \tilde A)^{1+\alpha}) < \infty$ for all hulls $A$.
\end{proposition}
Throughout the remainder of this subsection, we will suppose that $c_1 < 1$ and $c_1 \le c_0$, and that this $c_1$ is fixed. We then choose $\alpha = \alpha (c_1)$, and we define $M$ and $N$ using this value of $\alpha$.
We will then let $c$ vary in $[0, c_1]$. It follows from the previous proposition that for all $c \le c_1$,
$$N(A) = \E ( M ( \Phi_A (\tilde A))) \le \E ( M ( \tilde A)) \le \E ( \diam(\tilde A)^{1+\alpha}) < \infty .$$
Here is a more elaborate consequence of the previous proposition:
\begin{corollary} \label{rectangleNbound}
Consider $c \leq c_1$ and $\alpha = \alpha(c_1)$ fixed as above.  For any hull $A$ and any $S \in \mathcal S$, if $A_S = A \cap S$, then
$$\E( M ( \Phi_A (\tilde A_S))) \le C(c) M(S),$$ for some constant $C(c)$ depending only on $c$ and
tending to zero as $c \to 0$.
\end{corollary}
\begin{proof}

By scaling, it suffices to consider the case where $S=[0,1] \times [1,2]$ and hence $R=R(S)$ is the rectangle $[0,1] \times [0,2]$.
Proposition \ref{p.swbound} then implies that
$$
\E (M(\Phi_A (\tilde A_S))) \le \E ( M (\tilde A_S )) \le \E ( M ( \tilde R)) \le \E ( \diam( \tilde R)^{1+\alpha} ) < \infty.
$$
We want to prove that
$\E (M(\Phi_A (\tilde A_S)))$ tends to zero uniformly with respect to $A$ as $c \to 0$.
Let $E(c)$ denote the event that some loop-soup cluster (in the loop-soup of intensity $c$) intersecting the rectangle $R$ has radius more than $\epsilon^2$.
When $E(c)$ does not hold, then standard distortion estimates yield an $\epsilon$ bound on the height of (i.e., the largest imaginary part of an element of) $\Phi_A (\tilde A_S)$. But we then also know that $\tilde A_S$ is a subset of $[-1, 3] \times [0,3]$, so that a Brownian motion started from $is$ will hit $\tilde A_S$ before hitting $A \cup \R$ with a probability bounded by $s^{-1}$ times some universal constant $C$. Hence, unless $E(c)$ holds, we have $M ( \Phi_A ( \tilde A_S )) \le C \eps^\alpha$.

Summing up, we get that
$$ \E ( M ( \Phi_A ( \tilde A_S))) \le C \eps^\alpha + \E ( 1_{E(c)} M ( \tilde A_S))
\le C \eps^\alpha + \E ( 1_{E(c)} \diam( \tilde R^{c_1})^{1+\alpha} ), $$
 where $\tilde R^{c_1}$ denotes the $\tilde R$ corresponding to a larger loop-soup of intensity $c_1$ that we couple to the loop-soup of intensity $c$.
But $\P(E(c)) \to 0$ as $c \to 0$ and $\E ( \diam( \tilde R^{c_1})^{1+\alpha} ) < \infty$, so that if we take $c$ sufficiently small,
$$ \E ( M ( \Phi_A ( \tilde A_S))) \le  2 C \eps^\alpha $$
for all hulls $A$. This completes our proof.
\end{proof}

We are now ready to prove our final Lemma:
\begin{lemma} \label{generalNbound}
For $c \le c_1$, there exists a finite constant $C_1 = C_1 (c)$ such that for
all hulls $A$, $N(A) \leq C_1(c) M(A)$. Furthermore, we can take $C_1(c)$ in such a way that
$C_1(c)$ tends to zero as $c \to 0$.
\end{lemma}
\begin{proof}
Putting together the estimates in
Lemma \ref{mhatsum} and Corollary \ref {rectangleNbound}
we have
\begin {eqnarray*}
 N(A)  & = &  \E ( M ( \Phi_A ( \tilde A))) \\
 & =  & \E ( M ( \Phi_A (\underline \cup_{S \in {\cal S}(A)} \tilde A_S )))
\\
& = & \E ( M ( \underline \cup_{S \in {\cal S}(A)} \Phi_A (\tilde A_S)))  \\
& \le &  \E ( \sum_{S \in {\cal S} (A)} M ( \Phi_A (\tilde A_S)))\\
&  \le&   \sum_{S \in {\cal S}(A)}  C(c) M(S) \\
& \le &  C(c) (C')^{-1} M(A)
\end {eqnarray*}
and $C(c) \to 0$ when $c \to 0+$, whereas $C'$ does not depend on $c$.
\end{proof}

As we will now see, this property implies that for $c = c_0$, $N$ is necessarily infinite for all positive $\alpha$ (i.e.\ it shows that the size of clusters at the critical point can not decay too fast), and this will enable us to conclude the proof of Proposition \ref{criticaltheorem} in the manner outlined after its statement:

\begin {proof}
Suppose that $c_0 < 1$. We choose $c_1=c_0$ (and $\alpha=\alpha (c_1)$). We take $c'$ to be positive but small enough
so that the product of the corresponding constants $C_1(c_0)$ and $C_1 (c')$ in Lemma \ref{generalNbound} is less than $1$.
We will view the loop-soup $\Gammabis$ with intensity $c_0 + c'$ as the superposition of a loop-soup $\Gammabis_0$ with intensity $c_0$ and an independent loop-soup
$\Gammabis'$ with intensity $c'$ i.e.\ we will construct $\Gamma$
 via the loop-soup cluster boundaries in $\Gamma_0$ and $\Gamma'$.

Now let us begin with a given hull $A$ (say the semi-disk of radius $1$ around the origin).
Suppose that $\Gammabis$ contains a chain of loops that join $A$ to the line $L_R= \{ z \in \H \ : \  \Im (z) =R  \}$. This implies that one can find a finite chain $\overline \gamma_1, \ldots, \overline \gamma_n$ (chain means that two consecutive loops intersect) of loops
in $\Gamma_0 \cup \Gamma'$ with $\overline \gamma_1 \cap A \not= \emptyset$ and $\overline \gamma_n \cap L_R \not= \emptyset$.
Since the loops in $\Gamma_0$ (resp. $\Gamma'$) are disjoint, it follows that the loops $\overline \gamma_1, \ldots , \overline \gamma_n$
alternatively belong to $\Gamma_0$ and $\Gamma'$.

Consider the loops of $\Gamma_0$ that intersect $A$. Let us consider $A_1$ the hull generated by the union of  $A$ with these loops (this is the $\tilde A$ associated to the loop-soup $\Gammabis_0$). Recall that the expected value of $M( A_1)$ is finite because $\alpha= \alpha( c_1)$.
Then add to $A_1$, the loops of $\Gamma'$ that intersect $A_1$. This generates a hull $B_2$ (which is the $\tilde A_1$ associated to the loop-soup $\Gammabis'$).
Then, add to $B_2$ the loops of $\Gamma_0$ that intersect $B_2$. Note that in fact, one basically adds only the loops of $\Gamma_0$ that intersect
$A_1 \setminus A$ (the other ones were already in $A_1$) in order to define a new hull $B_3$, and continue iteratively.
Let $F$ be the limiting set obtained. We can also describe this sort of exploration
by writing for all $n \ge 1$, $A_{n+1} = \Phi_{A_n}(\tilde A_n)$, where $\tilde A_n$ is alternately constructed from $A_n$ using a loop-soup with intensity $c_0$  or
$c'$ as $n$ is even or odd.
The expected value of $M(A_n)$ decays exponentially, which implies (Borel-Cantelli) that $M(A_n)$ almost surely decays eventually faster than some exponential
sequence.

We note that if $A$ is a hull such that for all $z \in A$, $\Im (z) \le 1$, we clearly have $\hcap (A) \le M(A)$. On the other hand, we know that if $A$ is a hull such that
there exists $z \in A$ with $\Im (z) \ge 1$, then $M(A) \ge c$ for some absolute constant $c$. Hence, we see that almost surely, for all large enough $n$, $\hcap (A_n) \le M(A_n)$, which implies that almost surely
 $\sum_n \hcap(A_n) < \infty$. But the half-plane capacity behaves additively under conformal iterations, so that in fact
$ \hcap (F) = \sum_{n \ge 0} \hcap (A_n)$.
Hence, for large enough $R$, the probability that $F$ does not intersect $L_R$ is positive, and
 there is a positive probability that no chain of loops in $\Gammabis$ joins $A$ to $L_R$.
It follows that $\mathbb P_{c_0 +c'} ( \overline {\mathcal C} =  \{ \overline \HH \} ) < 1$ and
 Proposition~\ref{p.satisfiesconformal} would then imply that $c_0  +c' \le c_0$ which is impossible.
This therefore implies that $c_0 \ge 1$. As explained after the proposition statement, we also know that $c_0$ can not be strictly larger than $1$, so that we can finally
conclude that $c_0= 1$.
\end {proof}

As explained at the beginning of this section, this completes the proof of  Theorems \ref{loopsoupsatisfiesaxioms} and
\ref{kappacorrespondence}.

\section{An open problem}

When $\kappa \in (4,8)$, the CLE$_\kappa$ described in \cite{Sh} are random collections of non-simple loops.  However, the outer boundaries of the sets of points traversed by the outermost loops should be simple loops.  We expect (though this was not established in \cite{Sh}, in part because of the CLE$_\kappa$ construction was not shown there to be starting-point independent) that these random loop collections will satisfy all of the axioms we used to characterize CLE except that they will a.s.\ contain loops intersecting one another and the boundary of the domain.  Are these the only random loop collections for which this is the case?

This is an apparently difficult type of question.  We expect that the constructions of one-point and two-point pinned measures in the present paper will have analogs in this non-disjoint setting; however the pinned loop will intersect the boundary at other points as well, and our argument for establishing the connection with SLE does not appear to extend to this setting in a straightforward way.  Moreover, we have no direct analog of the loop-soup construction of CLEs in the case where $\kappa \in (4,8)$.

\medbreak

\noindent\begin{tabular}{lll}
{\it Department of Mathematics MIT, 2-180} & & {\it Laboratoire de Math\'ematiques}\\
{\it 77 Massachusetts Avenue} & & {\it Universit\'e Paris-Sud 11, b\^{a}timent 425}\\
{\it Cambridge, MA 02139-4307} & & {\it 91405 Orsay Cedex, France} \\
sheffield@math.mit.edu & & wendelin.werner@math.u-psud.fr
         \end{tabular}

\end {document}